\documentclass[10pt]{article}
\usepackage{arxiv}
\usepackage{tabu}       % professional-quality tables

\usepackage{nicefrac}       % compact symbols for 1/2, etc.
\usepackage{microtype}      % microtypography
\usepackage{lipsum}
\usepackage[utf8]{inputenc} % allow utf-8 input
\usepackage[T1]{fontenc}    % use 8-bit T1 fonts

\usepackage{graphicx}  
\usepackage{amsmath}  
\usepackage{hyperref}
\usepackage{multirow}
\usepackage{url}
\usepackage[all]{xy}
\usepackage{microtype}
\usepackage{graphicx}
\usepackage{subcaption}
\usepackage{ upgreek }
\usepackage{booktabs} % for professional tables
\usepackage{ dsfont }
\usepackage{amsfonts}       % blackboard math symbols
\usepackage{ amssymb }

\usepackage{amsthm}
\usepackage{algorithm}
\usepackage{algorithmic}
\usepackage{enumerate}
\usepackage{forloop}
\usepackage{comment}
\usepackage[numbers]{natbib}
\usepackage{float}
\usepackage{tabularx}
\usepackage[nottoc]{tocbibind}
\usepackage{ latexsym }
\usepackage{pst-node}
\usepackage{amssymb,amsthm}
\usepackage{tikz-cd}
\usepackage{mathabx}
\usepackage{esvect}
\usepackage{tikz}
\usepackage{paralist}
\usetikzlibrary{arrows}
\usepackage{nomencl} 

\usepackage{ bbold }
\makenomenclature

\newcommand{\overbar}{\overline}

\newcommand{\R}{\mathbb{R}}
\newcommand{\N}{\mathbb{N}}
\newcommand{\field}{\mathds{k}}
\newcommand\SComplex{K}

\newcommand{\Filt}{\mathrm{Filt}}

\newcommand\End{\mathrm{End} (\I, \leq)}
\newcommand\Aut{\mathrm{Aut}(\I, \leq)}
\newcommand\EndR{\mathrm{End} (\R, \leq)}
\newcommand\AutR{\mathrm{Aut}(\R, \leq)}

\newcommand{\IncHomSpace}{\Aut}
\newcommand{\NonDecMaps}{\End}
\newcommand{\IncHomSpaceR}{\AutR}
\newcommand{\NonDecMapsR}{\EndR}

\newcommand{\maxvalue}{\infty}

\newcommand\Id{\mathrm{Id}}

\newcommand\simplex{\sigma}

\newcommand\Polytope{\mathcal{P}}

\newcommand\SymK{\mathbb{G}(\SComplex)}
 
\newcommand\Barc{\mathrm{Bar}}
\newcommand\persmap{\mathrm{PH}}

\newcommand\StratumD{\mathcal{B}}
\newcommand\StratumF{\mathcal{S}}

\newcommand\Holink{\mathrm{Holink}}
\newcommand\ChartBarc{\nu}
\newcommand\ChartFilt{\mu}
\newcommand\Cone{\mathrm{C}}
\newcommand\pt{\mathrm{*}}

\newcommand\Dzeroleft{\StratumD_0^1} 
\newcommand\Dzeromid{\StratumD_0^2}
\newcommand\Dzeroright{\StratumD_0^3}

\newcommand\Donemid{\StratumD_1^2}

\newcommand\Dtwo{\StratumD_2^1}
\newcommand\Dthreeleft{\StratumD_3^1}
\newcommand\Dthreemid{\StratumD_3^2}

\newcommand\Dfourleft{\StratumD_4^1}
\newcommand\Dfourright{\StratumD_4^2}
\newcommand\Dfive{\StratumD_5}

\newcommand\Btwo{D_2}
\newcommand\Bthreeleft{D_3^1}
\newcommand\Bthreemid{D_3^2}

\newcommand\Bfourleft{D_4^1}

\newcommand\rb{\partial}

\newcommand\StandSimplex{\Delta}

\newcommand\PolyComplex{\Pi}
\newcommand\Monodromy{\mathcal{L}}
\newcommand\Path{\mathbf{Path}}
\newcommand\TopSpace{{X}}
\newcommand\TopSpaceBis{{Y}}

\newcommand\PolyMap{\mathbf{\Phi}}

\newcommand\Ent{\mathbf{Ent}}
\newcommand{\Vset}{\mathrm{V}}
\newcommand\Low{\mathrm{Low}}
\newcommand\FirstSimplexMobius{S}
\newcommand\SecondSimplexMobius{\bar{S}}

\newcommand\CatBarc{\Barc_\SComplex}
\newcommand\CategoryBarc{\mathbf{Bar_\SComplex}}
\newcommand\IncMaps{\CategoryBarc}

\newcommand\SCompCat{\mathbf{SCpx}}
\newcommand\PersCat{\mathbf{Pers}}
\newcommand\Vect{\mathbf{Vect}_{\field}}

\newcommand\I{\text{I}}

\newcommand\homleft{\alpha}
\newcommand\homright{\beta}
\newcommand\phinail{\Phi}

\newcommand\PolytopeOpen{\persmap_{|\StratumF}^{-1}(D)}
\newcommand\PolytopeClosed{\overbar{\persmap_{|\StratumF}^{-1}}(D)}
\newcommand\phipre{\phi_{\mathrm{pre}}}
\newcommand\phipost{\phi_{\mathrm{post}}}

\providecommand{\keywords}[1]
{
  \small	
  \textbf{Keywords:} #1
}

\makeatletter
\newcommand{\rmathbb}[1]{{\mathpalette\dude@rmathbb{#1}}}
\newcommand{\dude@rmathbb}[2]{%
  \scalebox{-1}[1]{$\m@th#1\mathbb{#2}$}%
  }
\makeatother

\theoremstyle{plain}
\newtheorem{theorem}{Theorem}[section]
\newtheorem{lemma}[theorem]{Lemma}
\newtheorem{proposition}[theorem]{Proposition}
\newtheorem{corollary}[theorem]{Corollary}

\newtheorem*{theorem*}{Theorem}
\newtheorem*{proposition*}{Proposition}

\theoremstyle{definition}
\newtheorem{remark}[theorem]{Remark}

\newtheorem{definition}[theorem]{Definition}
\newtheorem{example}[theorem]{Example}

\definecolor{dred}{rgb}{0.7,0.0,0.0}

\newcommand\UT[1]{\textcolor{blue}{#1}}

\begin{document}

	\title{The Fiber of Persistent Homology for simplicial complexes}

	\author{
  Jacob Leygonie%\thanks{Corresponding author.}
  \\
  Mathematical Institute\\
  University of Oxford\\
  Oxford OX2 6GG, UK \\
  \texttt{jacob.leygonie@maths.ox.ac.uk} \\
  \And
  Ulrike Tillmann \\
   Mathematical Institute \\
University of Oxford\\
  Oxford OX2 6GG, UK \\
   \texttt{tillmann@maths.ox.ac.uk} \\
}
	
\maketitle

\begin{abstract}
We study the inverse problem for persistent homology: For a fixed simplicial complex $\SComplex$, we analyse the fiber of the  continuous map 
$\persmap $ on the space of filters  %$\Filt_\SComplex $ 
that assigns to a filter $f: \SComplex \to \mathbb R$ the total barcode of its associated sublevel set filtration of $\SComplex$. We find that $\persmap$ is best understood as a map of stratified spaces. Over each stratum of the barcode space %$\Barc _\SComplex:= \persmap (\Filt_\SComplex)$, 
the map~$\persmap$ restricts to a (trivial) fiber bundle with fiber a polyhedral complex. Amongst other we derive a bound for the dimension of the fiber depending on the number of distinct endpoints in the barcode. Furthermore, taking the inverse image $\persmap^{-1}$ can be extended to a monodromy functor on the (entrance path) category of barcodes. We demonstrate our theory on the example of the simplicial triangle giving a complete description of all fibers and monodromy maps. This example is rich enough to have a M\"obius band as one of its fibers.
\end{abstract}


\setcounter{tocdepth}{2}
\tableofcontents

\newpage

\section*{Introduction}
\label{section_introduction}
Topological Data Analysis (TDA) is a rapidly expanding, new area  \cite{carlsson2009theory,edelsbrunner2010computational, oudot2015persistence} which has been applied to a large variety of data science problems. Its best-known tool, persistent homology, provides a non-linear dimension reduction method which is computable \cite{otter2017roadmap,zomorodian2005computing} and robust with respect to small perturbations of the underlying data \cite{cohen2007stability}. A growing number of vectorisation methods \cite{adams2017persistence,bubenik2015statistical} enable statistical studies of the outcome of persistent homology and combining it with machine learning methods. 

It is thus natural to ask how much information can be recovered from  persistent homology: Given a particular  instance of a persistence module, what can we say about the data set it has been derived from?  Any qualitative or quantitative understanding of the  information loss would be of great value for future applications, and several approaches to variations of this question have recently appeared \cite{curry2018fiber,curry2018many, cyranka2018contractibility,  turner2014persistent}. We refer to \cite{oudot2018inverse} for a survey of inverse problems for persistent homology. 

In this paper we analyse this foundational problem in a  general form:  For a fixed simplicial complex~$\SComplex$, we study the persistent homology map~$\persmap$ (and its inverse) from the real-valued functions on~$\SComplex$ to the space of barcodes. 
We are naturally led to study~$\persmap$ as a map of stratified spaces and to extend~$\persmap ^{-1}$ to a functor defined on a natural category of barcodes. This rich structure we expect may  also be  of interest outside the data science community.
Indeed, persistent homology has attracted much recent interest from other branches of mathematics, in particular symplectic topology, stemming from its connection to Morse theory and its close cousin Floer homology~\cite{Barannikov94, PolterovichShelukhin16, UsherZhang16}, and may yet find uses in other areas, see for example~\cite{ManinMarcolli20}.
%is justified by the interest in persistent homology also from pure mathematics stemming from its close connection to Morse theory and its close cousin, Floer homology; see
%\cite{Barannikov94, PolterovichShelukhin16, UsherZhang16, ManinMarcolli20}.

%We illustrate our results on the example of a triangle which is surprisingly rich. 
\subsection*{Content and results}
Given a fixed finite simplical complex~$\SComplex$, let~$\Filt_\SComplex$ be the space of its filters. By definition these are functions~$f: \SComplex \to \I =[0,1] \subset \R$ 
that are monotonic with respect to face inclusions, 
\[
\simplex\subseteq \simplex'  \Rightarrow f(\simplex)\leqslant f(\simplex')   \quad \quad \text{ for all } \simplex, \simplex' \in \SComplex.
\]
Thus each sublevel set~$f^{-1} ((-\infty, t])$ defines a simplicial subcomplex of~$\SComplex$ and  every~$f$ gives rise to a filtration.  Persistent homology then defines a continuous map
\[
\persmap : \Filt_\SComplex \longrightarrow \Barc^\infty,
\]
where~$\Barc^\infty$ is the space of total barcodes and~$\persmap$ assigns the union of barcodes in all homological degrees. As~$\SComplex$ is fixed we restrict our attention to the image
\[
\CatBarc := \persmap (\Filt_\SComplex).
\]
In this notation, to understand the information loss of persistent homology is to understand the fiber~$\persmap^{-1} (D)$ at a barcode~$D \in \CatBarc$. This naturally leads us to a closer analysis of the spaces~$\Filt_\SComplex$ and~$\CatBarc$ themselves. We will endow them with monoid actions and stratifications, and show that~$\persmap$ is compatible with these extra structures.

Let~$\End$ be the monoid of order preserving continuous maps of the unit interval~$\I$ that fix the endpoints, and let~$\Aut$ be its subgroup of homeomorphisms. 
The monoid~$\End$ and hence~$\Aut$ act continuously on~$\Filt_\SComplex$ by post-composition and on~$\Barc_\SComplex$ by moving the endpoints of the bars. As the endpoints of the bars in~$\persmap(f)$ are a subset of the values of~$f$, the map~$\persmap $ is readily seen to be equivariant with respect to these actions (Lemma~\ref{lemma_commutes_PH_homeo}).

For our further analysis it is important that both~$\Filt_\SComplex$ and~$\CatBarc$ have a natural stratification where each $i$-dimensional stratum is identified with an open simplex 
\[
\mathring{\StandSimplex}^i= \big \{ (x_1,\cdots,x_i) \, |   \, \, 0< x_1 < \dots < x_i <1  \big \}
\]
such that the coordinates are given by  the distinct values in the image of~$f$, and respectively, the distinct endpoints of the bars in~$D$. We identify each such stratum as an $\Aut$-orbit, and thus~$\persmap$, by its equivariance, is a strongly stratified map taking a stratum of filters surjectively onto a stratum of barcodes (Proposition~\ref{proposition_PH_sends_strata_to_strata}). In particular,  the inverse image~$\persmap^{-1} (\StratumD)$ of any barcode stratum~$\StratumD$ is a finite union of filter strata. We then show that~$\persmap$ over a stratum~$\StratumD$ is a fiber bundle with fiber a polyhedral complex (Theorem~\ref{theorem_fiber_bundle_polyhedral}), and derive some general properties of this fiber. Thus, for example, we show that the dimension of the fiber over a barcode~$D$ is bounded by half the difference between the number~$\sharp \SComplex$  of simplices and the number~$\sharp  D$ of endpoints  in the barcode (Proposition~\ref{proposition_improved_bounding_dimension_fiber}):
\[\dim \persmap^{-1}(D) \leqslant \frac{\sharp \SComplex-\sharp D}{2}.\]
Unlike~$\Filt_\SComplex$,~$\CatBarc$ is not (the realisation of) a simplical complex: On the boundary of a barcode stratum, viewed as an open simplex, the $1$-dimensional subspaces corresponding to bars~$(x_j, x_j)$ of length zero are collapsed. %if~$x_j$ is not also the endpoint of another bar with positive length (Example~\ref{example_line_complex}).  
Nevertheless, we are able to describe the attaching (or monodromy) maps of the fiber over~$\StratumD$ to the fiber of a lower dimensional stratum~$\StratumD'\subseteq \bar{\StratumD}$ in
its closure. These attaching maps are homotopic to maps of polyhedra but are not in general homotopic to each other (Proposition~\ref{proposition_monodromy}). We find that this structure is most naturally described in terms of the category~$\CategoryBarc$: Its objects are the barcodes in~$\CatBarc$ and its space of morphisms from~$D$ to~$D'$ is the subspace of maps~$\phi \in \End $ that send the endpoints of the bars in~$D$ surjectively to those of~$D'$. Each of these morphism spaces is discrete up to homotopy (Theorem~\ref{theorem_barcode_category}). Taking the inverse image then extends to a functor from~$\CategoryBarc$ to the category of topological spaces and continuous maps
\[
\persmap^{-1} : \CategoryBarc \longrightarrow \mathbf {Top}
\]
 taking a morphism~$\phi \in \CategoryBarc (D, D')$ to the continuous map~$\Monodromy_\phi: f \in \persmap ^{-1} (D) \longmapsto \phi \circ f \in \persmap ^{-1} (D')$, which is indeed well-defined by the equivariance of~$\persmap$ under the action of~$\End$.

The category~$\CategoryBarc$, which we were naturally led to consider, is closely related to the entrance path category~$\Ent(\CatBarc)$ of the space~$\CatBarc$ which we prove to be homotopically stratified (Proposition~\ref{proposition_barcodes_homotopically_stratified}) in the sense of Quinn
\cite{quinn1988homotopically}.
%(with locally path-connected and locally simply-connected strata). 
Indeed, we show that descending to the homotopy category, i.e. replacing morphism spaces with the set of their connected components, induces an isomorphism of categories (Proposition~\ref{proposition_entrance_path_category_equal_catbarc})
\[
h\CategoryBarc \simeq \Ent(\CatBarc).
\]
Recall %from \cite{woolf2008fundamental} 
that the entrance path category is the analogue for stratified spaces of the fundamental groupoid, and that functors from the entrance category to the category of sets are in correspondence with branched covers. Taking this analogy one step further by replacing the category of sets with the homotopy category of spaces~$h\mathbf{Top}$, we may most naturally think of~$\persmap$ as a stratified fiber bundle with polyhedral fibers:
\[
\persmap^{-1} : \Ent(\Barc  _\SComplex) \longrightarrow h\mathbf {Top}.
\]
%, and the induced composite functor
%$$
%\Ent(\CatBarc) \overset {\persmap ^{-1} }\longrightarrow h\mathbf {Top} \overset {\pi_0} \longrightarrow 
%\mathbf {Sets}
% $$
%as a branched covering; compare \cite{woolf2008fundamental}. Here $\pi_0$ is the functor that takes a topological space  to its set of path-connected components.

%Finally we remark in section~\ref{section_unbounded_situation} that there is no loss of generality by restricting to filters that take value in the bounded interval~$\I \subset \mathbb R$. We also note in section~\ref{section_lower_stars} that a similar analysis as above holds if instead of all filter functions we consider only lower star filters. Indeed, the space~$\Low_\SComplex$ of lower star filters is a union of strata in~$\Filt_\SComplex$. Thus the fiber of~$\persmap$ restricted to the lower star filters is again a polyhedral complex.  

Finally in section 5 we consider variants of our fiber problem and interactions with the symmetries of the underlying simplicial complex. Thus, in section 5.1 we consider the case where filters and barcodes are allowed to take values in the real line~$\R$ instead of~$\I$. 
The results on the fiber~$\persmap^{-1}(D)$ adapt to this situation, with the only difference that the polyhedra in the fiber may now be unbounded. However, we illustrate with examples that the topology of~$\CatBarc$ is more complicated in the unbounded situation: we show that the bottleneck topology on barcodes does not in general agree with the easy to understand quotient topology, unlike in the bounded case (Proposition
~\ref{prop_quotient_topology}). This is in general no longer true in the unbounded situation. In addition, in section~\ref{section_lower_stars} we show that the overall analysis holds if instead of all filters we consider only lower star filters. Indeed, the space~$\Low_\SComplex$ of lower star filters is a union of strata in~$\Filt_\SComplex$. Thus the fiber of~$\persmap$ restricted to the lower star filters is again a polyhedral complex.  

In Appendix~\ref{section_triangle}, illustrating our theory, we describe in complete detail the case when~$\SComplex$ is a triangle (with~$6$ simplices). For each of the~$34$ barcode strata in~$\CatBarc$ we describe the fiber with the action of the symmetry group of~$\SComplex$ and their monodromy maps. While most fibers consist of a set of discrete points, three fibers are homeomorphic to a circle, one is homeomorphic to two copies of the circle, and one is homeomorphic to the M\"obius band.
As far as we are aware this is the first non-contractible simplicial complex for which the fibers of~$\persmap$ have been studied and also the first example where the fibers are not homotopy discrete.

\subsection*{Related work}
Our set-up here is most closely related to that in  the work of Cyranka, Mischaikow and Weibel~\cite{cyranka2018contractibility} where the authors consider lower star filters on the $n$-fold subdivided interval and show that the fibers of~$\persmap$ are homotopy discrete. 

Previously, Curry in \cite{curry2018fiber} considers the interval with the set of continuous maps. In particular he  bounds the connected components of the fiber in terms of the nestings of the intervals in the barcode. Curry with coauthors also
%Catanzaro, Fasy, Lazovski, Malen, Riess,  Wang  and Zabka 
studies the higher dimensional example of a sphere in \cite{catanzaro2020moduli} with the set of functions that arise as compositions of an embedding of~$\mathbb S^2$ into~$\R^3$ followed by a projection onto the last coordinate. 
%Here we add to this list the triangle. 

In general, the persistent homology associated to a single filter cannot determine the underlying simplicial complex or its homotopy type, no more than homology can determine the homotopy type of the underlying space. However, under some  conditions a family of such functions might suffice. To understand this question Turner, Mukherjee and Boyer introduced the persistent homology transform (PHT) \cite{turner2014persistent}, and proved that indeed under certain circumstances PHT is injective on shapes embedded in~$\R^3$, see also~\cite{curry2016classification,ghrist2018persistent} for a generalisation to higher dimensions. It has even been possible to find algorithmically a left inverse for PHT for some specific classes of sets~\cite{belton2020reconstructing, betthauser2018topological,fasy2019persistence, micka2020searching}. 

There are other persistence based invariants of spaces. One such  (stable and computable) invariant  for metric graphs has been proposed by Dey, Shi, and Wang  \cite{bey2015}. In~\cite{oudot2017barcode} Oudot and Solomon show that the fiber of this intrinsic  transform is generically globally and always locally injective. 

Since this paper was submitted, several other properties of the fiber~$\persmap ^{-1}(D)$ have been studied.
Motivated by our analysis here, in~\cite{LeygonieHenselman} the first author with Gregory Henselman-Petrusek provides an algorithm and software for the computation of the fiber~$\persmap ^{-1}(D)$. This allows for many more examples to be computed explicitly. Furthermore, with David Beers he analyzes in~\cite{leygonie2021fiber} the fiber of~$\persmap$ for Morse functions on any smooth compact manifold with boundary~$\mathcal{M}$ and shows that each path connected component in the fiber equals the orbit of one of its Morse functions under the action of isotopies of~$\mathcal{M}$. This enables in particular the computation of the homotopy type of each path connected component in the fiber for almost all surfaces. %To the best of our knowledge, these are all the examples where the fibers of  persistent homology have been analysed.
Meanwhile,  the second and third author of \cite{cyranka2018contractibility} have extended their analysis in ~\cite{mischaikow2021persistent}  to the $n$-fold subdivided circle where the fibers are shown to be homotopy equivalent to~$\mathbb S^1$.

\subsection*{Acknowledgements}
We are indebted to the reviewer of the Journal of Pure and Applied Algebra, where this manuscript was accepted, for his many valuable insights. We also wish to thank Heather Harrington for her interest in this project, and acknowledge the support of the Centre for Topological Data Analysis, EPSRC grant EP/R018472/1. 

\newpage
%%%%%%%%%%%%%%%%%%%%%%%%%%%%%%%%%%%%%%%%%%%%%%%%%%%%%
%%%%%%%%%%%%%%%%%%%%%%%%%%%%%%%%%%%%%%%%%%%%%%%%%%%%%

\section{Stratifications of the spaces of filters and barcodes}
We first recall some background theory 
%in section~\ref{section_background} 
and define the persistence map~$\persmap$, the fibers of which are the object of interest. In section~\ref{section_equivariance}, we introduce the topological monoid~$\NonDecMaps$ of non-decreasing maps of the interval, which provides essential structure: it acts continuously on the spaces of filters and barcodes, and  the persistence map~$\persmap$ is equivariant with respect to this action. In section~\ref{sec:stratifications_barcodes_filters}, we show furthermore that the orbits of the subgroup of homeomorphisms,~$\IncHomSpace$,  provide stratifications for the space of filters and the space of barcodes and that, due to its equivariance,~$\persmap$ is a strongly stratified map between them. It follows now easily that the fibers over barcodes from the same  stratum are pairwise homeomorphic, thus turning the identification of the fiber into a finite problem. In section~1.4 we also show  that the image~$\Barc_\SComplex$ of the space of filters under~$\persmap$ has the quotient topology.
\subsection{The definition of the persistence map}
\label{section_background}
Let~$\SCompCat$ denote the category of finite (abstract) simplicial complexes and inclusions, and let~$\SComplex \in\SCompCat{}$ be an arbitrary, non-empty  simplicial complex of dimension~$d\in \mathbb{N}$, which is fixed throughout the paper. We consider~$\SComplex$ as a subset of the power set on its vertices. Recall, if~$\sigma \in \SComplex$ then all its non-empty subsets~$\sigma' \subset \sigma$, i.e. its {\em faces}, are also in~$\SComplex$. We write~$\sharp \SComplex$ for the total number of simplices in the complex~$\SComplex$.

We denote by~$\I:=[0,1]$ the closed unit interval. A typical function on~$\SComplex$ valued in~$\I$ is denoted by~$f\in\I^\SComplex$.

\begin{definition}
A {filter function}, or {\em filter} for short, on~$\SComplex$ is a map~$f: \SComplex \to \I$ that is monotonic with respect to face inclusions: For all
simplices~$\sigma', \sigma \in K$
\[
\simplex'\subseteq \simplex  \Rightarrow f(\simplex')\leqslant f(\simplex).
\]
The set of all filters on~$\SComplex$ is denoted by~$\Filt_\SComplex$. 
%We are interested in the subspace of {\em filters}, $\Filt_\SComplex\subseteq \I^\SComplex$, of functions~$f$ valued in~$\I$ that are monotonic w.r.t. face inclusions: 
%
%\[
%\simplex\subseteq \simplex' \in \SComplex \Rightarrow f(\simplex)\leqslant f(\simplex').
%\]
%
\end{definition}
The monotonicity condition on filters is equivalent to the property that their sublevel sets are simplicial subcomplexes of~$\SComplex$. 
Thus a filter~$f$ gives rise to a filtration $ \SComplex(f) = \{f^{-1}((-\infty, t]) \}_{t \in \R}$, of~$\SComplex$ which we may think of  as a functor from~$\R$ (as an ordered set) to the category of simplicial complexes
\[
\SComplex(f) : (\R, \leq) \longrightarrow \SCompCat.
\]
We can then compose this with the functor~$\mathrm{H}_p$ which takes a simplicial complex to its $p$-th simplicial homology with coefficients in a fixed field~$\field$.
%here $0 \leqslant p\leqslant d$. 
This defines the {\em $p$th persistent homology functor}
\[
\mathrm{H}_p( K(f) ) : (\R, \leq ) \longrightarrow
\Vect,
\]
which is an instance of a {one-parameter, pointwise finite dimensional, finite persistence module}, or {\em persistence module} for short. 
We denote by~$\PersCat$ the category of such persistence modules and natural transformations between them.
%
%Let $\overbar{\I}:=\I \sqcup \{\infty\}$, and $\HalfSpace:=\{(b,d)\in \I\times \overbar{\I},  b<d \}\subseteq \I\times \overbar{\I} $ be the open, extended, left upper-half square above the diagonal. 

Given an interval~$J\subseteq \R$, the associated {\em interval module}~$\mathbb{I}_{J}\in \PersCat$ has copies of the field~$\field$ over~$J$ and zero  elsewhere, the copies of~$\field$ being connected by identity maps.
Given a persistence module~$\mathbb{V}\in \PersCat$, by the Decomposition Theorem~\cite{crawley2015decomposition}, there exists a unique finite multiset~$\mathcal{J}$ of intervals such that we have an isomorphism:
\[
\mathbb{V}\cong \bigoplus_{J\in \mathcal{J}} \mathbb{I}_{J}.
\]
The finite multiset~$\Barc(\mathbb{V})$ of pairs~$(\inf J, \sup J)\in (\R\sqcup\{-\infty\})\times (\R\sqcup\{\infty\})$ for intervals~$J\in \mathcal{J}$ appearing in the above decomposition is the so-called {\em barcode} of the module~$\mathbb{V}$.  If~$\mathbb{V}=\mathrm{H}_p(K(f))$ is a persistent homology module, then the intervals that occur are all half-open intervals of the form~$J =[b, d)$ with restricted values %~$\Barc(\mathbb{V})$ consists of pairs 
$(b,d)\in \I\times (\I \sqcup\{\infty\})$. Consequently, we formally define barcodes as follows. 

\begin{definition}
\label{definition_barcode}
A barcode~$D$ is a finite multi-set of pairs~$(b,d)$ in~$\I\times (\I\sqcup\{\infty\})$, with~$b<d$, called the {\em intervals} or {\em bars} of~$D$. An interval~$(b,d)$ is {\em bounded} (resp. {\em infinite}) if~$d<\infty$  (resp.~$d=\infty$). The multiplicity of an interval~$(b,d)\in D$ is denoted by~$D(b,d)\in \N$. The set of all barcodes is denoted by~$\Barc$. 
\end{definition}
We can now define the {\em degree $p$ persistence map} as the composition
\[
\persmap_p := \Barc( \mathrm{H}_p  \circ K (.)): \Filt_\SComplex \longrightarrow \Barc
\]
and the (total) {\em persistence map} as the product
\[
\persmap := (\persmap_0, \dots , \persmap_d): \Filt_\SComplex \longrightarrow \Barc ^{d+1}.
\]
We will refer to elements
$D = (D_0 , \dots, D_d) \in \Barc^{d+1}$ simply as barcodes. We will mainly be interested  in barcodes in the image of~$\persmap$:
\[\CatBarc:=\persmap(\Filt_\SComplex)\subseteq\Barc^{d+1}.\]
The set of filters~$\Filt_\SComplex $ is naturally topologised as a subset of the  finite dimensional Euclidean space~$\R ^K$. 

The standard topology on~$\Barc$, and hence on~$\CatBarc$, is induced by an (extended) metric which we now recall.
A {\em matching}~$\gamma$ between two barcodes~$D,D'\in \Barc$ is a partial injective map~$\gamma$ from intervals of~$D$ to those of~$D'$. The {\em cost}~$c(\gamma)$ of a matching is the maximum of the  following three quantities: (i) the maximum $\|(b,d)-\gamma(b,d)\|_\infty$ over intervals~$(b,d)\in D$ where~$\gamma$ is defined, (ii) the maximal length~$\frac{d-b}{2}$ over intervals~$(b,d)\in D$ where~$\gamma$ is not defined, and (iii) the maximal length~$\frac{d'-b'}{2}$ over intervals~$(b',d')\in D'$ that are not in the image of~$\gamma$. Here we allow~$\infty$ as a possible value for~$d$,~$d'$, and hence also for the  maxima; and~$\| . \| _\infty$ denotes the (extended) supremum norm on~$\R \times (\R \cup \{ \infty\})$.
The \textit{bottleneck distance},~$d_b$, between~$D$ and~$D'$ is then defined as:  
\[d_b(D,D'):= \inf_{\gamma \text{ matching}} c(\gamma). \]
Since our barcodes are finite multisets,~$d_b$ defines a true (extended) metric on~$\Barc$. We endow~$\Barc$ with the induced \textit{bottleneck topology}. 

We thus have the following instance of the Stability Theorem~\cite{bl-indaspb-15, chazal2016structure, cohen2007stability} in our context.

\begin{theorem}
\label{theorem_stability}
The {degree $p$ persistence map}~$\persmap_p$ is Lipschitz continuous, and thus so is the {persistence map}~$\persmap$.
\end{theorem}
For later reference, we record the following elementary fact.
\begin{proposition}
\label{proposition_bottleneck_ball}
Let~$D\in \Barc^{d+1}$ be a non-empty barcode. Then for small enough~$\epsilon$, another barcode~$D'$ is $\epsilon$-close to~$D$ if and only if its intervals satisfy the following:
\begin{itemize}
\item for each integer~$0\leqslant p \leqslant d$ and interval~$(b,d)\in D_p$, the intervals~$(b',d')$ in~$D'_p$ satisfying $\|(b',d')-(b,d)\|_{\infty}< \epsilon$ have multiplicities summing up to~$D_p(b,d)$, the multiplicity of~$(b,d)$ in~$D_p$;
\item the other intervals~$(b',d')\in D'$, that is those that are not~$\epsilon$-close to intervals in$~D$, are $\epsilon$-{\em small}, i.e.~$|d'-b'|<\epsilon$.
\end{itemize}
\end{proposition}
\begin{proof}
Take~$\epsilon \leq\frac{\alpha}{2}$ where~$\alpha$ is the minimum of  (a) the lengths~$d-b$ of intervals~$(b,d)\in D$ and  (b) all pairwise  distances  $\|(b,d)-(\overbar{b},\overbar{d})\|_{\infty}$ for any two geometrically distinct intervals~$(b,d)$ and~$(\overbar{b},\overbar{d})$ in~$D$. Note that~$\alpha >0$ as  by definition all our barcodes are finite, i.e.~$D$ has finite support.
\end{proof}
%

%%%%%%%%%%%%%%%%%%%%%%%%%%%%%%%%%%%%%%%%%%%%%%%%%%%%%
%%%%%%%%%%%%%%%%%%%%%%%%%%%%%%%%%%%%%%%%%%%%%%%%%%%%%

\subsection{Actions on filters and barcodes, and 
equivariance of the persistence map}
\label{section_equivariance}

Let~$\IncHomSpace$ be the space of orientation preserving homeomorphisms of~$\I$, and~$\NonDecMaps$ be the space of continuous non-decreasing maps that fix the boundary points~$0$ and~$1$. We consider them as subspaces of the space of all continuous maps of~$\I$ to itself with the compact open (or equivalently~$||.||_\infty$-metric) topology. In this topology~$\NonDecMaps$ is the closure of~$\IncHomSpace$. For future reference we note that the straight line interpolation 
\[ \phi_t:=(1-t)\phi+t\phi'\]
between maps~$\phi,\phi'\in \NonDecMaps$ defines a continuous path in~$\NonDecMaps$.

Since the boundary points are fixed by elements in~$\IncHomSpace$ and~$\NonDecMaps$, they extend by the identity to automorphisms and endomorphisms  of the  real line~$\R$ and the extended real line~$\R \cup \{ \pm \infty \}$. When the context requires it, we will tacitly extend our maps without changing notation. 
%\footnote{It is clear that the closure of~$\IncHomSpace$ is contained in the set of continuous non-decreasing maps~$\phi$ fixing the endpoints. The converse inclusion follows by taking straight line homotopies from such~$\phi$ to the identity map of the unit interval. }

The monoid~$\NonDecMaps$ acts from the left on~$\Filt_\SComplex$ by post-composition: 
\[
\phi .f  := \phi \circ f.
\]
It also acts from the left on~$\Barc$, and hence diagonally on~$\Barc^{d+1}$, by applying~$\phi$ to all the endpoints of the bars in~$D$ with the convention that~$\phi(\infty) = \infty$ and bars of length zero are suppressed:
\begin{equation}
    \label{eq_definition_action_barcodes}
    \phi . D := \{ (\phi(b), \phi(d)) \, | \, (b,d) \in D  \text { and } \phi(b) \neq \phi (d)\}.
\end{equation}
Thus~$\phi.D$ contains (with multiplicities) an interval~$(\phi(b),\phi(d))$ for each interval~$(b,d)\in D$ as long as~$\phi(b)\neq \phi(d)$, and an interval~$(\phi(b),\infty)$ for each interval~$(b,\infty)\in D$. Note that~$\phi$ being non-decreasing does not imply that~$\phi(x) \geq x$ for all~$x$. In particular, through the action of~$\phi$, intervals can move to the left, to the right, be contracted or expanded.
%The action extends diagonally to an action on~$\Barc^{d+1}$.  

A key result used in this work is that the persistence map is  equivariant with respect to  these actions. 
\begin{lemma}[Equivariance]
\label{lemma_commutes_PH_homeo}
The persistence map~$\persmap$ is $\NonDecMaps$-equivariant: For all~$\phi \in \NonDecMaps $ and~$f\in \Filt_\SComplex$
\[\persmap(\phi\circ f)= \phi.\persmap (f).\]
\end{lemma}
\begin{proof}

We fix a filter~$f\in \Filt_\SComplex$ and a map~$\phi\in \NonDecMaps$.

Recall that~$\persmap (f)$ is the union of~$\persmap_p(f)$ for~$p = 0, \dots , d$,  and~$\persmap_p(f)$ is given by the composition~$\Barc( \mathrm{H}_p \circ K(f))$.  
By definition of~$K(\phi \circ f)$, for~$t \in \mathbb R$ we have
\[K(\phi\circ f)(t) =(\phi \circ f)^{-1}\big((-\infty,t]\big)=f^{-1}\big((-\infty, \max(\phi^{-1}(\{t\}))]\big)=K(f)\big(\max({\phi}^{-1}(\{t\}))\big).\]
%&= f^{-1}(\phi ^{-1} (-\infty; t]) \\
%&=f^{-1}((-\infty; \text{max}(\phi^{-1}(t)]) \\
%&=K(f)(\text{max}({\phi}^{-1}(t))); 
%
%\begin{align*}
%K(\phi\circ f)(t) &=(\phi \circ f)^{-1}((-\infty;t])) \\
%&= f^{-1}(\phi ^{-1} (-\infty; t]) \\
%&=f^{-1}((-\infty; \text{max}(\phi^{-1}(t)]) \\
%&=K(f)(\text{max}({\phi}^{-1}(t))); 
%\end{align*}
%
Note that~$\phi$ is non-decreasing and continuous. Thus the inverse image~$\phi^{-1}(\{t\})$ of the point~$t$ is a closed, bounded interval and hence contains its maximum. On composition with the singular homology functor~$\mathrm{H}_p$ this yields
\begin{equation*}
\label{eq_homology_functor_equiv}
\mathrm{H}_p\circ K(\phi\circ f) (t) = \mathrm{H}_p\circ K(f) \big(\max (\phi^{-1}(\{t\}))\big).
\end{equation*}
Hence the barcode~$\persmap_p(\phi\circ f)$ is the barcode of the persistence module~$t\mapsto \mathrm{H}_p\circ K(f) \big(\max (\phi^{-1}(\{t\}))\big)$, which rewrites uniquely as a sum of interval modules:
\begin{align*}
     \mathrm{H}_p\circ K(f) \big(\max (\phi^{-1}(\{t\}))\big) &\simeq \big[ \bigoplus_{(b,d)\in \persmap_p(f)}\mathbb I_{[b,d)} \big] \big(\max (\phi^{-1}(\{t\}))\big) \\ & =  \bigoplus_{(b,d)\in \persmap_p(f)} \mathbb{I}_{[b,d)} \big(\max (\phi^{-1}(\{t\}))\big)\\
     & = \bigoplus_{(b,d)\in \persmap_p(f)} 
\mathbb I_{[\phi(b),\phi(d))}(t).
\end{align*}
The first equality follows from the definition of the barcode~$\persmap_p(f)$, i.e.~$\mathrm{H}_p\circ K(f)$ decomposes as~$\bigoplus_{(b,d)\in \persmap_p(f)} \mathbb I_{[b,d)}$. The second equality holds because pre-composition by the map~$t\mapsto \max (\phi^{-1}(\{t\}))$ induces an additive endofunctor on persistence modules. The third one is a consequence of~$\phi$ being non-decreasing, as then~$\max \phi^{-1}(\{t\})\in [b,d)$ is equivalent to~$t\in [\phi(b),\phi(d))$.  

This yields~$\persmap_p(\phi\circ f)= \phi.\persmap_p(f)$, and hence~$\persmap(\phi\circ f)= \phi.\persmap(f)$. 
\end{proof}

\begin{remark}
In the above proof we indirectly made use of the following more general categorical framework where both~$\SComplex (f)$ and~$\mathrm{H}_p \circ \SComplex (f)$  are considered as functors from the category~$(\R , \leq )$ of ordered real numbers defining concrete instances of  filtrations and persistence modules, that is functors
\[
F: (\mathbb R, \leq) \longrightarrow \SCompCat \quad \quad \text { and } \quad \quad \mathbb V : (\mathbb R, \leq) \longrightarrow \Vect.
\]
Precomposition with any endofunctor~$\alpha$ of~$(\mathbb R, \leq)$ defines a right action  both on filtrations  and on persistence modules. Furthermore,  composition by any functor~$L: \SCompCat \to \Vect$ defines a map from filtrations to persistence modules, and we have the following general equivariance result due to  associativity for composition of functors:
\[
L ( F. \alpha) = L \circ (F \circ \alpha) = (L \circ F) \circ \alpha  = L(F) . \alpha.
\]
The endofunctors of~$(\mathbb R, \leq)$ are the (weakly) order preserving maps of~$\R$, that is maps~$\alpha$ satisfying: $t' <t \Rightarrow \alpha (t') \leq \alpha (t)$. Note that~$\alpha$ does  not have to be continuous. 
%
%The functor $\mathrm{H}_p$ is then clearly seen to be equivariant with respect to this action.

For the proof of Lemma~\ref{lemma_commutes_PH_homeo} we have used  the functor~$L= \mathrm{H}_p$ and the fact that an element~$\phi \in \NonDecMaps$ gives rise to an endofunctor~$\alpha:= \max \phi^{-1}$ defined by~$ t \mapsto \max \phi^{-1} (\{ t\})$.
We furthermore used that $\SComplex (\phi \circ f) = \SComplex (f) \circ \max \phi ^{-1}$ and 
$\mathbb I_{[\phi (b) , \phi(d) )} = \mathbb I_{[b,d) } \circ \max \phi ^{-1}$ to translate the given action on the space of  filters and barcodes  into the functorial setting. 
\end{remark}

It is an easy exercise to show that the action of~$\NonDecMaps$ on filters is continuous. We next show that the action is also continuous on barcodes.
\begin{proposition}
\label{proposition_continuous_barcode_action}
$\NonDecMaps$ acts continuously on~$\Barc$, i.e. it is induced by a continous map
\[
\NonDecMaps\times \Barc \rightarrow \Barc.
\]
\end{proposition}
\begin{proof}
%
%It is enough to consider the case where $d=0$. 
We show that the map $\NonDecMaps\times \Barc \rightarrow \Barc$ is sequentially continuous. Let $(\phi_n,D_n)\in \NonDecMaps \times \Barc$ be a sequence converging to some $(\phi,D)$. Let $\epsilon>0$ be small enough such that the intervals of any barcode that is $\epsilon$-close to $D$ satisfy the alternative of Proposition~\ref{proposition_bottleneck_ball}. As~$\I$ is compact, $\phi$ is uniformly continuous and there exists $\eta>0$ such that $|\phi(d)-\phi(b)|<\epsilon$ whenever~$b,d\in \I$ satisfy $|d-b|< \eta$. Let~$n$ be large enough such that~$D_n$ is $\min(\epsilon,\eta)$-close to~$D$, and moreover $\|\phi_n-\phi\|_\infty<\epsilon$.  

If $(b_n,d_n)\in D_n$ is a small interval, i.e. $d_n-b_n<\min(\epsilon,\eta)$, then:
\[|\phi_n(d_n)-\phi_n(b_n)| < |\phi_n(d_n)-\phi(d_n)|+ |\phi(d_n)-\phi(b_n)|+|\phi(b_n)-\phi_n(b_n)| <3\epsilon.\]
Therefore~$(\phi_n(b_n),\phi_n(d_n))$ is a $3\epsilon$-small interval in~$\phi_n.D_n$. 

Else, $(b_n,d_n)$ is $\min(\epsilon,\eta)$-close to a unique interval $(b,d)$ of $D$, and then
\[\|(\phi_n(b_n),\phi_n(d_n))-(\phi(b),\phi(d))\|_\infty\leqslant \|(\phi_n(b_n),\phi_n(d_n))-(\phi(b_n),\phi(d_n))\|_\infty + \|(\phi(b_n),\phi(d_n))-(\phi(b),\phi(d))\|_\infty < 2\epsilon.\]
This yields a canonical matching from $\phi_n.D_n$ to $\phi.D$ with cost less than $3\epsilon$.
\end{proof}
%
%%%%%%%%%%%%%%%%%%%%%%%%%%%%%%%%%%%%%%%%%%%%%%%%%%%%%
%%%%%%%%%%%%%%%%%%%%%%%%%%%%%%%%%%%%%%%%%%%%%%%%%%%%%

\subsection{Stratifications of the spaces of filters and barcodes}
\label{sec:stratifications_barcodes_filters}
%
%In this section, we subdivide the domain and co-domain of~$\persmap$ into pieces, called {strata}, on which the persistence map is well behaved. 
%The notion of stratification we use in this subsection is a rather weak one (compared to Whitney, topological or Quinn stratified spaces) as we do not require gluing conditions between strata.

We introduce a weak notion of stratification in the sense that we will not require (for now) any  conditions on how strata are glued together. We will return to this in Section
\ref{section_barcode_stratification_regularity}.

\begin{definition}
\label{definition_stratification}
A {\em stratification} of a topological space~$\TopSpace$ is a {filtration} $\emptyset =\TopSpace_{-1}\subseteq \TopSpace_0 \subseteq \TopSpace_1 \subseteq \cdots $ by a (possibly infinite) sequence of closed subspaces $\TopSpace_i$, $i\in \N$, where the sets~$\TopSpace_i \setminus {\TopSpace_{i-1}}$ are topological manifolds of dimension~$i$. The  path connected components of~$\TopSpace_i\setminus \TopSpace_{i-1} $ are called {\em $i$-strata}, or strata of dimension~$i$. A {\em stratified map} between two stratified spaces~$\TopSpace$ and   $\TopSpace'$  is a continuous, filtration preserving  map of the underlying spaces. A {\em strongly stratified map} is a stratified map that maps any stratum of~$\TopSpace$ surjectively to a  stratum of~$\TopSpace '$.
\end{definition} 
We will show that~$\Filt_\SComplex$ and~$\CatBarc$ are stratified spaces with strata given by the $\IncHomSpace$-orbits, each homeomorphic to an open standard simplex for some~$i$,
%To fix notation, we denote by $\StandSimplex^i:=\big\{(x_1,\cdots, x_i) | \, 0\leqslant x_1\leqslant \cdots\leqslant x_i \leqslant 1\big\}\subseteq \R^i$  the (geometric realisation) of the standard $i$-simplex, and its interior by
%
\[
\mathring{\StandSimplex}^i:=\big\{(x_1,\cdots, x_i) \,  |\,  \, 0< x_1< \cdots< x_i < 1\big\} \subset \mathbb R^i.
\] 
%

%\subsubsection{Stratification of the space of filters}
%
Recall that~$\IncHomSpace$ acts continuously on the space of filters~$\Filt_\SComplex$ by post-composition. 
%The orbits form a partition. 
For~$f\in \Filt_\SComplex$, we denote the associated orbit by
\[
\StratumF = \StratumF _f := \{ \phi .f \, | \, \phi \in \IncHomSpace \}.
\]
%
%We denote by~$\StratumF$ a typical orbit of the (restricted) action of~$\IncHomSpace$. 
%It is convenient to think of 
Two filters are in the same orbit if they induce the same pre-order on the simplices of~$\SComplex$: $\simplex  \leq \simplex' \iff f(\simplex ) \leq f(\simplex ')$. Inside a given orbit a filter is uniquely determined by the sequence of its values that are not equal to~$0$ or~$1$ sorted in increasing order. Varying~$f$ by an element in~$\IncHomSpace$ varies this sequence over the whole open standard simplex. Thus for each orbit~$\StratumF$ the map
\begin{equation*}
\label{eq_coord_filter_inverse}
   \StratumF\overset \cong \longrightarrow \mathring{\StandSimplex}^{\dim \StratumF}, \quad  \quad \dim \StratumF := \sharp (\text{Im} (f) \cap (0,1)) 
\end{equation*}
that sends a filter to the increasing sequence of its distinct values that are not equal to~$0$ or~$1$ defines an affine homeomorphism. The inverse map~$\ChartFilt{}$ is a coordinate chart for the stratum~$\StratumF$ which in fact extends to the closure, 
\begin{equation}
\label{eq_coord_filter}
  \ChartFilt{} : \StandSimplex ^{\dim \StratumF}  \overset \cong \longrightarrow \overbar{ \StratumF},
\end{equation}
and~$\overbar {\StratumF}$ is the orbit of~$\NonDecMaps$.  We record that the orbits define a stratification.

\begin{proposition}
\label{prop_orbits_stratification_filter}
For each $i\geq 0$, let~$F_i$ be the union of $\IncHomSpace$-orbits~$\StratumF$ with~$\dim \StratumF \leq i$. This
defines a stratification of~$\Filt_\SComplex$. 
The $i$-strata are given by the orbits  
of dimension~$i$.
\end{proposition}
\begin{remark}
\label{remark_filter_stratification_Whitney}
A stratum is simply an equivalence class of filters, where filters are declared equivalent if they induce the same pre-order on simplices. 
This point of view was already adopted in~\cite{leygonie2019framework} in the context of persistence differentiation. Equivalently, the stratification is the hyperplane arrangement generated by the equalities~$f(\simplex)=f(\simplex')$. It is well-known to be a Whitney stratification, but we will not make use of this richer structure here.
\end{remark}
%
%\subsubsection{Stratification of the space of barcodes}
%
Similarly we construct a stratification of barcodes~$\Barc^{d+1}$. For~$D \in \Barc^{d+1}$, we consider the associated $\IncHomSpace$-orbit
\[
\StratumD = \StratumD_D := \{ \phi . D \, | \, \phi \in \IncHomSpace\}.
\]
The orbits partition the space of barcodes~$\Barc^{d+1}$. 
Within such an orbit, the multiplicities and the nestings of bars are constant, and it is only the consecutive values of the interval endpoints that can vary. Thus for each orbit~$\StratumD$ the map
\begin{equation*}
\label{eq_coord_barcode_inverse}
\StratumD\overset \cong \longrightarrow \mathring{\StandSimplex}^{\dim \StratumD}, \quad \quad \dim \StratumD = \dim \StratumD_D= \dim D := \sharp \text { distinct endpoints of } D \text { that are in } (0,1)
\end{equation*}
that sends a barcode to the increasing sequence of its distinct values of interval endpoints that are not equal to~$0$,~$1$ or~$\infty$ defines a homeomorphism. The inverse map~$\ChartBarc$ is then a coordinate chart for the stratum:
\begin{equation}
\label{eq_coord_barcode}
\ChartBarc{}: \mathring{\StandSimplex}^{\dim \StratumD}  \overset \cong \longrightarrow \StratumD.
\end{equation}

\begin{remark}
In fact~$\ChartBarc{}$ is even a local isometry when~$\StratumD$ is equipped with the bottleneck distance and~$\mathring{\StandSimplex}^{\dim \StratumD}$ with the $\|.\|_\infty$-metric: This is because in a fixed stratum barcodes have  a constant number and nestings of bars; hence endpoints can be matched (in an increasing order) and the~$\|.\|_\infty$-metric gives us the cost of the  induced matching, which will be optimal when the barcodes are close enough. 
\end{remark}
\begin{proposition}
\label{prop_orbits_stratification_barcodes}
For each~$i\geq 0$, let~$B_i$ be the union of the $\IncHomSpace$-orbits~$\StratumD$ with~$\dim \StratumD \leq i$. This defines a stratification of~$\Barc^{d+1}$. The $i$-strata are the orbits  of dimension~$i$.
\end{proposition}
\begin{proof}
As orbits provide a partition, the sets~$B_i$ for~$i \in \mathbb N$ form a filtration of~$\Barc^{d+1}$. Each of the~$B_i$ is also closed as the complement is open: barcodes close to a given barcode~$D$ have the same number of bars with endpoints in~$(0,1)$ or more as can be deduced from Proposition~\ref{proposition_bottleneck_ball}.

Furthermore, the complements~$B_i \setminus B_{i-1} $ are by definition the union of finitely many (disjoint) orbits~$\StratumD_1, \dots , \StratumD _{l_i}$, each homeomorphic to~$\mathring{\StandSimplex}^i$. The lemma below implies that the closure of each~$\StratumD_j$ does not intersect any of the~$\StratumD_k$ for~$k \neq j$. Thus a path~$D_t$ of diagrams in~$B_i \setminus B_{i-1}$ with~$D_0 \in \StratumD_j$ cannot leave the orbit~$\StratumD_j$. On the other hand, each stratum~$\StratumD_j$ is path connected since~$D_0$ can be connected to any other diagram~$D_1 \in \StratumD_j$ by a linear path~$D_t :=( t\phi+(1-t)\Id).D_0$ where~$\phi \in \IncHomSpace$ is such that~$D_1=\phi.D_0$.
Hence the path connected components of~$B_i\setminus B_{i-1}$ are the orbits of dimension~$i$, and~$B_i \setminus B_{i-1}$ is a manifold of dimension~$i$.
\end{proof}
\begin{lemma}
\label{lemma_closure_barcode_strata}
Let $D, D'\in \Barc^{d+1}$ be two barcodes. Then the following are equivalent:
\begin{itemize}
\item [(1)] There exists a non-decreasing map~$\phi\in \NonDecMaps$ such that~$D'=\phi . D$;
\item [(2)] $\StratumD_{D'}\subseteq \overbar{\StratumD_{D}}$, i.e. the stratum containing~$D'$ is in the closure of that containing~$D$. 
\end{itemize}
%Furthermore, when~(1) and~(2) hold,  $D' \notin \StratumD_D$ if and only if  $\phi \notin \IncHomSpace$ and $\dim \StratumD_{D'} < \dim \StratumD_D$.
\end{lemma}

\begin{proof}
Assume (1) and let $\phi\in \NonDecMaps$ be such that~$D' = \phi .D$. Consider the paths 
\[
\phi_t := (1-t)\mathrm{Id}+t\phi \quad \quad \text {  and  } \quad \quad D_t:= \phi_t .D.
\]
For~$t \in [0, 1)$,~$\phi_t \in \IncHomSpace$ and hence~$D_t\in \StratumD_D$. By continuity of the monoid action,  Proposition~\ref{proposition_continuous_barcode_action}, the path of barcodes~$ D_t $ is continuous in~$t$ on the whole interval~$ [0,1]$. Consequently, in the limit,~$D'=D_1\in \overbar{\StratumD_D}$.
If~$D'' \in \StratumD_{D'}$ is another barcode from the orbit defined by~$D'$ then there exists a~$\beta \in \IncHomSpace$ with~$D'' = \beta . D'$. Consider $\beta . D_t = (\beta \circ \phi_t ). D$. By the same argument as above, this is a continuous path of barcodes from~$D$ to~$D''$ that is contained entirely in~$\StratumD_D$ with the possible exception when~$t=1$. Hence~$D'' \in \overbar {\StratumD_D}$, and more generally~$\StratumD_{D'} \subseteq \overbar \StratumD_D $ which is (2).

Conversely, assume (2). If~$D' \in \StratumD _D$ then by definition of~$\StratumD_D$ there exists a~$\phi \in \IncHomSpace$ with~$D' = \phi . D$ and~(1) is satisfied. So we may assume~$D' \notin \StratumD_ D$ (and hence the entire orbit~$\StratumD_{D'}$ is contained in the boundary~$\overbar {\StratumD_D} \setminus \StratumD_D$). 
Let~$D_n , n\geq 0$, be a sequence in~$\StratumD_D$ converging to~$D'$, and let~$\phi_n \in \IncHomSpace$ such that~$D_n = \phi_n . D$.
Then, by the characterisation of the local neighborhoods in~$\Barc$, Proposition~\ref{proposition_bottleneck_ball} for small enough~$\epsilon$ and large enough~$n$, the bars in~$D'$ can be matched up (one-to-one) with bars in~$D_n$ that are~$\epsilon$-close, and furthermore any additional bar in~$D_n$ is of length less~$\epsilon$. Let~$\gamma$ be an optimal matching from~$D_n$ to~$D'$ which collapses the small bars. The number of intervals in~$D_n$ is the same as in~$D$, in particular finite, so for~$\epsilon$ small enough relative to the distances between consecutive endpoints~$x_i,x_{i+1}$ of~$D'$, the union of the $\epsilon$-small intervals in~$D_n$ do not cover any segment~$[x_i,x_{i+1}]$. In this case we can use~$\gamma$ to construct a non-decreasing map~$\phi' \in \NonDecMaps$ with~$\phi'. D_n = D'$. Hence, $D' = \phi' . D_n  = (\phi ' \circ \phi _n) . D = \phi .D$ with~$\phi := \phi' \circ \phi_n$. This gives~(1). 
\end{proof}
\begin{remark}
\label{remark_extension_coordchart_barcodes}
The monoid~$\NonDecMaps$ acts coordinate-wise on any simplex~$\StandSimplex^i$, and under this action the orbit of any point in  the interior~$\mathring{\StandSimplex}^i$ is the closed simplex~$\StandSimplex^{i}$. As both~$\ChartFilt$ and~$\ChartBarc{}$ are compatible with the action restricted to~$\IncHomSpace$, they can be extended to equivariant maps from the closed simplex:
\begin{equation}
    \label{eq_extension_mu_inverse}
    \ChartBarc{}: \StandSimplex^{\dim \StratumD}\longrightarrow \overbar{\StratumD}, \quad \text{with} \quad \ChartBarc{}(\phi.x):=\phi.\ChartBarc{}(x). 
\end{equation}
This is easily seen to be well-defined, i.e. given~$x,x'\in \mathring{\StandSimplex}^{\dim \StratumD}$ and~$\phi,\phi'\in \NonDecMaps$ such that~$\phi.x=\phi'.x'$ we have~$\phi.\ChartBarc{}(x)=\phi'.\ChartBarc{}(x')$.
From the above lemma~$\overbar {\StratumD}$ is the monoid orbit of~$\NonDecMaps $, therefore the extension~$\ChartBarc{}$ is surjective. 
Hence~$\overbar {\StratumD}$ (as a set) can be identified as a quotient of the closed standard simplex. Indeed,~$\ChartBarc{}$ is a strongly stratified map where the stratification on the standard simplex is the usual one and~$\overbar {\StratumD}$ is considered a sub-stratified space of~$\Barc^{d+1}$.
In general~$\ChartBarc{}$ is not injective on the boundary, see Example~\ref{example_line_complex}.%, and in particular~$\ChartBarc{}$ may not be invertible to the closure~$\overbar {\StratumD}$. %However we note that in the interior~$\ChartBarc{}: \StratumD\rightarrow \mathring{\StandSimplex}^{\dim \StratumD}$ is a local isometry when~$\mathring{\StandSimplex}^{\dim \StratumD}$ is equipped with the $\|.\|_\infty$-metric and~$\StratumD$ with the bottleneck distance: The~$\|.\|_\infty$-metric translates into a matching of intervals in increasing order of endpoint values, and this is the optimal matching for barcodes that are close enough. 
\end{remark}
With both the stratifications of~$\Filt_\SComplex$ and~$\Barc^{d+1}$ in place, the persistence map $\persmap: \Filt_\SComplex \to \Barc^{d+1} $ is then a map of stratified spaces in the following strong way:
\begin{proposition}
\label{proposition_PH_sends_strata_to_strata}
The persistence map is a strongly stratified map. Namely, let $\StratumF$ be an $i$-stratum in the space of filters. Then there exists a $j$-stratum $\StratumD$  with $j\leq i$ and $\persmap(\StratumF)=\StratumD$.
\end{proposition}
\begin{proof}
Strata in the spaces of filters and barcodes are the $\IncHomSpace$-orbits with respect to which~$\persmap$ is equivariant by Lemma~\ref{lemma_commutes_PH_homeo}. We thus have for  all $f \in \Filt_\SComplex$  and associated stratum $\StratumF_f$
\[\persmap (\StratumF_f) = 
\{ \persmap (\phi. f) = \phi . \persmap (f) \, | \, \phi \in \IncHomSpace \} =  \StratumD_{\persmap(f)} . \qedhere
\]
\end{proof}
Furthermore, over a fixed stratum in $\CatBarc$ the fibers are all homeomorphic.
\begin{proposition}
\label{proposition_same_fiber_common_stratum}
Let $\StratumD\subseteq \CatBarc$ be a barcode stratum. The pre-images of $\persmap$ over elements in~$\StratumD$ are pairwise homeomorphic.
\end{proposition}
\begin{proof}
Let $D,D'\in \StratumD$ so that~$D'=\phi .D$ for some $\phi\in \IncHomSpace$. By the equivariance of~$\persmap$, Lemma~\ref{lemma_commutes_PH_homeo}, the action (by post-composition) of~$\phi$ on~$\Filt_\SComplex$ restricts to a map  from~$\persmap^{-1}(D)$ to $\persmap^{-1}(D')= \persmap ^{-1}( \phi. D) = \phi . \persmap ^{-1} (D)$. 
\end{proof}
Therefore~$\CatBarc$ is a stratified subspace of~$\Barc^{d+1}$ consisting of the union of strata~$\persmap(\StratumF)$, where~$\StratumF\subseteq\Filt_\SComplex$ varies over the set of strata of the space of filters. In particular~$\CatBarc$ is a finite union of strata, finite dimensional and compact, unlike~$\Barc^{d+1}$ which has infinitely many strata of arbitrarily large dimensions. 
\subsection{The space $\Barc_\SComplex$ as a quotient space}
We will now show that as a topological space~$\CatBarc$ is the quotient of the space of filters~$\Filt_\SComplex$ induced by the persistence map.
\begin{proposition}
\label{prop_quotient_topology}
The quotient topology on~$\Barc_\SComplex=\persmap(\Filt_\SComplex)$ induced by~$\persmap$ agrees with the bottleneck topology, that is we have a homeomorphism from the quotient
 \[\overbar{\persmap}:(\Filt_\SComplex \, /   \sim ) \, \, \overset \cong  \longrightarrow \, \CatBarc,\]
where~$\sim$ is defined by $f\sim f' \Leftrightarrow \persmap(f)=\persmap(f')$.
\end{proposition}
\begin{proof}
%
%Let~$U$ be an open set in~$\CatBarc$. 
By Theorem~\ref{theorem_stability},~$\persmap: \Filt_K \rightarrow \CatBarc$ is continuous, and hence, by the universal property of the quotient, it induces a continuous bijection~$\overbar{\persmap}$.
%: hence~$U$ is an open set in~$(\Filt_\SComplex \, /   \sim )$.  
It remains to prove that the inverse~$\overbar{\persmap}^{-1}$ is also continuous, or equivalently that~$\overbar{\persmap}$ is open.

Let~$U$ be an open set in~$(\Filt_\SComplex \, /   \sim )$ and let~$D\in U$. Then by definition of the quotient topology~$\persmap^{-1}(U)$ is open and contains~$\persmap^{-1}(D)$. 
Since~$\persmap$ is continuous,~$\persmap^{-1}(D)$ is closed, and being a subset of~$\I^\SComplex$ it is in fact compact. Thus for some~$\eta>0$ we have that the~$\eta$-offset of~$\persmap^{-1}(D)$ lies in~$\persmap^{-1}(U)$:
\[
\persmap^{-1}(D)_{\eta}:=\big\{ f\in \Filt_\SComplex, \,  \exists g\in \persmap^{-1}(D), \, \|f-g\|_\infty<\eta \big\} \subseteq \persmap^{-1}(U).  \]
 
We will show that~$\persmap^{-1}(D')\subseteq \persmap^{-1}(D)_\eta$ for~$D'$ close enough to~$D$ in the bottleneck metric. By the above this implies that~$\persmap^{-1}(D')\subseteq \persmap^{-1}(U)$, which amounts to~$D'\in U$, and hence~$U$ is an open set in~$\CatBarc$.

By Proposition~\ref{proposition_bottleneck_ball}, for~$\epsilon$ small enough, bars~$(b',d')\in D'$ that are not~$\epsilon$-small are matched up with bars of~$D$ that are~$\epsilon$-close. We consider the following equivalence relation on interval endpoints~$b',d'$ of~$D'$. First, we deem equivalent all endpoints that are~$\epsilon$-close to the same endpoint~$x_i$ of~$D$. Then, we deem equivalent two~$\epsilon$-small intervals that overlap:~$[b',d']\cap [b'',d'']\neq \emptyset$ and take the transitive closure of that relation. Since there are at most~$\sharp \SComplex$ endpoints in~$D'$, the endpoints in the same equivalence class span a range of size at most~$\sharp \SComplex \times \epsilon$. Thus if~$\epsilon$ has been chosen small enough to start with, then it is impossible to find in the same equivalence class two endpoints of~$D'$ that are~$\epsilon$-close to distinct endpoints~$x_i\neq x_j$ of~$D$.

This allows constructing a map~$\phi:\I \to \I$ such that~$\phi.D'=D$ as follows: Over the span of an equivalence class of endpoints we define~$\phi$ as the constant map with value~$x_i$ in the case where there is an endpoint~$b'$ or~$d'$ which is~$\epsilon$-close to the endpoint~$x_i$ of~$D$, and with an arbitrary value in the span in the case where there is no such endpoint in the equivalence class. We extend~$\phi$ linearly on~$\I$. 
By design~$\phi$ differs from the identity map by at most~$\sharp \SComplex \times \epsilon$, because the span of each equivalence class has diameter bounded by~$\sharp \SComplex \times \epsilon$. Hence if we take an arbitrary~$f\in \persmap^{-1}(D')$, then~$g:=\phi\circ f$ belongs to~$\persmap^{-1}(D)$ by equivariance of~$\persmap$, and~$g$ is~$(\sharp \SComplex \times \epsilon)$-close to~$f$. Up to shrinking~$\epsilon$ so that~$\sharp \SComplex \times \epsilon<\eta$, we have~$f\in \persmap^{-1}(D)_{\eta}$. Therefore~$\persmap^{-1}(D')\subseteq \persmap^{-1}(D)_{\eta}$.
\end{proof}
The top dimensional strata of $\Filt_\SComplex$ consists of the injective filters~$f: \SComplex \to  \I$ that do not take the values~$0$ or~$1$. Hence the dimension of the top strata is~$\sharp \SComplex$. The interval endpoints of a barcode~$D=\persmap (f)$ form a subset of the values of~$f$ and in general~$\dim \StratumD_D \leq \dim \StratumF_f$. However, when~$f$ is injective, each simplex enters the next sublevel set of the filtration by itself and hence induces a change in homology. Thus in particular we see that the dimension of the top dimensional barcode strata in~$\Barc_\SComplex$ is again~$\sharp \SComplex$. Let
\[
\Barc_\SComplex ^{\mathrm {top}}:=  \bigcup _{\dim \mathcal B = \sharp \SComplex } \mathcal B =
%  \big \{D\in \CatBarc \, | \,  \dim D= \sharp \SComplex \big\}= 
\{\persmap(f) \, | \,  f\in \Filt_\SComplex \text{ is injective and does not take the values $0$ or $1$}\big\}.
\]
%We now observe that~$\CatBarc$ is the closure of its collection of top dimensional strata.
%
\begin{proposition}
\label{proposition_image_PH_determined_by_topdim}
The barcode space~$\CatBarc$ is the closure of its top dimensional strata, i.e.
\[\CatBarc=\overbar{\CatBarc^{\mathrm{top}}}.\]
\end{proposition}
\begin{proof}
%The interval endpoints of a barcode~$D=\persmap(f)$ form a subset of the values of~$f$, hence~$\dim D\leq \sharp \SComplex$. The upper-bound is attained exactly when~$f$ is injective, as then each simplex~$\simplex$ enters alone in the associated sublevel set filtration hence induces change in homology. 

%It remains to show that~$\CatBarc=\overbar{\CatBarc^{\mathrm{top}}}$. 
Let $D=\persmap(f)$ be a barcode in the image. We can always factor~$f$ as~$f=\phi\circ g$ for some injective filter~$g\in \Filt_\SComplex$ and a non-decreasing map~$\phi\in \NonDecMaps$. By the equivariance of~$\persmap$ (Lemma~\ref{lemma_commutes_PH_homeo}), we have~$D=\phi.\persmap(g)$. Up to an arbitrarily small perturbation,~$g$ does not take the values~$0$ and~$1$, and  hence~$\persmap(g)$ is an element in a top dimensional  stratum. The result then follows from Lemma~\ref{lemma_closure_barcode_strata}.
\end{proof}
The space $\Barc_\SComplex$ can thus be built as a quotient of a finite collection of 
(closed) simplices  corresponding to the top dimensional barcode strata where some of the faces may be identified to each other and where $i$-dimensional faces may be reduced to a
$j$-dimensional simplex through collapsing $(i-j)$-dimensional affine subspaces (corresponding to bars of length zero). The following example illustrates this.
\begin{example}
\label{example_line_complex}
Let~$\SComplex$ represent the unit interval with vertices~$a,b$ and 1-simplex~$\sigma$.
The space of filters consists of two 3-dimensional strata corresponding to the induced orderings~$a<b<\sigma$ and~$b<a<\sigma$, which are mapped to each other via the action on~$\SComplex $ given by the  involution~$(a,b)$ on its set of vertices. All faces are included in~$\Filt_\SComplex$ and the two simplices are glued together along their common face corresponding to $a=b \leq \sigma$. Under the map~$\persmap$ the two 3-simplices are identified to one 3-simplex giving a unique top dimensional stratum~$\StratumD _{\mathrm{top}}$ in the space of barcodes~$\Barc_\SComplex $ parametrising barcodes of the form $\{(x_1, \infty), (x_2, x_3)\}$ with $0< x_1< x_2 < x_3 < 1$. The barcode space~$\Barc_\SComplex$ can then be identified with the quotient space of the closed 3-simplex where the  2-dimensional face corresponding to $0\leq x_1 < x_2=x_3 \leq 1$ is collapsed to the line segment $0\leq x_1=x_2=x_3 \leq 1$, i.e. we identify:
\[
(x_1, x_2, x_3 ) \sim (x_1, x_2', x_3') \quad \text{ whenever }  x_2 =x_3 \text{ and } x_2'= x_3'.
\]
See Figure~\ref{fig:image_line_complex}.
We thus see that the fiber $\persmap^{-1} (D)$ of a barcode $D = \{(x_1, \infty), (x_2, x_3)\} \in \Barc_\SComplex$ consists of two points if $x_1 <x_2 <x_3$, of one point if $x_1=x_2 < x_3$, and of an interval (consisting of  two intervals glued together) if $x_1 < x_2=x_3$.
\begin{figure}[ht]
\centering
\includegraphics[width=0.3\textwidth]{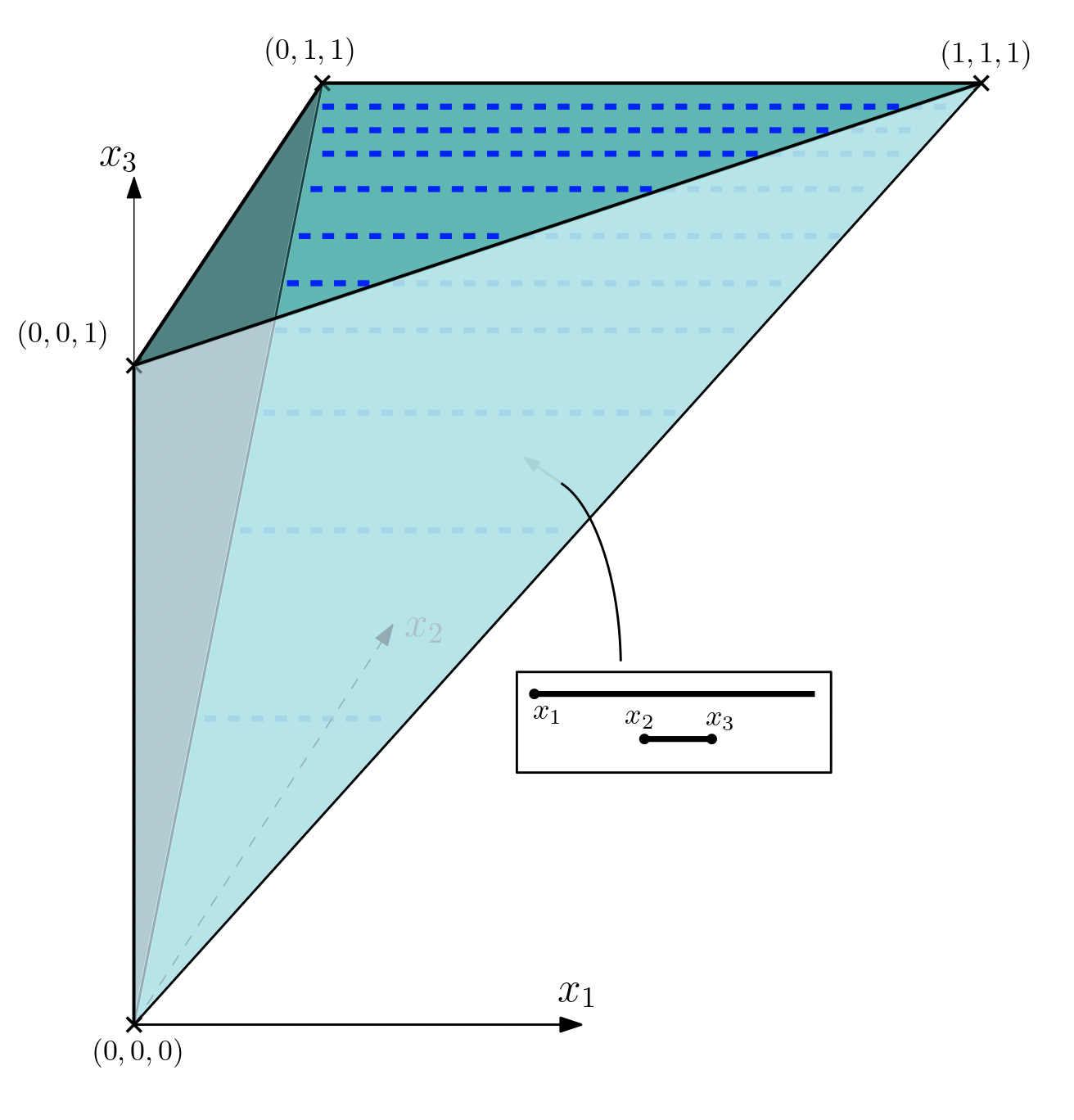}
\caption{When $\SComplex$ is the simplicial interval, the space $\Barc_\SComplex$ is the quotient space of the closed 3-simplex with each dotted line on the back face collapsed to one point. 
}
\label{fig:image_line_complex}
\end{figure}
\end{example}

\newpage
\section{The persistence map as a polyhedral stratified fiber bundle}
Although the persistence map~$\persmap$ is globally not a fibration, we show in section~\ref{sec:theorem_1} that it is a trivial fibration over each barcode stratum with a polyhedral complex as fiber. Using this structure in section~\ref{sec:topology_prediction}, we derive topological properties of the fiber.
%such as an upper-bound on the dimension of the polyhedra constituting the fiber.

%
\subsection{Polyhedral structure on the fiber}
\label{sec:theorem_1}
We strengthen Propositions~\ref{proposition_PH_sends_strata_to_strata} and~\ref{proposition_same_fiber_common_stratum}, showing that, over each barcode stratum in the image~$\CatBarc=\persmap(\Filt_\SComplex)$, the persistence map is a trivial fiber bundle with a polyhedral complex as fiber. The intuition behind this result is that~$\persmap$ can be viewed as a piecewise linear projection as follows. Given a filter stratum $\StratumF\subseteq \Filt_\SComplex$ and the barcode stratum $\StratumD=\persmap(\StratumF)$, the restriction $\persmap_{|\StratumF}$ can also be described as a map $\pi_{\StratumF}^{\StratumD}:\mathring{\StandSimplex}^{\dim \StratumF}\rightarrow \mathring{\StandSimplex}^{\dim \StratumD}$ via the coordinate charts~$\ChartFilt{}: \mathring{\StandSimplex}^{\dim \StratumF} \overset\cong \longrightarrow \StratumF $ and~$\ChartBarc: \mathring{\StandSimplex}^{\dim \StratumD} \overset\cong \longrightarrow  \StratumD $ given by Eq.~\eqref{eq_coord_filter} and Eq.~\eqref{eq_coord_barcode}:
\begin{equation}
\label{eq_persmap_projection}
\begin{tikzpicture}

\node[] (S) at (-2,2) {$\mathring{\StandSimplex}^{\dim \StratumF}$};
\node[] (D) at (-2,0) { $\mathring{\StandSimplex}^{\dim \StratumD}$};

\node[] (DeltaS) at (0,2) {$\StratumF$};
\node[] (DeltaD) at (0,0) {$\StratumD$};

\draw[->,dashed] (S)--(D) node[midway, left] {$\pi_{\StratumF}^{\StratumD}$};
\draw[->] (DeltaS)--(DeltaD) node[midway, right] {$\persmap$};
\draw[->] (S)--(DeltaS) node[midway, above]{$\ChartFilt$};
\draw[->] (D)--(DeltaD) node[midway, above]{$\ChartBarc$};
\end{tikzpicture}
\end{equation}
The map~$\pi_{\StratumF}^{\StratumD}$ is the linear projection from~$\R^{\dim \StratumF}$ to~$\R^{\dim \StratumD}$ that records the values of a filter~$f$ which are bounded interval endpoints in the associated barcode~$\persmap(f)$. That~$\persmap$ can be described by this diagram follows from the following two elementary observations: (i)~$\pi_{\StratumF}^{\StratumD}$ is $\IncHomSpace$-equivariant since~$\ChartFilt$ and~$\ChartBarc$ are equivariant, where $\IncHomSpace$ acts on~$\mathring{\StandSimplex}^{\dim \StratumF}$ and~$\mathring{\StandSimplex}^{\dim \StratumD}$ coordinate-wise; and (ii) the set of bounded interval endpoints in the barcode~$\persmap$ is a subset of the values of~$f$.

The fiber of such a projection map, restricted to the polyhedron~$\mathring{\StandSimplex}^{\dim \StratumF} \cong \StratumF$, is itself a polyhedron. In this section, we glue the polyhedra obtained in this way over the various filter strata~$\StratumF$ in order to describe the whole fiber of~$\persmap$ over a barcode stratum as a complex of polyhedra.

Recall that a {\em (bounded) polyhedron} is a bounded, finite intersection of closed half-spaces in a Euclidean space. The {\em dimension} of a polyhedron is the dimension of its affine hull. A {\em face} of a polyhedron~$\Polytope$ is the intersection of~$\Polytope$ with a supporting hyperplane, and is itself a polyhedron. In particular, a polyhedron~$\Polytope$ is a bounded convex Euclidean set. For such sets, we have the notions of relative interior and relative boundary, which offer the advantage not to depend on the ambient Euclidean space, and with a slight abuse of notations we denote them by~$\mathring{ \Polytope}$ and~$\rb \Polytope$ respectively.

\begin{definition}
\label{definition_polyhedral_complex}
A {\em polyhedral complex} is a finite set~$\PolyComplex$ of polyhedra in some Euclidean space~$\R^n$, such that (i) if~$F$ is a face of $\Polytope\in \PolyComplex$, then $F\in \PolyComplex$, and (ii) for all~$\Polytope,\Polytope' \in \PolyComplex$, the intersection~$\Polytope\cap \Polytope'$ is either empty or is a face of both~$\Polytope$ and~$\Polytope'$. By convention, the empty set is in~$\PolyComplex$. The {\em support} of~$\PolyComplex$ is $\bigcup_{\Polytope\in \PolyComplex} \Polytope \subseteq \R^n$. The {\em dimension} of a polyhedral complex is the maximal dimension of its polyhedra.
\end{definition} 
A map of polyhedral complexes, or {\em polyhedral map}, is a map that sends a polyhedron of the domain to a polyhedron of the co-domain surjectively, and whose restriction to each polyhedron $\Polytope$ of the domain is affine (i.e. can be extended to an affine map to the ambient space of $\Polytope$). In the case where the polyhedra are simplicial complexes, the notion of polyhedral map coïncides with that of simplicial map, in that it is induced by a map defined on the abstract simplicial complexes. More generally, polyhedral complexes can be thought of as geometric realisations of simplicial complexes. In fact, it is a standard fact that a polyhedral complex admits a finite triangulation on the same set of vertices. For the proof of this result and more about polyhedral geometry, we refer the reader to~\cite{gubeladze2009polytopes}.

\begin{theorem}
\label{theorem_fiber_bundle_polyhedral}
Let $\StratumD \subseteq \CatBarc$ be a barcode stratum in the image of $\persmap$, and $D\in \StratumD$ be any barcode. For each filter stratum $\StratumF\subseteq \Filt_\SComplex\cap \persmap^{-1}(\StratumD)$, let 
\[\PolytopeClosed=\overbar{\persmap^{-1}(D)\cap \StratumF} \subseteq \R^K\]
be the (closure of the) restriction to the stratum~$\StratumF$ of the fiber of the persistence map over~$D$. Then:
\begin{itemize}
\item[\bf{(a)}] Each~$\PolytopeClosed$ is a polyhedron in~$\R^K$ of dimension $\dim \StratumF-\dim \StratumD$, and is affinely isomorphic to the product of~$\dim \StratumD+1$ standard simplices (of various dimensions):
\[\StandSimplex_0\times \StandSimplex_1\times \StandSimplex_2 \times \cdots \times \StandSimplex_{\dim \StratumD -1} \times \StandSimplex_{\dim \StratumD}. \]
Moreover, the relative interior of~$\PolytopeClosed$ is~$\PolytopeOpen$;

\item[\bf{(b)}] The fiber $\persmap^{-1}(D)$ is the support of the polyhedral complex
\[\bigg\{\PolytopeClosed \, | \,  \StratumF \text{ filter stratum } \bigg\};\] 

\item[\bf{(c)}] There is a homeomorphism~$\PolyMap$ giving the following commutative diagram:
\begin{equation}
\label{eq_diagram_c_theorem_1}
\begin{tikzpicture}
\node[] (a) at (2.5,-3) {$\StratumD$};
\node[] (b) at (0,-1) {$\StratumD \times \persmap^{-1}(D)$};

\node[] (e) at (5,-1) {$\persmap^{-1}(\StratumD)$};
\draw[->] (b)--(e) node[midway, above] {$\PolyMap$};

\draw[->] (b)--(a) node[midway, left] {$\pi_1$};
\draw[->] (e)--(a) node[midway, right] {$\persmap$};

\end{tikzpicture}
\end{equation}
where $\pi_1$ denotes the projection onto the first factor. Additionally, for any barcode~$D'\in \StratumD$ and filter stratum~$\StratumF$, the restricted map $\PolyMap(D',.)_{|\StratumF}$ is an affine isomorphism between $\PolytopeOpen$ and $\persmap_{|\StratumF}^{-1}(D')$. In particular, the polyhedral structure of $\persmap^{-1}(D)$ is the same for all barcodes in~$\StratumD$.
\end{itemize}
\end{theorem}
In assertion {\bf (a)} above, the subcript in~$\StandSimplex_i$ is used as an index to order a list of standard simplices. This should not be  confused with the notation~$\StandSimplex^i$ where the superscript denotes the dimension of  the standard simplex.
In Fig.~\ref{fig:fiber_line_complex} we illustrate this part of the theorem and include a sample computation for $\dim \StratumF-\dim \StratumD$.
\begin{figure}[ht]
\centering
\includegraphics[width=0.7\textwidth]{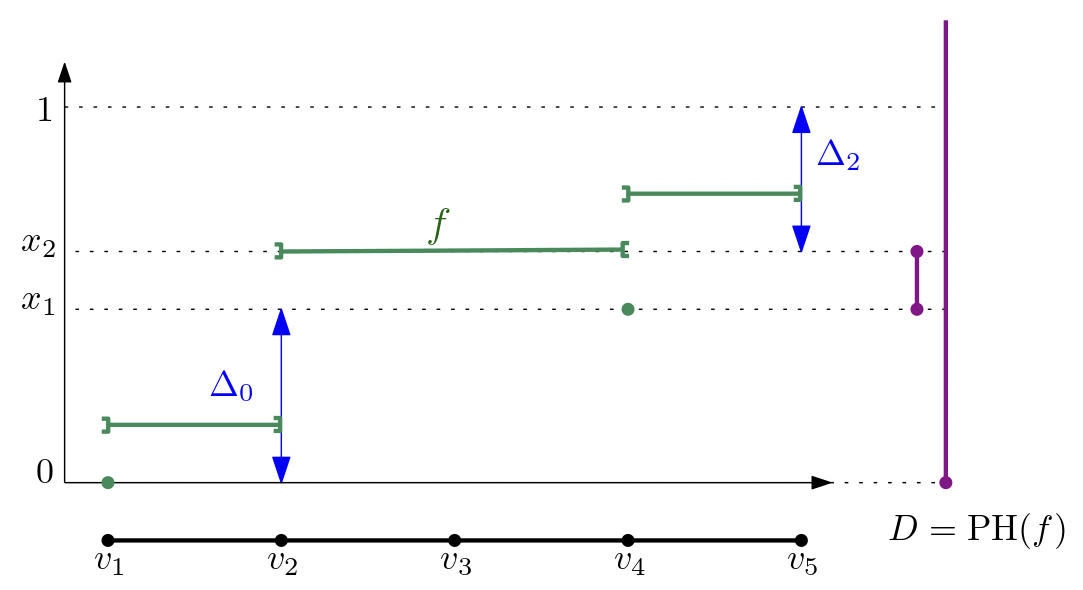}
\caption{The filter~$f$ (green) on the complex obtained from subdividing the unit interval with five vertices yields the $2$-dimensional barcode~$D$ (purple). The stratum~$\StratumF$ of~$f$ contains those filters which assign values in strictly increasing order to the vertices $v_1,v_2,v_4,v_3$ and~$v_5$ and whose values over edges $(v_i,v_{i+1})$ are the maxima of their values over the endpoints~$v_i$ and~$v_{i+1}$. The filters in the restricted fiber~$\PolytopeClosed$ must fix the equalities~$f(v_1)=0$,~$f(v_3)=x_2$ and~$f(v_4)=x_1$, but we may vary $f(v_2)\in [0,x_1]$ and $f(v_5)\in [x_2,1]$. These two degrees of freedom correspond to the two simplices~$\StandSimplex_0$ and~$\StandSimplex_2$, which are in this case of dimension~$1$, whereas~$\StandSimplex_1$ is a singleton as no value of~$f$ lies in~$(x_1,x_2)$. We thus get the affine isomorphism between~$\PolytopeClosed$ and the product $\StandSimplex_0 \times \StandSimplex_1 \times \StandSimplex_2$ of Theorem~\ref{theorem_fiber_bundle_polyhedral}. In particular, we check that the dimension of the polyhedron~$\PolytopeClosed$ is~$2=5-3= \dim \StratumF- \dim \StratumD_D$.}
\label{fig:fiber_line_complex}
\end{figure}
\begin{proof}
Recall from Eq.~\eqref{eq_persmap_projection} that the persistence map rewrites as the projection map $\pi_{\StratumF}^{\StratumD}:\mathring{\StandSimplex}^{\dim \StratumF} \rightarrow \mathring{\StandSimplex}^{\dim \StratumD}$ onto the~$\dim \StratumD$ consecutive filter values that modify the homology groups of the sublevel sets, and that the fibers of~$\pi_{\StratumF}^{\StratumD}$ are polyhedra. More precisely, let~$(x_1,\cdots,x_{\dim \StratumD}):=\ChartBarc^{-1}(D)$ be the consecutive endpoints of~$D$, and let~$1\leq i_1<\cdots <i_{\dim \StratumD}\leq  \dim\StratumF$ be the projection coordinates of~$\pi_{\StratumF}^{\StratumD}$. Then:
\begin{align*}
(\pi_{\StratumF}^{\StratumD})^{-1}(\ChartBarc^{-1}(D)) &= \big\{ (x'_i)_{i=1}^{\dim \StratumF} \in \mathring{\StandSimplex}^{\dim \StratumF} \mid \, x'_{i_1}=x_1, \cdots, x'_{i_{\dim \StratumD}} = x_{\dim \StratumD}\big\}\nonumber \\
&=\big\{ (x'_i)_{i=1}^{\dim \StratumF} \in \R^{\dim \StratumF} \mid \,(0< x'_1 < \cdots< x'_{i_1}=x_1) \cap \cdots \cap (x_{\dim \StratumD}=x'_{i_{\dim \StratumD}} < \cdots< x'_{\dim \StratumF}< 1)\big\}\nonumber\\ 
& \cong \mathring{\StandSimplex}^{i_1} \times \cdots \times \mathring{\StandSimplex}^{\dim \StratumF- i_{\dim \StratumD}},
\end{align*}
where the last homeomorphism is an affine isomorphism. Therefore the restricted fiber $\persmap^{-1}(D)\cap \StratumF$ is affinely isomorphic to the product 
\begin{equation}
\label{eq_2_theorem1}
\persmap^{-1}(D)\cap \StratumF \cong (\pi_{\StratumF}^{\StratumD})^{-1}(\ChartBarc^{-1}(D))\cong \mathring{\StandSimplex}_0\times \mathring{\StandSimplex}_1\times \mathring{\StandSimplex}_2 \times \cdots \times \mathring{\StandSimplex}_{\dim \StratumD -1} \times \mathring{\StandSimplex}_{\dim \StratumD} 
\end{equation}
of open simplices~$\mathring{\StandSimplex}_i$ whose dimensions sum up to~$\dim \StratumF- \dim \StratumD$, where the open simplex~$\mathring{\StandSimplex}_i$ corresponds to values in-between~$x_i$ and~$x_{i+1}$, the $i$-th and $i+1$-th endpoint values in~$D$ (recall the convention that~$x_0=0$ and~$x_{\dim D +1}=1$). The linear isomorphism~$\ChartFilt$ extends to a stratified linear isomorphism between the closed simplex~$\StandSimplex^{\dim \StratumF}$ and~$\overbar{\StratumF}$. Therefore, the homeomorphism in Eq.~\eqref{eq_2_theorem1} extends to the closure and yields the affine homeomorphism:
\begin{equation}
\label{eq_3_theorem1}
\overbar{\persmap^{-1}(D)\cap \StratumF} \cong \StandSimplex_0\times \StandSimplex_1\times \StandSimplex_2 \times \cdots \times \StandSimplex_{\dim \StratumD -1} \times \StandSimplex_{\dim \StratumD}. 
\end{equation}
This proves assertion {\bf (a)}. We can also deduce from the previous arguments that if $\StratumF'\subseteq \rb \StratumF$ is a boundary stratum, then 
\begin{equation}
\label{eq_boundary_polytope_theorem_1}
\overbar{\persmap_{|\StratumF'}^{-1}}(D) \subseteq \rb \PolytopeClosed.
\end{equation}
To see this, let us fix $f\in \StratumF \cap \persmap^{-1}(D)$ and let~$f'$ be an arbitrary filter in~$\overbar{\persmap_{|\StratumF'}^{-1}}(D)$. Then the respective coordinates~$\ChartFilt^{-1}(f)$ and~$\ChartFilt^{-1}(f')$ are in the fiber of~$\pi_{\StratumF}^{\StratumD}$ over~$\ChartBarc^{-1}(D)$, therefore so is the straight line $[\ChartFilt^{-1}(f),\ChartFilt^{-1}(f')]$ by linearity of the projection. Since $\ChartFilt([\ChartFilt^{-1}(f),\ChartFilt^{-1}(f')))\subseteq \StratumF$, we have $\ChartFilt([\ChartFilt^{-1}(f),\ChartFilt^{-1}(f')))\subseteq \PolytopeOpen$ and we deduce that $f'\in \PolytopeClosed$, and consequently $\overbar{\persmap_{|\StratumF'}^{-1}}(D) \subseteq \PolytopeClosed$. Since $\StratumF'\subseteq \rb \StratumF$, the polyhedron $\overbar{\persmap_{|\StratumF'}^{-1}}(D)$ in fact lies in the relative boundary of~$\PolytopeClosed$.

We now show assertion {\bf (b)}, namely that the set of polyhedra~$\PolytopeClosed$ is a polyhedral complex. So we first show that if~$\PolytopeClosed$ is a polyhedron and~$F\subset \PolytopeClosed$ is one of its face, then $F$ is of the form~$\overbar{\persmap_{|\StratumF'}^{-1}}(D)$. We assume that~$F$ is a proper face, i.e. has codimension~$1$ in~$\PolytopeClosed$. The general case follows by induction on the codimension. 

The restriction of the affine homeomorphism of Eq.~\eqref{eq_3_theorem1} to the face $F$ implies that $F$ is affinely isomorphic to the product
\begin{equation}
\label{eq_4_theorem1}
F \cong \StandSimplex_0\times \StandSimplex_1\times \StandSimplex_2 \times \cdots \times \StandSimplex'_i \times \cdots \times \StandSimplex_{\dim \StratumD -1} \times \StandSimplex_{\dim \StratumD},
\end{equation}
where $\StandSimplex'_i$ is a proper face of~$\StandSimplex_i$. Note that~$\StandSimplex'_i$ is obtained by replacing one inequality between consecutive coordinates in~$\StandSimplex_i$ with an equality. This uniquely defines a stratum $\StratumF'\subseteq \rb \StratumF$ of codimension~$1$ in~$\StratumF$ such that $\mathring{F}\subseteq \StratumF'$. Since $F\subseteq \persmap^{-1}(D)$, we have $F\subseteq \overbar{\persmap_{|\StratumF'}^{-1}}(D)$. However, by Eq.~\eqref{eq_boundary_polytope_theorem_1}, the polyhedron~$\overbar{\persmap_{|\StratumF'}^{-1}}(D)$ lies inside the relative boundary of~$\PolytopeClosed$, and in fact by convexity, inside a proper face of~$\PolytopeClosed$. Since $F\subseteq \overbar{\persmap_{|\StratumF'}^{-1}}(D)$ is itself a proper face, we deduce that $F=\overbar{\persmap_{|\StratumF'}^{-1}}(D)$, as desired.  

We now show that a non-empty intersection $\PolytopeClosed\cap \overbar{\persmap_{|\StratumF'}^{-1}}(D)$ of two polyhedra in the fiber is a common face of $\PolytopeClosed$ and $\overbar{\persmap_{|\StratumF'}^{-1}}(D)$. If $\StratumF=\StratumF'$ there is nothing to prove. Otherwise, $\StratumF\cap \StratumF'=\emptyset$ so that $\PolytopeOpen \cap \persmap_{|\StratumF'}^{-1}(D)=\emptyset$. We consider two cases:
\begin{enumerate}
\item Either $\StratumF'\subseteq \rb \StratumF$, and then $\overbar{\persmap_{|\StratumF'}^{-1}}(D)$ lies in a proper face $\overbar{\persmap_{|\StratumF''}^{-1}}(D)$ of $\PolytopeClosed$, and we must have~$\StratumF'\subseteq \rb \StratumF''$. Thus we are done by an induction on the codimension of~$\StratumF'$ in~$\StratumF$. Similarly if~$\StratumF\subseteq \rb \StratumF'$;
\item Or $\PolytopeClosed$ and $\overbar{\persmap_{|\StratumF'}^{-1}}(D)$ intersect only at their relative boundaries. In this case, by convexity of the two polyhedra, their intersection is the intersection $\overbar{\persmap_{|\StratumF_2}^{-1}}(D)\cap \overbar{\persmap_{|\StratumF'_2}^{-1}}(D)$ of some proper faces $\overbar{\persmap_{|\StratumF_2}^{-1}}(D)\subseteq \PolytopeClosed$ and $\overbar{\persmap_{|\StratumF'_2}^{-1}}(D)\subseteq \overbar{\persmap_{|\StratumF'}^{-1}}(D)$. We are then left with the initial problem with polyhedra of smaller dimensions. We then conclude via an induction on the dimension of the polyhedra.
\end{enumerate}

We now address the proof of assertion {\bf (c)}. Given a barcode $D'\in \StratumD$, define $\phi_{D'}: \I \to \I$ to be the unique piecewise linear interpolation of the increasing map that takes the (bounded) endpoints of $D'$ to the (bounded) endpoints of~$D$, further fixing~$0$ and~$1$. Clearly,~$\phi_{D'}$ is an orientation preserving homeomorphism. From the homeomorphism~$\ChartBarc:  \mathring{\StandSimplex}^{\dim \StratumD} \overset\cong \longrightarrow \StratumD$, the intervals' endpoints in a barcode~$D'$ vary continuously with~$D'\in \StratumD$. Therefore~$D'\mapsto \phi_{D'}$ is continuous, and in turn the map 
\begin{equation}
\label{equation_Phi}
\PolyMap: (D',f)\in \StratumD \times \persmap^{-1}(D) \mapsto \phi_{D'}^{-1}\circ f \in \persmap^{-1}(\StratumD)
\end{equation}
is continuous. Similarly, its inverse
given by $\PolyMap^{-1} (f') = 
(\persmap(f'), \phi _{\persmap(f')} \circ f')$ is continuous, so that $\PolyMap$ is a homeomorphism.
Using Lemma~\ref{lemma_commutes_PH_homeo}, we have
$\persmap\circ \PolyMap = \pi_1$, i.e. the diagram in Eq.~\eqref{eq_diagram_c_theorem_1} commutes.
%
\begin{comment}
\begin{center}
\begin{tikzpicture}
\node[] (a) at (2.5,-3) {$\StratumD$};
\node[] (b) at (0,-1) {$\StratumD \times \persmap^{-1}(D)$};

\node[] (e) at (5,-1) {$\persmap^{-1}(\StratumD)$};
\draw[->] (b)--(e) node[midway, above] {$\PolyMap$};

\draw[->] (b)--(a) node[midway, left] {$\pi_1$};
\draw[->] (e)--(a) node[midway, right] {$\persmap$};

\end{tikzpicture}
\end{center}
\end{comment}
%

Let $\StratumF\subseteq \persmap^{-1}(\StratumD)$ be a filter stratum in the fiber. For $D'\in \StratumD$, $\PolyMap(D',.)= \phi_{D'}^{-1}\circ\_$ is the post-composition by $\phi_{D'}^{-1}\in \IncHomSpace$, see Eq~\eqref{equation_Phi}, so by Lemma~\ref{lemma_commutes_PH_homeo}, the previous homeomorphism restricts to:
\begin{center}
\begin{tikzpicture}
\node[] (a) at (2.5,-3) {$\StratumD$};
\node[] (b) at (0,-1) {$\StratumD \times \PolytopeOpen$};

\node[] (e) at (5,-1) {$\StratumF$};
\draw[->] (b)--(e) node[midway, above] {$\PolyMap$};

\draw[->] (b)--(a) node[midway, left] {$\pi_1$};
\draw[->] (e)--(a) node[midway, right] {$\persmap$};
\end{tikzpicture}
\end{center}
To finish the proof, we show that the homeomorphism $\PolyMap(D',.):\persmap^{-1}(D)\cap \StratumF \rightarrow \persmap^{-1}(D')\cap \StratumF$ is the restriction of an affine endomorphism of~$\R^\SComplex$, for any barcode~$D'\in \StratumD$. Equivalently, we describe the coordinate functions $\PolyMap(D',.)_{\simplex}:f \mapsto \PolyMap(D',f)(\simplex)$ as affine forms, for each simplex $\simplex\in \SComplex$. Let~$(x'_1,\cdots,x'_{\dim \StratumD}):= \ChartBarc^{-1}(D')$, and let~$x'_{0}=0$ and~$x'_{\dim \StratumD+1}=1$ by convention. Given~$\simplex\in \SComplex$ and~$f\in \persmap^{-1}(D')\cap \StratumF$, there is an index~$0 \leqslant i \leqslant \dim\StratumD$ such that~$x'_{i}\leqslant f(\simplex)\leqslant x'_{i+1}$. If~$f'\in \persmap^{-1}(D')\cap \StratumF$ is another filter in the fiber, we have~$f'=\phi\circ f$ for some~$\phi \in \IncHomSpace$, and~$\phi.D'=D'$ by the equivariance Lemma~\ref{lemma_commutes_PH_homeo}, so that~$x'_{i}\leqslant f'(\simplex)\leqslant x'_{i+1}$ as well. Since $\phi_{D'}^{-1}$ is affine over $[x'_i,x'_{i+1}]$, we conclude that~$\PolyMap(D',.)_{\simplex}$ is the restriction of an affine map, as desired. 
\end{proof}
\subsection{Topology of the fiber}
\label{sec:topology_prediction}

In this section, we collect a few results that restrict the topology of the fiber of~$\persmap$ over a barcode~$D$ in the image~$\CatBarc=\persmap(\Filt_\SComplex)$. We make use of the previous sections, and in particular we derive a bound for the dimension of the polyhedra in the fiber~$\persmap^{-1}(D)$. 

With our first result we obtain finer control on the type of strata arising in~$\CatBarc$. 
Denote by
\[\mathrm{rk}(\partial_p) := \dim \text{Im}( \partial_p:C_p(\SComplex,\field)\rightarrow C_{p-1}(\SComplex,\field))\] 
the rank of the boundary map in the simplicial chain complex, and by $\beta_p(\SComplex):= \dim \mathrm{H}_p(\SComplex,\field)$ the $p$-th Betti  of $\SComplex$.
\begin{proposition}
\label{proposition_image_persmap_bounded_number_intervals}
Let $D= (D_0, \dots , D_d)\in \CatBarc$. Then, for any homology degree~ $0\leqslant p \leqslant d$:
\begin{itemize}
\item[{\bf(i)}] The number of infinite intervals in~$D_p$ equals~$\beta_p(\SComplex)$; and
\item[{\bf(ii)}] The number of bounded intervals in~$D_p$ is smaller than or equal to~$\mathrm{rk}(\partial_{p+1})$. 
\end{itemize} 
%
%Moreover, if~$D=\persmap(f)$ for some injective filter~$f$ with values in~$(0,1)$, then~$\dim D$ is maximal and~$\dim D=\sharp K$.
%
%In particular,~$\CatBarc$ is a stratified space with finitely many strata of dimension at most $\sharp K$. 
%
\end{proposition}
\begin{proof}
Item {\bf(i)} follows from the fact that infinite bars in~$D$ correspond to the homology of~$\SComplex$.

Let~$D=\persmap(f)$ be a barcode in the image, with endpoints~$0<x_1< \cdots < x_{\dim D}< 1$. By convention~$x_0=0$ and~$x_{\dim D+1}=1$. Via the Decomposition Theorem~\cite{crawley2015decomposition}, there is an isomorphism between the $p$-th persistent homology module~$\mathrm{H}_p\circ K(f)$ and~$\oplus_{(b,d)\in D_p} \mathbb{I}_{[b,d)}$.  Through such an isomorphism  we obtain a basis of homology classes~$[c_{(b,d)}]$. The representing simplicial chains~$c_{(b,d)}$ are linearly independent in $C_p(K)= K(f) (1)$ and are cycles in~$K(f)(t)$ for~$b\leq t < d$ that become boundaries in~$K(f)(d)$. In particular, the number of intervals~$(b,d)$ ending at~$x_i$, here~$i\geq 1$, is less or equal than~$\mathrm{rk}((\partial_{p+1})_{\mid K(f)(x_i)})-\mathrm{rk}((\partial_{p+1})_{\mid K(f)(x_{i-1})})$. Summing over~$i$ yields item {\bf(ii)}.
\end{proof}

Next we determine a bound for the dimension of the polyhedral complex~$\persmap^{-1}(D)$. Recall that barcode strata in~$\CatBarc$ have maximal dimension~$\sharp \SComplex$. Therefore by the {\bf(a)} of Theorem~\ref{theorem_fiber_bundle_polyhedral} we have the obvious bound~$\dim \persmap^{-1}(D)\leq \text{codim}(\StratumD{})$, where we define the {\em codimension} of a barcode stratum~$\StratumD{}$ as:
\begin{equation*}
 \text{codim}(\StratumD{}):=\sharp \SComplex-\text{dim}(\StratumD{})\geqslant 0.
\end{equation*} 
To improve this bound, we introduce the following quantity, which can be thought of as the number of missing bounded intervals in the target barcode~$D$. 
\begin{definition}
\label{definition_deficit}
Let~$\sharp \SComplex$ be the number of simplices in~$\SComplex$, and~$\sharp D$ be the number of interval endpoints in~$D$ with finite value (counted with multiplicities). The {\em bounded deficit} of~$D$ is the quantity:
\[\frac{\sharp \SComplex-\sharp D}{2}\geqslant 0.\]
\end{definition}
Note that the bounded deficit is the same for all barcodes inside a given stratum. 
\begin{proposition}
\label{proposition_improved_bounding_dimension_fiber}
For any barcode~$D\in \CatBarc$, the dimension of the fiber of~$\persmap$ over~$D$ is less than or equal to the bounded
deficit:
%by half the difference between the number of simplices in~$\SComplex$ and that of interval endpoints in~$D$ with finite value, counted with multiplicity:
%
\[\dim \persmap^{-1}(D) \leqslant \frac{\sharp \SComplex-\sharp D}{2}\leqslant \mathrm{codim}(D).\]
\end{proposition}
It follows in particular that~$\mathrm{H}_p(\persmap^{-1}(D))=0$ for~$p> \frac{\sharp K-\sharp D}{2}$.
We apply Proposition~\ref{proposition_improved_bounding_dimension_fiber} to various fibers in the case where~$\SComplex$ is a triangle in Appendix~\ref{section_triangle}. For these fibers, the bounded deficit is almost systematically a tight upper-bound on the dimension of the fiber. 
\begin{proof}
Since~$\dim D\leqslant \sharp \SComplex$ and~$\dim D\leqslant \sharp D$, we have~$\dim D\leqslant \frac{\sharp K+ \sharp D}{2}$, therefore

\[ \mathrm{codim}(D) = \sharp \SComplex - \dim D \geq \frac{\sharp \SComplex-\sharp D}{2}.\]

So we now investigate the left inequality~$\dim \persmap^{-1}(D) \leqslant \frac{\sharp \SComplex-\sharp D}{2}$. 

Let~$\StratumF$ be a filter stratum and let  $(x_1,\cdots,x_{\dim D}):=\ChartBarc^{-1}(D)\in \mathring{\StandSimplex}^{\dim D}$, with $x_0:=0$ and~$x_{\dim D+1}:=1$. We set $\{D\}:=\{x_i\}_{i=0}^{\dim D+1}$. Let~$\simplex\in K$ be a simplex. By the equivariance Lemma~\ref{lemma_commutes_PH_homeo}, a filter $f\in \persmap^{-1}(D)\cap \StratumF$ satisfy~$f(\simplex)\in \{D\}$ if and only if all filters $f\in \persmap^{-1}(D)\cap \StratumF$ satisfy $f(\simplex)\in \{D\}$. There are at least~$\sharp D$ such simplices, since each interval endpoint in the barcode~$D$ must correspond to at least one simplex entering the filtration.

Meanwhile, there are exactly~$\dim \StratumF- \dim D$ distinct values~$x=f(\simplex)$ that are not in~$\{D\}$ for all filters~$f\in \persmap^{-1}(D)\cap \StratumF$. Each such value~$x$ is attained by at least~$2$ simplices, as otherwise~$f^{-1}(x)$ would be a singleton hence would have non-zero Euler characteristic and we would have~$x\in \{D\}$. We obtain
\[2\times (\dim \StratumF-\dim D)+  \sharp D \leqslant \sharp \SComplex.\]
From the item {\bf (a)} of Theorem~\ref{theorem_fiber_bundle_polyhedral}, the polyhedron~$\PolytopeClosed$ has dimension~$\dim \StratumF -\dim D$, and the above inequality yields $\dim \persmap^{-1}(D) \leqslant \frac{\sharp \SComplex-\sharp D}{2}$. %The result about vanishing homology groups follows by taking an arbitrary triangulation of~$\persmap^{-1}(D)$ using simplices of dimensions upper bounded by~$\frac{\sharp K-\sharp D}{2}$.
\end{proof} 
When~$D$ is the image of an injective filter~$f$, then~$\sharp  K=\dim D= \sharp  D$ and hence, by Proposition \ref{proposition_improved_bounding_dimension_fiber}, the dimension of the fiber above is zero. We thus have the following immediate consequence.
\begin{corollary}
\label{corollary_discrete_fiber_injective}
If $f\in \Filt_\SComplex$ is injective, then the fiber $\persmap^{-1}(\persmap(f))$ is a finite set.
\end{corollary}
Barcodes corresponding to injective filters are of maximal dimension.
At the other extreme we have barcodes of dimension~$1$,
i.e. barcodes~$D$ consisting simply of an infinite interval~$(x, \infty )$, possibly with multiplicity. The constant filter with value~$x$ gives rise to such a barcode. 
Although the fiber~$\persmap^{-1}(D)$ does not reduce to this constant filter, we nevertheless show that it retracts to it.  
\begin{proposition}
\label{proposition_barcode_minimal_endpoints}
A barcode of dimension~$1$ in~$\CatBarc$ has contractible fiber.
\end{proposition}
\begin{proof}
Let~$x$ be the unique endpoint value in~$D$. By assumption, there exists a filter $f$ in~$\persmap^{-1}(D)$. Note that we must have~$\min f =x$. We show that the straight line homotopy $(1-t)f+tx$ lies in~$\persmap^{-1}(D)$. For $t<1$, the filters $(1-t)f+tx$ and $f$ induce the same pre-order on simplices of~$\SComplex$, hence lie in a common stratum~$\StratumF$. By Proposition~\ref{proposition_PH_sends_strata_to_strata}, $\persmap(\StratumF)=\StratumD$ where $\StratumD$ is the stratum containing $D$, which consist in barcodes $D'$ obtained from $D$ by moving the unique endpoint value $x$ to any other value $x'$. Hence $\persmap((1-t)f+tx)$ is such a barcode, and must equal $D$ since $x'=\min (1-t)f+tx= \min f =x$. By continuity, at $t=1$, we further get that the constant filter $x$ is in the fiber. Then $\persmap^{-1}(D)$ is star-shaped around $x$.
\end{proof}
\begin{remark}
\label{remark_open_question_dimension_polyhedra_pure}
The dimension~$\dim \persmap^{-1}(D)$ being upper-bounded, we can ask if conversely the star of every polyhedron in the fiber has dimension~$\dim \persmap^{-1}(D)$. The example section, Appendix~\ref{section_triangle}, suggests that this property holds true in the case where the complex~$\SComplex$ is a manifold. In general however, two distinct filters in the fiber may have neighborhoods of distinct dimensions. Consider the following barcode~$D$
\begin{figure}[H]
\centering
\includegraphics[width=0.5\textwidth]{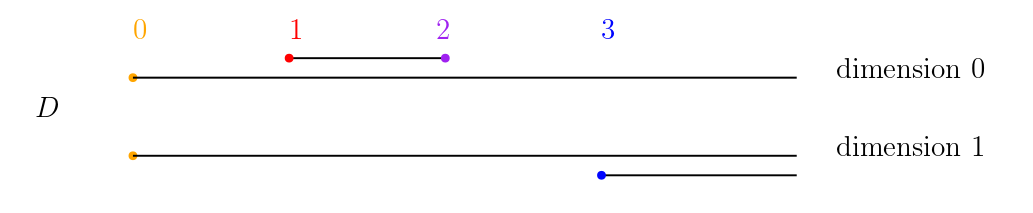}
\caption*{}
\end{figure}
where each endpoint value is given a different color.  In Fig.~\ref{fig:fiber_not_pure} below we draw the simplicial complex (left), together 
\begin{figure}[H]
\centering
\includegraphics[width=0.6\textwidth]{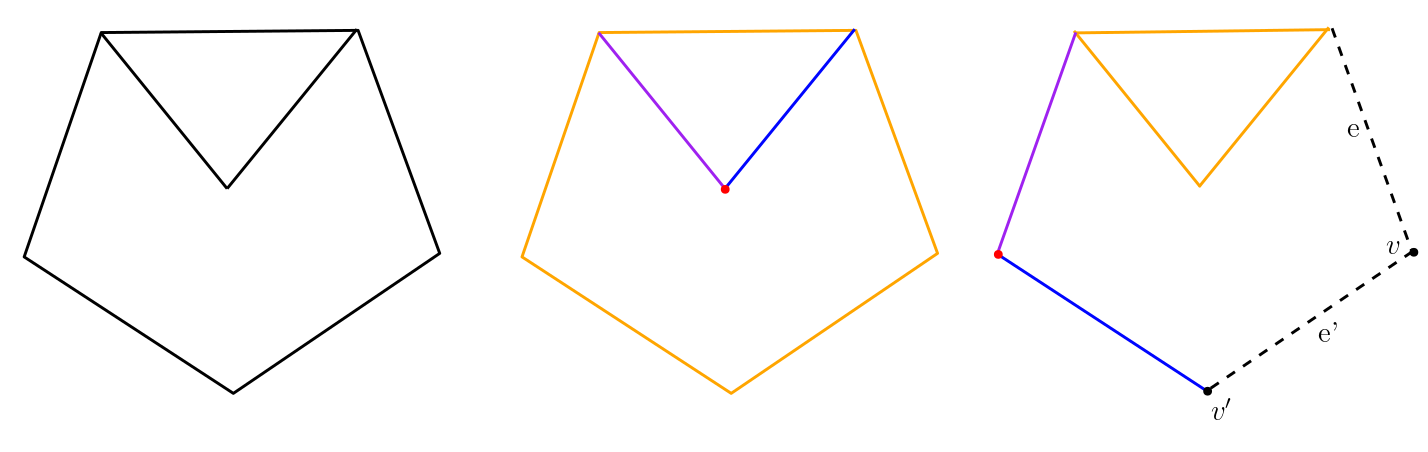}
\caption{}
\label{fig:fiber_not_pure}
\end{figure}
with two filters (middle and right) in the fiber~$\persmap^{-1}(D)$ whose values on simplices are indicated by the colors. The filter~$f$ on the right has unspecified values on vertices and edges~$v,v',e,e'$ (colored in black), which means that whatever these values are, we can modify them as long as~$f(v)=f(e)$ and~$f(v')=f(e')$ and still get a filter whose barcode is~$D$. Therefore, the star of~$f$ is $2$-dimensional. However, the filter in the middle is alone in its neighborhood: all its values are fixed to endpoints of~$D$ and infinitesimal changes of any of these values yield filters out of~$\persmap^{-1}(D)$. 
\end{remark}
\section{The barcode category and the fiber functor}
In section~\ref{sec:barcode_category}, we define morphisms between barcodes. This makes the image of the persistence map into a topological category, which is homotopy discrete (Theorem~\ref{theorem_barcode_category}). We further show that morphisms of barcodes can be deformed into particularly nice morphisms which we refer to as {\em simplicial} morphisms (Proposition~\ref{proposition_component_contains_simplicial}). Such morphisms can always be described as finite compositions of morphisms between codimension~$1$ barcodes (Proposition~\ref{proposition_simplicial_morphism_composition_codim1}). In section~\ref{sec:theorem_2}, morphisms of barcodes are pulled-back to provide maps of fibers, that are furthermore maps of polyhedral complexes up to homotopy (Proposition~\ref{proposition_monodromy}).
\subsection{The barcode category is homotopy discrete}
\label{sec:barcode_category}
We make the image~$\CatBarc$ of the persistence map into a category, so that in the next section we view the fiber~$\persmap^{-1}$ as a functor. For this we use the action of~$\NonDecMaps$ on~$\CatBarc$ as defined in Eq.~\eqref{eq_definition_action_barcodes} in order to define the morphisms between any two barcodes: 

\begin{definition}
We denote by~$\CategoryBarc$  the~$\mathbf{Top}$-enriched {\em category of barcodes} with~$D\in \Barc_\SComplex$ as objects and non-decreasing continuous maps~$\phi\in \NonDecMaps$ such that~$\phi.D=D'$ as the space of morphisms between~$D$ and~$D'$,~$\IncMaps(D,D')$.
Two morphisms~$\phi_0,\phi_1\in \IncMaps(D,D')$ are {\em homotopic as morphisms} %,~$\phi_0 \sim \phi_1$, 
if they belong to the same path connected component of~$\IncMaps(D,D')$.
\end{definition}
By definition, the sets of morphisms in a $\mathbf{Top}$-enriched
category~$\mathcal{C}$ come equipped with a topology, and so taking path connected components yields the associated {\em homotopy category} which is denoted by~$h\mathcal{C}$. The following result, whose proof is delayed to the end of the section, states that the spaces of morphisms in~$\CategoryBarc$ are made of contractible components. 
\vskip .1in
\begin{theorem}
\label{theorem_barcode_category}
For any two barcodes~$D,D'$ in~$\CatBarc$, the space of morphisms~$\IncMaps(D,D')$ has finitely many path connected components each of  which is contractible. In particular,
the category~$\CategoryBarc$ is {\em homotopy discrete}, i.e. $\CategoryBarc (D, D') \simeq h\CategoryBarc(D , D')$.
\end{theorem}
In fact, when~$D$ and~$D'$ belong to the same stratum,~$\IncMaps(D,D')$ is contractible. Indeed, a morphism~$\phi$ from~$D$ to~$D'$ is then simply a non-decreasing map that sends the $i$-th endpoint of~$D$ to the $i$-th endpoint of~$D'$. Hence, given an arbitrary $\phi_0\in \IncMaps(D,D')$, the straight line homotopy $(t,\phi)\mapsto (1-t)\phi + t \phi_0$ is a deformation retract of~$\IncMaps(D, D')$ onto the point~$\phi_0$. In general, however, there may be more than one homotopy class of morphisms in~$\IncMaps(D, D')$.

%For the rest of this section we fix the barcodes~$D,D'$ and 
Before we embark on the proof of this,  we explore some other properties of the morphism space~$\IncMaps(D,D')$. 
As before, we use the coordinate charts~$\ChartBarc:  \mathring{\StandSimplex}^{\dim D} \overset\cong \longrightarrow \StratumD_D $ of Eq.~\eqref{eq_coord_barcode} in order to define the consecutive endpoint values $  (x_1,\cdots, x_{\dim D})=\ChartBarc^{-1}(D)$ and $ (x'_1,\cdots, x'_{\dim D'})=\ChartBarc^{-1}(D')$. By convention, we also set~$x_0=x'_0=0$ and~$x_{\dim D+1}=x_{\dim D'+1}=1$.

We denote by~$\I(D)$ the simplicial complex obtained from subdividing the unit interval by the~$\dim D$ endpoints~$0<x_1<\cdots<x_{\dim D}<1$. The morphisms from~$D$ to~$D'$ that send endpoints to endpoints piecewise linearly are of particular interest to us.
\begin{definition}
\label{definition_simplicial_morphism}
A morphism~$\phi\in \IncMaps(D,D')$ is {\em simplicial} if it is induced from the geometric realisation of a simplicial map from~$\I(D)$ to~$\I(D')$. 
\end{definition}
\begin{proposition}
\label{proposition_component_contains_simplicial}
For any two barcodes $D,D'\in \CatBarc$, each connected component of~$\IncMaps(D,D')$ contains at least one simplicial map.
\end{proposition}
\begin{proof}
Let~$\Omega$ be a path connected component of~$\CategoryBarc(D,D')$ and~$\phi\in \Omega$. Let~$\psi$ be the piecewise linear extension of the map sending each~$x_i$ to the unique endpoint~$x'_j$ satisfying $x'_j\leqslant \phi(x_i)<x'_{j+1}$. The morphism~$\psi$ is then simplicial, and it is homotopic to~$\phi$ through a straight-line homotopy.
%We can modify~$\phi$ by sending each~$x_i$ to the unique endpoint~$x'_j$ satisfying $x'_j\leqslant \phi(x_i)<x'_{j+1}$.%\footnote{Note that the morphism $\phi$ is then a maximal element for the pre-order $\nearrow$.} 
%
\end{proof}
Therefore, a morphism of barcodes is (up to homotopy) a simplicial map over the unit interval. We next show that simplicial morphisms are finite compositions of simplicial morphisms between barcodes that differ by one dimension. Note that this implies that the morphisms in~$h\CategoryBarc$ are generated by morphisms between barcodes that differ by one dimension.
\begin{proposition}
\label{proposition_simplicial_morphism_composition_codim1}
Let~$\phi\in \CategoryBarc(D,D')$ be a simplicial morphism. There exists a finite sequence of barcodes
\[D=:D_1, D_2, \cdots , D_k:=D'\]
satisfying~$0\leqslant \dim D_i-\dim D_{i+1}\leqslant 1$ for~$1\leqslant i \leqslant k-1$, together with simplicial morphisms $\phi_i\in \CategoryBarc(D_i,D_{i+1})$ between them, such that
\[\phi= \phi_k \circ \phi_{k-1}\circ \cdots \circ \phi_2 \circ \phi_1.\]
\end{proposition}
\begin{proof}
We proceed by induction on~$\dim D-\dim D'$. If $\dim D-\dim D'\in \{0,1\}$, the statement is trivial. So we assume that~$\dim D-\dim D'\geqslant 2$. We first treat the case where~$\phi$ sends the first endpoint~$x_1$ of~$D$ to~$0$. Then, we may write~$\phi$ as a composition~$\phi'\circ \phi_1$ where~$\phi_1$ is the map that collapses the interval~$[0,x_1]$ to~$0$ and sends~$[x_1,1]$ to~$[0,1]$ linearly. Clearly then,~$D_1:= \phi_1(D)$ satisfies~$\dim D-\dim D_1=1$, and the morphisms~$\phi_1\in \CategoryBarc(D,D_1)$ and~$\phi'\in \CategoryBarc(D_1,D')$ are simplicial. After repeating this operation as many times as necessary, we may assume that~$\phi$ does not send any endpoint~$x_i$ to~$0$, for $1\leq i \leq \dim D$. Symmetrically, we may assume that~$\phi$ does not send any endpoint~$x_i$ to~$1$, for $1\leq i \leq \dim D$.

The morphism~$\phi$ being simplicial, there is an endpoint~$x'_j$, here~$1\leqslant j \leqslant \dim D'$, whose pre-image by~$\phi$ is a sequence of at least two consecutive endpoints~$x_i$ of~$D$. We choose an endpoint~$x_i\in \phi^{-1}(x_j)$ for which there exists a non-trivial interval~$(x_i,d)$ (or~$(b,x_i)$) in~$D$ such that~$\phi(d)\neq x'_j$ (or~$\phi(b)\neq x'_j$). Such an endpoint must exist because~$x'_j$ is the endpoint of a non-trivial interval in~$D'$, which is the image by~$\phi$ of an interval in~$D$. We may assume that~$x_{i+1}\in \phi^{-1}(x_j)$, as otherwise~$x_{i-1}\in \phi^{-1}(x_j)$ and the rest of the proof can be conducted similarly. So~$\phi(x_i)=\phi(x_{i+1})=x'_j$, and in fact $\phi([x_i,x_{i+1}])=x'_j$ since~$\phi$ is non-decreasing. We may thus factor~$\phi$ as $\phi'\circ \phi_1$, where~$\phi_1\in \NonDecMaps$ is the map that collapses the interval~$[x_i,x_{i+1}]$ onto~$x'_j$, extended linearly on~$\I$. The image~$D_1:= \phi_1(D)$ then satisfies~$\dim D-\dim D_1=1$, since~$\phi_1((x_i,d))=(x'_j,\phi(d))$ is a non-trivial interval in~$D_1$ and~$\phi$ acts injectively on the endpoints~$x_k$, for $k\notin \{i,i+1\}$. Finally, the morphisms $\phi_1\in \CategoryBarc(D,D_1)$ and~$\phi'\in \CategoryBarc(D_1,D')$ are simplicial, which concludes the proof.
\end{proof}
We return to the study of the homotopy type of~$\IncMaps(D,D')$. In order to characterise when any two morphisms in~$\IncMaps(D,D')$ are homotopic, the following structure will be instrumental:
\begin{definition}
\label{definition_index}
Let $\phi\in \IncMaps(D,D')$. The {\em index} of~$\phi$ is the collection~$\phinail := \{\phinail (j)\}_{j=0}^{\dim D'+1}$, where
\begin{equation*}
\phinail(j):= \big\{ 1\leqslant i \leqslant \dim D \, | \, \, \phi(x_i)=x'_j \big\}.
\end{equation*}
Given~$\phi_0,\phi_1\in \IncMaps(D,D')$, we write~$\phi_0\nearrow \phi_1$ whenever
\[ \forall 0\leqslant j \leqslant \dim D'+1, \, \, \phinail_0 (j)\subseteq \phinail_1 (j).\]
\end{definition}
\begin{figure}[H]
\centering
\includegraphics[width=0.8\textwidth]{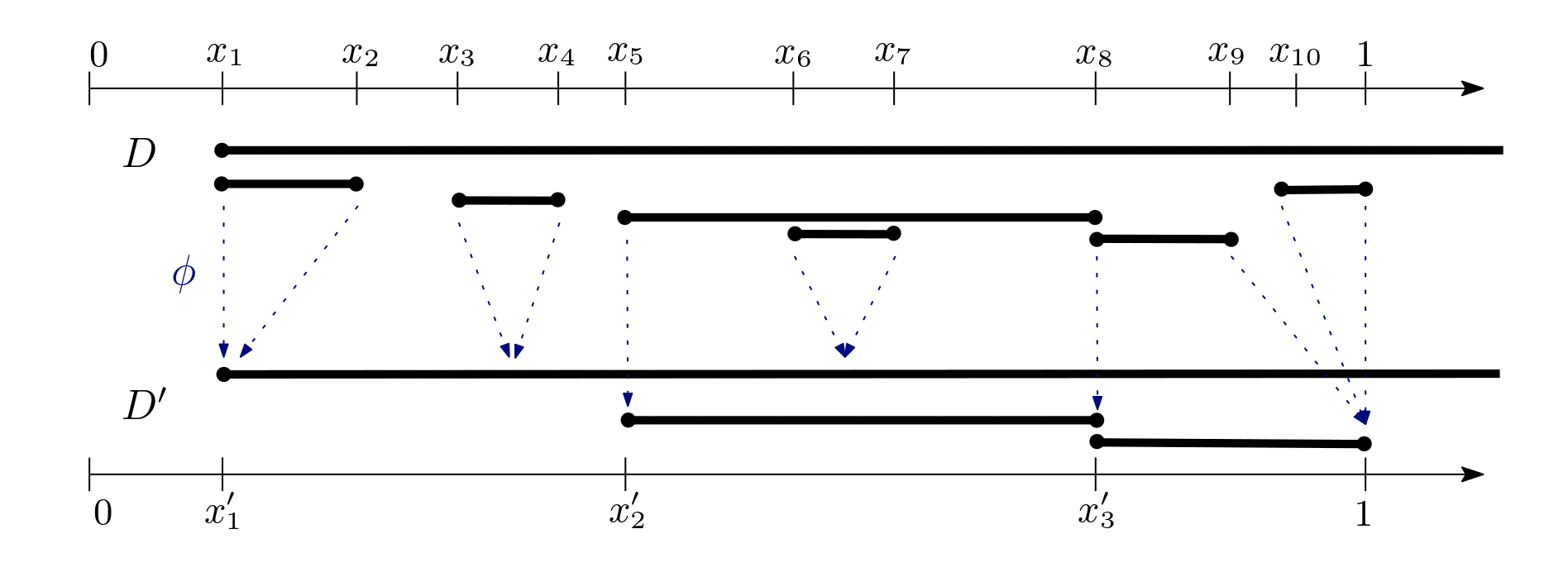}
\caption{The map $\phi : D \to D'$ has index $\phinail$ consisting of $\phinail (0) =\emptyset$, 
$\phinail (1) =\{1,2\}$,  $\phinail (2)=\{5\}$, $\phinail (3) =\{8\}$ and $\phinail (4)=\{9,{10}\}$.}
\label{fig:index}
\end{figure}
\begin{lemma}
\label{lemma_homotopic_zigzag_index}
Let $\phi_0,\phi_1 \in \IncMaps(D,D')$. Then the following are equivalent:
\begin{itemize}
\item[(i)]  The two morphisms are homotopic, i.e. $\phi_0\sim \phi_1$ as morphisms.
\item[(ii)] For all $(b,d)\in D$ and $(b',d')\in D'$, we have $(\phi_0(b),\phi_0(d))=(b',d')$ if and only if $(\phi_1(b),\phi_1(d))=(b',d')$.
%
%\item $\phi_0(b)=\phi_0(d)$ if and only if $\phi_1(b)=\phi_1(d)$. 
%
\item[(iii)] The straight line interpolation $t \mapsto t\phi_0 +(1-t)\phi_1$ is a path in~$\IncMaps(D,D')$.
\item[(iv)] There is some $\phi\in\IncMaps(D,D')$ such that $\phi \nearrow \phi_0$ and $\phi \nearrow \phi_1$. 
\end{itemize}
In particular, if~$\phi_0 \nearrow \phi_1$, then~$\phi_0$ and~$\phi_1$ are homotopic. 
\end{lemma}
\begin{proof}
$[(i)\Rightarrow (ii)]$:
Let~$\phi_t$ be a path in~$\CategoryBarc(D,D')$ joining $\phi_0$ and~$\phi_1$. Given~$(b,d)\in D$ and~$(b',d')\in D'$, let
\[\I_{b,d}^{b',d'}:=\big\{t\in \I, \, \phi_t(b,d)=(\phi_t(b),\phi_t(d))=(b',d') \big\} \subseteq \I.\]
The sets $\I_{b,d}^{b',d'}$ are closed in~$\I$. They are also open since the map $\sum_{[(b,d),(b',d')]\in D\times D' } \mathbb{1}_{\I_{b,d}^{b',d'}}: \I \rightarrow \N$
is constant and equals the number of intervals in~$D'$. Therefore, $\I_{b,d}^{b',d'}=\I$ or  $\I_{b,d}^{b',d'}=\emptyset$.

$[(ii)\Rightarrow (iii)]$: For $t\in \I$, $\phi_t:=t\phi_0+(1-t)\phi_1$ is non-decreasing. Then, $\phi_t(b,d)$ equals a non-trivial interval $(b',d')\in D'$ if and only if $\phi_0(b,d)=\phi_1(b,d)=(b',d')$. All the other intervals $(b,d)\in D$ must then be trivialized, i.e. $\phi_0(b)=\phi_0(d)$ and $\phi_1(b)=\phi_1(d)$, so that $\phi_t(b)=\phi_t(d)$. This ensures that $\phi_t\in \CategoryBarc(D,D')$ since for instance $\phi_0\in \CategoryBarc(D,D')$. 

$[(iii)\Rightarrow (i)]$: This implication is immediate. From now on, we have $(i)\Leftrightarrow (ii)\Leftrightarrow (iii)$.

$[(iii)\Rightarrow (iv)]$: Let $\phi_t$ be the straight line interpolation between~$\phi_0$ and~$\phi_1$. Let  $0\leqslant j \leqslant \dim D'+1$. Let~$1\leqslant i \leqslant \dim D$ be such that~$i\notin \phinail_0(j)$. Without loss of generality, we assume that~$\phi_0(x_i)<x'_j$. Since $x'_j$ is the endpoint of a non-trivial interval of $D'$ and $\phi_0(D)=D'$, there must exist an endpoint $x_{i'}$, with $i'>i$, of an interval in~$D$ such that $\phi_0(x_{i'})=x'_j$. By the item~$(ii)$, we also have $\phi_1(x_{i'})=x'_j$. In turn, $\phi_1(x_i)\leqslant x'_j$ as~$\phi_1$ is non-decreasing. Therefore $\phi_{\frac{1}{2}}(x_i)< x'_j$, and in particular~$i\notin \phinail_{\frac{1}{2}}(j)$. We have therefore proved that $\phinail_{\frac{1}{2}}(j) \subseteq \phinail_{0}(j)$, so that~$\phi_{\frac{1}{2}} \nearrow \phi_0$. Similarly,~$\phi_{\frac{1}{2}} \nearrow \phi_1$.

$[(iv)\Rightarrow (i)]$: It is enough to show that if~$\phi \nearrow \phi_0$, then~$\phi$ and~$\phi_0$ are homotopic, as we then show in the exact same way that~$\phi \sim \phi_1$, which implies~$\phi_0 \sim \phi_1$. Let~$(b,d)\in D$ be a bounded interval (the case where $d=d'=\infty$ is dealt with similarly) such that~$\phi(b,d)=(b',d')$ for some~$(b',d')\in D'$. Then there are indices~$i<i'$ and~$j<j'$ such that $(b,d)=(x_i,x_{i'})$ and $(b',d')=(x'_j,x'_{j'})$. Since $\phinail(j)\subseteq \phinail_0(j)$ and $\phinail(j')\subseteq \phinail_0(j')$, we have $\phi_0(b,d)=(b',d')$ as well. Therefore, the images $\phi_0(b,d)$ of intervals $(b,d)$ such that $\phi(b,d)\in D'$ cover all the intervals in~$D'$. Hence, any other interval in~$D$ is trivialized by~$\phi_0$, which guarantees that the assertion $(ii)$ holds. We are done since $(ii)\Rightarrow (iii)\Rightarrow (i)$.
\end{proof}
\begin{remark}
\label{remark_index_invariant}
Two morphisms~$\phi_0,\phi_1 \in \IncMaps(D,D')$ with the same index are homotopic. In the case where~$D$ and~$D'$ belong to strata that differ by one dimension, the converse is true as well so that in this case the index is a (complete) homotopy invariant. To see this, observe that the index~$\phinail$ associated to a morphism~$\phi$ in the case $\dim D- \dim D'=1$ cover the endpoints of~$D$ and must consist of singletons~$\phinail(j)$ except for a unique~$\phinail(k)$ which is the unique pair~$\{i,i+1\}$ of consecutive endpoints of~$D$ collapsed by~$\phi$.
Therefore, if $\phi_0,\phi_1 \in \IncMaps(D,D')$ are homotopic, by Lemma~\ref{lemma_homotopic_zigzag_index} there is a third morphism $\phi\in \IncMaps(D,D')$ such that $\phi \nearrow \phi_0$ and $\phi \nearrow \phi_1$. By the above restriction on the index of morphisms in~$\IncMaps(D,D')$, this immediately implies that~$\phi_0$ and~$\phi_1$ have the same index.  
\end{remark}
\begin{proof}[Proof of Theorem~\ref{theorem_barcode_category}]
Let $D,D'\in \CatBarc$. By the assertion $(iv)$ of Lemma~\ref{lemma_homotopic_zigzag_index}, if $\phi_0,\phi_1\in \CategoryBarc(D,D')$ are two morphisms such that $\phi_0 \nearrow \phi_1$, then~$\phi_0$ and~$\phi_1$ belong to the same path connected component of~$\CategoryBarc(D,D')$. Since there are finitely many possible indices,~$\CategoryBarc(D,D')$ has finitely many path connected components.

Let~$\Omega$ be a path connected component of~$\CategoryBarc(D,D')$. Let~$\psi \in \Omega$. By the~$(iii)$ of Lemma~\ref{lemma_homotopic_zigzag_index}, the map
\[(t,\phi)\in [0,1] \times \Omega \longmapsto t\psi +(1-t)\phi \in \Omega \]
is a deformation retraction of $\Omega$ onto~$\{\psi\}$. Consequently,~$\Omega$ is contractible.
%
%By finiteness of all possible indices~$\phinail$, we can find a minimal~$\psi$ in~$\Omega$ for the pre-order~$\nearrow$, i.e. if~$\phi\in \CategoryBarc(D,D')$ is any another morphism satisfying~$\phi \nearrow \psi$, then~$\psi \nearrow \phi$ as well. In fact, the minimality of~$\psi$ together with the assertion~$(iv)$ of Lemma~\ref{lemma_homotopic_zigzag_index} implies that for any morphism~$\phi\in \Omega$, we have~$\psi \nearrow \phi$. We then have a deformation retraction
%
%\[(t,\phi)\in [0,1] \times \Omega \longmapsto t\psi +(1-t)\phi \in \Omega \]
%
%of $\Omega$ onto~$\{\psi\}$, which is well-defined by the assertion $(iii)$ of Lemma~\ref{lemma_homotopic_zigzag_index}. Consequently,~$\Omega$ is contractible.
%
\end{proof}
\subsection{Monodromies and polyhedral maps of fibers}
\label{sec:theorem_2}
%
%Knowing that the fiber of the persistence map is homogeneous over each barcode stratum and in fact enjoys the structure of a polyhedral complex, 
We now analyse how  fibers relate to each other as we cross barcode strata. We associate to each map of barcodes a map between the corresponding fibers as follows.
\begin{definition}
\label{definition_monodromy}
Let $D\in \Barc^{d+1}$ and $\phi\in \NonDecMaps$.
The {\em monodromy} $\Monodromy_{\phi}$ associated to $\phi$ is the map between fibers:
\[\Monodromy_{\phi}: f\in \persmap^{-1}(D)\longmapsto \phi \circ f \in \persmap^{-1}(\phi.D).\]
\end{definition}
Note that the monodromy is well-defined %: if $f\in \persmap^{-1}(D)$, then
since $\persmap$ is $\NonDecMaps$-equivariant by Lemma~\ref{lemma_commutes_PH_homeo}, 
\[
\persmap(\phi\circ f)=\phi.\persmap(f),
\]
and hence $\phi\circ f \in \persmap^{-1}(\phi.D)$. Furthermore, given another~$\phi'\in \NonDecMaps$ by definition 
\[
\Monodromy_{\phi' \circ \phi} = \Monodromy_{\phi'} \circ \Monodromy_{\phi},
\]
and in particular, if~$\phi$ is invertible then~$\Monodromy_\phi$ is a homeomorphism. Furthermore,  the monodromy assignment~$\phi \mapsto \Monodromy_\phi$ is continuous. We thus see that monodromies turn the inverse image~$\persmap^{-1}$ into a functor  as follows.%\footnote{Here we are only considering $\persmap^{-1}$ as a functor of~$\mathbf{Top}$-enriched categories. In fact,~$\persmap^{-1}$ is also continuous as a map of object spaces with respect to the natural topology on~$\Barc_\SComplex$ and the Gromov-Hausdorff metric on compact spaces.}
\begin{definition}
\label{definition_fiber_functor}
The {\em fiber functor} is a functor of ${\bf Top}$-enriched categories
\[\persmap^{-1}:\CategoryBarc\longrightarrow \mathbf{Top} \]
that sends a barcode~$D$ to the fiber~$\persmap^{-1}(D)$ and a morphism~$\phi\in \IncMaps(D,D')$ to the monodromy~$\Monodromy_\phi$. 
\end{definition}
\begin{remark}
\label{remark_functor_homotopy_category}
%
%Here we are only considering~$\persmap^{-1}$ as a functor of~$\mathbf{Top}$-enriched categories. 
$\persmap^{-1}$ also descends to define a functor of homotopy categories:
\[\persmap^{-1}:h\CategoryBarc\longrightarrow h\mathbf{Top}. \]
\end{remark}
%\vskip .05in
%
When~$\phi$ is not simplicial (Definition~\ref{definition_simplicial_morphism}), it may send a bounded interval~$(b,d)$ of~$D$ "to the middle of nowhere", i.e.~$\phi(b)=\phi(d)$ may not equal an interval endpoint in~$D'$. In turn, if this is the case, the monodromy~$\Monodromy_\phi$ is not a polyhedral map, i.e. it is not an affine map on each polyhedron that sends polyhedra to polyhedra surjectively; see Fig.~\ref{fig:bar_disappear} below.
\begin{figure}[H]
\centering
\includegraphics[width=0.6\textwidth]{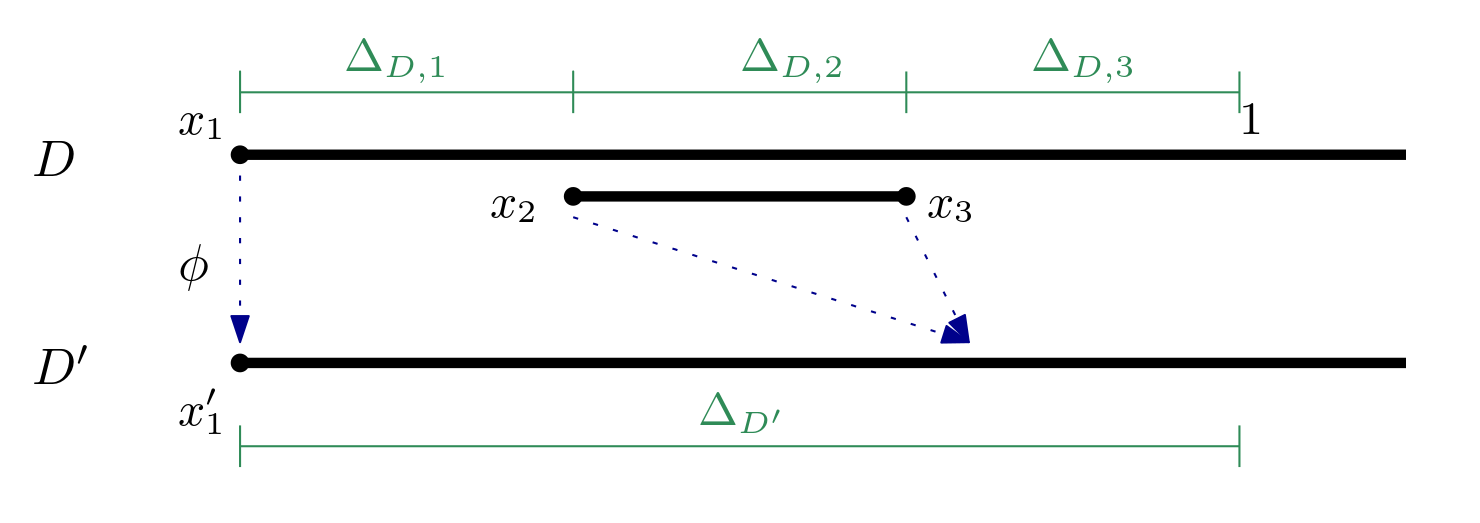}
\caption{The barcode~$D$ has dimension~$3$. By Theorem~\ref{theorem_fiber_bundle_polyhedral}, the polyhedral complex~$\persmap^{-1}(D)$ is made of polyhedra that are products $\StandSimplex_{D,1}\times \StandSimplex_{D,2}\times \StandSimplex_{D,3}$ of $3$ standard simplices. Each simplex~$\StandSimplex_{D,i}$ corresponds to the values of the filters in the fiber that are in-between the $i$-th endpoint $x_i$ and $(i+1)$-th endpoint~$x_{i+1}$ of~$D$. Likewise, a polyhedron in~$\persmap^{-1}(D')$ is simply a unique standard simplex $\StandSimplex_{D'}$. The map~$\phi$ collapses~$x_2$ and~$x_3$ strictly in-between~$x'_1$ and~$1$. In turn, the monodromy $\Monodromy_\phi$ collapses the second standard simplex~$\StandSimplex_{D,2}$ of any polyhedron in~$\persmap^{-1}(D)$ in the interior of~$\StandSimplex_{D'}$, and so is not a polyhedral map.
}
\label{fig:bar_disappear}
\end{figure}
However, simplicial maps induce polyhedral monodromies, as stated in the following result. 
\begin{proposition}
\label{proposition_monodromy}
Let $D,D'\in \CatBarc$ and $\phi\in \CategoryBarc(D,D')$. If~$\phi$ is simplicial, then  the monodromy $\Monodromy_\phi: \persmap^{-1}(D)\rightarrow \persmap^{-1}(D')$ is a polyhedral map. In particular, for any $\phi\in \CategoryBarc(D,D')$, the monodromy $\Monodromy_\phi$ is homotopic to a polyhedral map.
\end{proposition}
Given a barcode $D$, we denote by $\IncHomSpace_D$ the stabilizer of $D$, i.e. the group of homeomorphisms of the real line that fix the endpoints of $D$ and hence act trivially on~$D$. The proof of Proposition~\ref{proposition_monodromy} mainly relies on the following lemma.
\begin{lemma}
\label{lemma_alpha_beta}
Let~$\phi\in \IncMaps(D,D')$ be a simplicial morphism. Then there exists a group homomorphism
\[\phipost: \IncHomSpace_D \longmapsto \IncHomSpace_{D'},\]
such that for all $\homleft\in \IncHomSpace_D$, $\phi\circ \homleft= \phipost(\homleft) \circ \phi$. Similarly, there exists a group homomorphism
\[\phipre: \IncHomSpace_{D'} \longmapsto \IncHomSpace_{D},\]
such that for all~$\homright \in \IncHomSpace_{D'}$, $\phi\circ \phipre(\homright)= \homright \circ \phi$.
\end{lemma}
\begin{proof}
Let~$\homleft\in \Aut$ such that~$\homleft(D)=D$. Note that, since $\phi$ is simplicial, for any index $0\leqslant i\leqslant \dim D$, the following alternative holds:
\begin{itemize}
\item[(a)] Either $\phi([x_i,x_{i+1}])=x'_j$ for some index $0\leqslant j \leqslant \dim D'+1$;
\item[(b)] Or $\phi_{|[x_i,x_{i+1}]}$ is a linear bijection onto $[x'_j,x'_{j+1}]$ for some index $0\leqslant j \leqslant \dim D'$.
\end{itemize}
According to this alternative, we define $\homright:=\phipost(\homleft)$ on each $\phi([x_i,x_{i+1}])$ as follows:
\begin{itemize}
\item[(a)] Either $\phi([x_i,x_{i+1}])=x'_j$, in which case we set $\homright(x'_j):=x'_j$;
\item[(b)] Or $\phi_{|[x_{i},x_{i+1}]}$ is a linear bijection onto $[x'_j,x'_{j+1}]$, in which case we set $\homright_{|[x'_j,x'_{j+1}]}:=  \phi \circ\homleft\circ \phi_{|[x_{j},x_{j+1}]}^{-1}$.
\end{itemize}
Note that in case (b), $\homright$ is well-defined since $\homleft(D)=D$ implies that $\homleft$ restricts to a homeomorphism of each line segment $[x_i,x_{i+1}]$. Moreover, $\homright$ is defined on the whole unit interval $\I$ since $\phi$ is surjective. It is then clear that we have $\homright \circ \phi \circ \homleft=\phi$ on each $[x_i,x_{i+1}]$, so that the equality holds on $\I$ as desired. Note that $\homright \in \IncHomSpace_{D'}$ since 
\[\homright(D')=\homright(\phi\circ \homleft(D))=\phi.D=D'.\]
By construction, the association $\phipost: \homleft \mapsto \homright$ is a group homomorphism. Conversely, if we are rather given a map $\homright\in \Aut$ such that $\homright(D')=D'$, then we can construct the map $\homleft:= \phipre(\homright)$ satisfying $\phi \circ \homleft =\homright \circ \phi$ as follows:
\begin{itemize}
\item[(a)] Either $\phi([x_i,x_{i+1}])=x'_j$, in which case we set $\homleft([x_i,x_{i+1}]):=\mathrm{Id}_{|[x_i,x_{i+1}]}$;
\item[(b)] Or $\phi_{|[x_{i},x_{i+1}]}$ is a linear bijection onto $[x'_j,x'_{j+1}]$, in which case we set $\homleft_{|[x_i,x_{i+1}]}:=  \phi_{|[x_i,x_{i+1}]}^{-1} \circ\homright\circ \phi$.
\end{itemize}
This construction also yields a group homomorphism $\phipre: \homright \mapsto \homleft$.
\end{proof}

\begin{proof}[Proof of Proposition~\ref{proposition_monodromy}]
The second part of the statement follows directly from Proposition~\ref{proposition_component_contains_simplicial}, which states that any morphism of barcodes is homotopic to a simplicial map. Henceforth, we fix a simplicial morphism~$\phi\in \IncMaps(D,D')$ and show that the monodromy~$\Monodromy_\phi$ is a polyhedral map, i.e. it is affine on each polyhedron and sends polyhedra to polyhedra surjectively. 

Let $\StratumF\subseteq \Filt_\SComplex$ be a filter stratum. We have $\PolytopeClosed\cong \StandSimplex_0 \times \cdots \times \StandSimplex_{\dim D}$, where the standard simplex~$\StandSimplex_i$ corresponds to filter values that are in-between the endpoints~$x_{i}$ and~$x_{i+1}$
of~$D$. Since~$\phi$ is affine over~$[x_i,x_{i+1}]$, each coordinate function $\Monodromy_{\phi,\simplex}: f\in \PolytopeClosed \mapsto \Monodromy_\phi(f)(\simplex)\in \R$, for $\simplex\in \SComplex$, is affine as well. So the restriction of $\Monodromy_\phi$ to $\PolytopeClosed$ is an affine map, as desired.

It remains to show that the image $\Monodromy_\phi(\PolytopeClosed)$ equals a polyhedron $\overbar{\persmap_{|\StratumF'}^{-1}}(D')$. We fix a filter $f\in \PolytopeOpen$ and denote by~$\StratumF'$ the stratum containing~$\phi\circ f$. Note that, if~$g\in \Filt_\SComplex$ is a filter, then:
\begin{align}
g\in  \Monodromy_\phi(\PolytopeOpen)&\Longleftrightarrow \exists f'\in \PolytopeOpen, \, g=\phi\circ f'  \nonumber\\
 &\Longleftrightarrow  \exists \homleft \in \IncHomSpace_D, \, g=\phi\circ (\homleft \circ f)  \,  \tag*{by the equivariance Lemma~\ref{lemma_commutes_PH_homeo} }\\
  &\Longleftrightarrow  \exists \homright \in \IncHomSpace_{D'}, \, g=\homright \circ (\phi \circ f)  \,  \tag*{by Lemma~\ref{lemma_alpha_beta}}\\
  &\Longleftrightarrow   g\in \persmap_{\StratumF'}^{-1}(D'). \,  \tag*{by the equivariance Lemma~\ref{lemma_commutes_PH_homeo}}
\end{align}
%
\begin{comment}
Note that
%
\begin{align}
\Monodromy_\phi(\PolytopeOpen) &= \Monodromy_{\phi}(\IncHomSpace_D.f) \tag*{ by Lemma~\ref{lemma_commutes_PH_homeo}}\\
 &= [\phipost(\IncHomSpace_D)]^{-1}.\Monodromy_{\phi}(f)\subseteq  \persmap_{|\StratumF'}^{-1}(D'), \tag*{by Lemma~\ref{lemma_alpha_beta} }
\end{align}
%
\UT{I think it would be clearer to spell things out a bit and work with elements in  $\IncHomSpace_D $. Also, I don't think $[\phipost(\IncHomSpace_D)]^{-1}$ makes sense.}
and that conversely
%
\begin{align}
\persmap_{|\StratumF'}^{-1}(D') &= \IncHomSpace_{D'}. \Monodromy_{\phi}(f) \tag*{ by Lemma~\ref{lemma_commutes_PH_homeo}}\\
 &= [\Monodromy_{\phi}(\phipre(\IncHomSpace_{D'})]^{-1}f) \subseteq  \Monodromy_{\phi}(\persmap_{|\StratumF'}^{-1}(D)). \tag*{by Lemma~\ref{lemma_alpha_beta} }
\end{align}
\end{comment}
%
Therefore, $\Monodromy_\phi(\PolytopeOpen)$ equals $\persmap_{|\StratumF'}^{-1}(D')$ and in fact $\Monodromy_\phi(\PolytopeClosed)$ equals $\overbar{\persmap_{|\StratumF'}^{-1}}(D')$ since the image of a closed polyhedron via an affine map is again a closed polyhedron.
\end{proof}
The following is a consequence of the Propositions~\ref{proposition_component_contains_simplicial},~\ref{proposition_simplicial_morphism_composition_codim1} and~\ref{proposition_monodromy}:
\begin{corollary}
\label{corollary_monodromoy_homotopic_polyhedral_and_composition}
For any $\phi\in \CategoryBarc(D,D')$, the monodromy~$\Monodromy_\phi$ is homotopic to a polyhedral map. This polyhedral map may further be chosen as a composition 
\[\Monodromy_{\phi_k\circ \cdots \circ \phi_1}= \Monodromy_{\phi_k}\circ \cdots \circ \Monodromy_{\phi_1}\]
of monodromies~$\Monodromy_{\phi_k}$ that are polyhedral maps between fibers over barcodes that differ by one dimension.
\end{corollary}
\begin{remark}
\label{remark_monodromy_projection}
The monodromy associated to a simplicial morphism~$\phi$ acts as a projection map on each polyhedron~$\PolytopeClosed$ of the fiber. Indeed, recalling that~$\PolytopeClosed$ is (isomorphic to) a product $\StandSimplex_0 \times \cdots \times \StandSimplex_{\dim D}$ of standard simplices, we have the commutative diagram
\begin{center}
\begin{tikzpicture}[baseline=(current  bounding  box.north)]
\node[] (a) at (10,-3) {$\Monodromy_\phi(\PolytopeClosed)$};
\node[] (b) at (0,-1) {$\StandSimplex_0 \times \cdots \times \StandSimplex_{\dim D} $};

\node[] (e) at (10,-1) {$\StandSimplex_0 \times \cdots \times \StandSimplex_{\dim D}/ \prod_{\phi(x_i)=\phi(x_{i+1})} \StandSimplex_i$};

\node[] (c) at (0,-3) {$\PolytopeClosed$};

\draw[->] (b)--(e) node[midway, above] {$\pi$};

\draw[->] (c)--(a) node[midway, above] {$\Monodromy_{\phi}$};
\draw[->] (e)--(a) node[midway, right] {$\cong$};
\draw[->] (b)--(c) node[midway, left] {$\cong$};

\end{tikzpicture}
\end{center}
where $\pi$ is the projection map. In other words, the product of simplices describing the image polyhedron~$\Monodromy_\phi(\PolytopeClosed)$ is obtained from the product describing~$\PolytopeClosed$ by collapsing standard simplices~$\StandSimplex_i$ whenever~$\phi$ collapses the~$i$-th and~$i+1$-th endpoint of~$D$.
\end{remark}
\newpage

\section{The space of barcodes is homotopically stratified}
\label{section_stratification_barcodes_and_entrance_path_category}

In this section, we take a closer look at the stratification of barcodes and show that~$\Barc_K$ is homotopically stratified. This naturally leads us to introducing the entrance path category of barcodes. We observe that the entrance path category is isomorphic to the homotopy category,~$h\CategoryBarc$, of barcodes. Recall from Theorem~\ref{theorem_barcode_category} that~$\CategoryBarc$ is homotopy discrete. Similarly, we find that the space of entrance paths between fixed barcodes is contractible.
\subsection{Regularity of the stratification of barcodes}
\label{section_barcode_stratification_regularity}

The notion of stratification used so far (Definition~\ref{definition_stratification}) is a very weak one, as it does not impose restrictions on the neighborhoods of strata. The space of filters~$\Filt_\SComplex$ is in fact Whitney stratified, being a polyhedron in~$\I^\SComplex\subseteq \R^{\SComplex}$. However,~$\CatBarc$ is the quotient of~$\Filt_\SComplex$ induced by~$\persmap$ (see Proposition~\ref{prop_quotient_topology}), so the regularity of its stratification is less apparent. We show that the space~$\CatBarc$ of barcodes in the image of the persistence map is homotopically stratified in the sense of Quinn~\citep{quinn1988homotopically}. Stratified coverings over homotopically stratified spaces are classified by the entrance path category, which we introduce and analyse in the next section in the case of barcodes. 

The local neighborhoods in a homotopically stratified space~$\TopSpace$ are defined in terms of paths that cross strata in decreasing order of dimension:
%We choose to consider homotopically stratified spaces for two main reasons. On the one hand, the resulting category remains very general as it contains Siebenmann locally cone-like stratified spaces, Thom-Mather stratified spaces, Goresky-MacPherson topologically stratified pseudomanifolds and Whitney stratified spaces. On the other hand, stratified coverings over homotopically stratified spaces are  classified by the entrance path category, which we introduce and analyse in the next section in the case of barcodes. The local neighborhoods in a homotopically stratified space~$\TopSpace$ are defined in terms of paths that cross strata in decreasing order of dimension. 
%
\begin{definition}
\label{definition_entrance_path}
Let~$\TopSpace$ be a stratified space. A continuous path~$\gamma\in \TopSpace^{I}$ is an {\em entrance path} if for any $0\leqslant t\leqslant t' \leqslant 1$, the stratum containing~$\gamma(t)$ has greater or equal dimension than that containing~$\gamma(t')$.
\end{definition}
\begin{definition}
\label{definition_homotopy_link}
Let~$\TopSpace$ be a stratified space. An entrance path~$\gamma$ is {\em elementary} if it stays in a unique stratum until the very last moment, that is if~$\gamma([0,1))$ belongs to a fixed stratum. Given two strata~$\TopSpace^i$ and~$\TopSpace^j$, $j<i$, the {\em homotopy link}~$\Holink(\TopSpace^i,\TopSpace^j)$ is the space of elementary paths starting in~$\TopSpace^i$ and ending in~$\TopSpace^j$ with the compact open topology. 
\end{definition}
\begin{definition}
\label{definition_homotopically_stratified}
A stratified space~$\TopSpace$ is {\em homotopically stratified} if it satisfies the following conditions for any pair of strata~$\TopSpace^i$ and~$\TopSpace^j$, where $j<i$:
\begin{enumerate}
\item The inclusion $\TopSpace^j \hookrightarrow \TopSpace^i\cup \TopSpace^j$ is {\em tame}, which means that there is a strong deformation retraction  of a neighborhood of~$\TopSpace^j$ in~$\TopSpace^i\cup \TopSpace^j$ onto~$\TopSpace^j$ such that points remain in the same stratum until the very last moment during the deformation; 
\item The evaluation at time $t=1$
\[\mathrm{ev}_1: \gamma \in \Holink(\TopSpace^i,\TopSpace^j) \longmapsto \gamma(1) \in \TopSpace^j \]
is a fibration.
\end{enumerate}
\end{definition}
Note that, in the original formulation of homotopically stratified spaces~\cite{quinn1988homotopically}, the strata are not necessarily topological manifolds, and so we should really refer to the spaces of Definition~\ref{definition_homotopically_stratified} as manifold stratified spaces, as done in~\cite{weinberger1994topological} for instance. However, this distinction is irrelevant for our purposes, since the strata in~$\CatBarc=\persmap(\Filt_\SComplex)$ are manifolds. 
\begin{proposition}
\label{proposition_barcodes_homotopically_stratified}
The filtered space~$\CatBarc$ is homotopically stratified. 
\end{proposition}
\begin{proof}
Let~$\StratumD$ and~$\StratumD'$ be two barcode strata with~$\StratumD' \subseteq \bar{\StratumD}$.  Recall that the coordinate chart~$\ChartBarc$ extends to a continuous, surjective, stratum-preserving map~$\ChartBarc:  \StandSimplex^{\dim \StratumD}\rightarrow \bar{\StratumD}  $; see  Eq.~\eqref{eq_extension_mu_inverse}. Then, the inverse image~$\ChartBarc^{-1}(\StratumD')$ of~$\StratumD'$ is a union of faces in~$\StandSimplex^{\dim \StratumD}$. Any collection of faces has a neighborhood in $\StandSimplex^{\dim \StratumD}$ that deformation retracts back to itself such that points in the interior of the simplex are mapped to points in the interior but at the last moment.
%Thus the inclusion~$\ChartBarc^{-1}(\StratumD')\hookrightarrow \ChartBarc^{-1}(\StratumD')\cup \mathring{\StandSimplex}^{\dim \StratumD}$ is tame.
Pulling back neighborhoods and composing the deformation retraction with~$\ChartBarc$, we get the required  deformation retraction of a neighborhood of $\StratumD'$ in~$\StratumD \cup \StratumD'$ onto~$\StratumD'$. Therefore, the inclusion $\StratumD' \hookrightarrow \StratumD\cup \StratumD'$ is tame.   

To check that the evaluation map~$\mathrm{ev}_1: \gamma \in \Holink(\StratumD,\StratumD') \mapsto \gamma(1) \in \StratumD'$ is a fibration, it is enough (see e.g.~\cite{hurewicz1955concept}) to find a section for the map:
\[(\tilde{\gamma}_h)_{h\in [0,1]} \in \Holink(\StratumD,\StratumD')^{[0,1]} \longmapsto \big((\mathrm{ev}_1\circ \tilde{\gamma}_h)_{h\in [0,1]}, \tilde{\gamma}_0\big)\in \StratumD'^{[0,1]} \times_{\StratumD'} \Holink(\StratumD,\StratumD'), \]
where~$\StratumD'^{[0,1]} \times_{\StratumD'} \Holink(\StratumD,\StratumD')$ is the fiber product
\[\StratumD'^{[0,1]} \times_{\StratumD'} \Holink(\StratumD,\StratumD'):= \big\{(\gamma, \tilde{\gamma}) \in \StratumD'^{[0,1]} \times \Holink(\StratumD,\StratumD')\,  |  \, \gamma(0)=\tilde{\gamma}(1)\big\}.\]
Using the coordinate chart~$\ChartBarc: \mathring{\StandSimplex}^{\dim \StratumD'}\overset\cong \longrightarrow \StratumD'$, we may view a path~$\gamma: h\in [0,1] \mapsto \gamma(h)\in \StratumD'$ via its coordinates $0<x_1^\gamma(h)<\cdots <x_{\dim \StratumD'}^\gamma(h)<1$. Given~$0\leqslant h \leqslant 1$, let~$\phi^{\gamma}_h\in \IncHomSpace$ be the map that sends the endpoint values~$x_i^{\gamma}(0)$ to~$x_i^{\gamma}(h)$, and is extended linearly on each line segment $[x_i^{\gamma}(0),x_{i+1}^{\gamma}(0)]$. Clearly, the association~$(\gamma,h) \mapsto \phi^{\gamma}_h$ is continuous, and we have $\gamma(h)=\phi_h^{\gamma}.\gamma(0)$. Then, the map
\[(\gamma, \tilde{\gamma}) \in \StratumD'^{[0,1]} \times_{\StratumD'} \Holink(\StratumD,\StratumD') \longmapsto (h \mapsto \phi_h^{\gamma}. \tilde{\gamma})\in  \Holink(\StratumD,\StratumD')^{[0,1]}\]
is the desired section. 
\end{proof}

\subsection{The entrance path category of barcodes}
\label{section_entrance_path_category_barcodes}

The entrance path category is the suitable generalisation of the fundamental groupoïd for stratified spaces where ordinary paths are replaced by entrance paths between points; see~\cite{treumann2009exit} for the original constructions.

%We show that the homotopy category of the barcode category studied in the previous section is isomorphic to the entrance category of the associated barcode space. The fiber functor is thus seen to give rise to a functor on the entrance category.
%
\begin{definition}
\label{definition_entrance_path_category}
Let~$\TopSpace$ be a stratified space. The {\em entrance path category} of~$\TopSpace$ is $\Ent(\TopSpace):=h\Path_\leq(X)$, where~$\Path_\leq(X)$ is the topologically enriched category with~$\TopSpace$ as the set of objects and spaces of entrance paths equipped with the compact open topology as morphisms. 
In other words,~$\Ent(\TopSpace)$ has the points of~$\TopSpace$ as objects and the homotopy classes of entrance paths as morphisms.
\end{definition}

\begin{remark}
In order to make~$\Path_\leq(X)$ into a category where concatenation of paths defines a strictly associative composition of morphisms, paths of all positive lengths need to be allowed. This is analogous to replacing loop spaces by Moore loop spaces and the resulting morphism spaces are homotopy equivalent. In particular the definition of $\Ent(\TopSpace)$ is not affected. We will ignore this subtlety in what follows.
%Note that, strictly speaking,~$\Path_<(X)$ is not a category as concatenation is not associative. We do not worry about this subtlety because, up to homotopy, one can always replace the space of unit length paths with that of Moore paths, for which composition is associative.
\end{remark}

%In this section, we embed the category~$h\CategoryBarc$ into~$\Ent(\CatBarc)=h\Path_<(\CatBarc)$. 
%For this, we view a morphism

Let~$\phi$ in~$\CategoryBarc$ be a morphism  between barcodes~$D$ and~$D'$. Define the path:
\[\gamma_\phi(t):=(t\phi+ (1-t)\mathrm{Id}).D.\]
For times~$t<1$, $t\phi+ (1-t)\mathrm{Id}$ is a homeomorphism of the unit interval and so~$\gamma_\phi(t)$ stays in the stratum containing~$D$. Hence,~$\gamma_\phi(t)$ is in fact an elementary entrance path. It is then clear that the association 
\[\phi \in \CategoryBarc(D,D')\longmapsto \gamma_\phi \in \Path_\leq(\CatBarc)(D,D') \]
is continuous. %Although this correspondance is clearly not surjective, i.e. there are much more entrance paths than there are morphisms~$\phi$, we show that it yields a functorial isomorphism between the categories~$h\CategoryBarc$ and~$\Ent(\CatBarc)$. 
\begin{proposition}
\label{proposition_entrance_path_category_equal_catbarc}
The association~$[\phi]\mapsto [\gamma_\phi]$ is functorial and induces an isomorphism of categories:
\[h\CategoryBarc\cong \Ent(\CatBarc).\]
\end{proposition} 

Before proving this, we first characterize when elementary entrance paths are homotopic, in a similar fashion to Lemma~\ref{lemma_homotopic_zigzag_index} for morphisms of barcodes. Recall that an entrance path~$\gamma$ from~$D$ to~$D'$ is elementary if it stays in a unique stratum until the very last moment, that is if~$\gamma([0,1))$ belongs to the stratum~$\StratumD_D$. Using the coordinate chart~$\ChartBarc:  \mathring{\StandSimplex}^{\dim D} \overset\cong \longrightarrow \StratumD_D $, we may view~$\gamma_{|[0,1)}$ via its coordinates $0<x_1^{\gamma}(t)<\cdots <x_{\dim D }^{\gamma}(t)<1$, $0\leqslant t <1$. Alternatively, for each interval~$(b,d)\in D$, there is a continuously evolving interval~$(b^\gamma(t),d^\gamma(t))$ starting at~$(b,d)$, $0\leqslant t <1$, and together the intervals~$(b^\gamma(t),d^\gamma(t))$ form the barcode~$\gamma(t)$. 
\begin{lemma}
\label{lemma_homotopic_entrance_paths}
Let $\gamma_0,\gamma_1 \in \Path_\leq(\CatBarc)(D,D')$ be elementary entrance paths. Then the following are equivalent:
\begin{itemize}
\item[(i)]  The two entrance paths are homotopic %,i.e. $\gamma_0\sim \gamma_1$, 
through elementary entrance paths.
\item[(ii)] For all~$(b,d)\in D$ and~$(b',d')\in D'$, we have $\lim_{t\rightarrow 1^-}(b^{\gamma_0}(t),d^{\gamma_0}(t))=(b',d')$ if and only if $\lim_{t\rightarrow 1^-}(b^{\gamma_1}(t),d^{\gamma_1}(t))=(b',d')$.
\item[(iii)] For all~$(b,d)\in D$ and~$(b',d')\in D'$, we have $\lim_{t\rightarrow 1^-}(d^{\gamma_0}(t)-b^{\gamma_0}(t))=0$ if and only if $\lim_{t\rightarrow 1^-}(d^{\gamma_1}(t)-b^{\gamma_1}(t))=0$.
\end{itemize}
\end{lemma}
\begin{proof}
$[(i)\Rightarrow (ii) \text{ and } (iii)]$: Let~$\gamma_h$ be a homotopy between~$\gamma_0$ and~$\gamma_1$ through entrance paths, and let~$h\in [0,1]$. Using Proposition~\ref{proposition_bottleneck_ball}, we can partition the intervals in~$D$ into sets~$D_{(a)}^{\gamma_h}$ and~$D_{(b)}^{\gamma_h}$ as follows:
\begin{itemize}
\item[(a)] For each interval $(b',d')\in D'$, there is a unique $(b,d)\in D$ for which $\lim_{t\rightarrow 1^-}(b^{\gamma_h}(t),d^{\gamma_h}(t))=(b',d')$;
\item[(b)] For all other intervals $(b,d)\in D$, we have $\lim_{t\rightarrow 1^-}(b^{\gamma_h}(t)-d^{\gamma_h}(t))=0$.
\end{itemize}
By continuity, the classification remains constant along the homotopy~$h\mapsto \gamma_h$, i.e.~$D_{(a)}^{\gamma_h}$ and~$D_{(b)}^{\gamma_h}$ are the same for all~$0\leqslant h\leqslant 1$.

$[(ii)\Rightarrow (iii)]$: By assumption~$D_{(a)}^{\gamma_0}=D_{(a)}^{\gamma_1}$, hence~$D_{(b)}^{\gamma_0}=D_{(b)}^{\gamma_1}$.

$[(iii)\Rightarrow (ii)]$: By assumption~$D_{(b)}^{\gamma_0}=D_{(b)}^{\gamma_1}$, hence~$D_{(a)}^{\gamma_0}=D_{(a)}^{\gamma_1}=:D_{\mathrm{(a)}}$. Then,~$\gamma_0$ and~$\gamma_1$ induce bijections from~$D_{\mathrm{(a)}}$ to the set of intervals in~$D'$. Since~$\gamma_0$ and~$\gamma_1$ are entrance paths, these bijections are monotonic with respect to the endpoint values. So they are in fact the same bijections. 

$[(ii) \text{ and } (iii) \Rightarrow (i)]$: We define a homotopy~$h\mapsto \gamma_h$ between the restrictions of~$\gamma_0$ and~$\gamma_1$ to~$[0,1)$ by interpolating the coordinates: 
\[\forall 0\leqslant t< 1, \, \forall 1\leqslant i \leqslant \dim D, \,  x_i^{\gamma_h}(t):=hx_i^{\gamma_1}(t)+(1-h)x_i^{\gamma_0}(t). \]
For each interval~$(b,d)\in D$, we then have 
\[\forall 0\leqslant t< 1, \,  (b^{\gamma_h}(t),d^{\gamma_h}(t))=(hb^{\gamma_1}(t)+(1-h)b^{\gamma_0}(t), hd^{\gamma_1}(t)+(1-h)d^{\gamma_0}(t)). \]
Since $D_{(b)}^{\gamma_0}=D_{(b)}^{\gamma_1}$ and~$D_{(a)}^{\gamma_0}=D_{(a)}^{\gamma_1}$, we have $\lim_{t\rightarrow 1^{-}} (b^{\gamma_h}(t),d^{\gamma_h}(t))= (b',d')$ (resp. $\lim_{t\rightarrow 1^{-}} d^{\gamma_h}(t)-b^{\gamma_h}(t)=0$) if and only if  $\lim_{t\rightarrow 1^{-}}(b^{\gamma_0}(t),d^{\gamma_0}(t))= (b',d')$ (resp. $\lim_{t\rightarrow 1^{-}} d^{\gamma_0}(t)-b^{\gamma_0}(t)=0$). Therefore,
\[\forall h \in [0,1], \lim_{t\rightarrow 1^-} \gamma_h(t)=D',\]
hence the homotopy~$h\mapsto \gamma_h$ extends to a homotopy between~$\gamma_0$ and~$\gamma_1$ on the whole unit interval. 
\end{proof}
\begin{proof}[Proof of Proposition~\ref{proposition_entrance_path_category_equal_catbarc}]
The functoriality of~$[\phi] \mapsto [\gamma_\phi]$ amounts to showing that if~$\phi \in \CategoryBarc(D,D')$ and~$\psi \in \CategoryBarc(D',D'')$ are two morphisms, then the entrance paths~$\gamma_{\psi\circ\phi}$ and~$\gamma_{\psi}. \gamma_\phi$ are homotopic. For~$t\in [0,1]$, let~$\phi_t$ denote the interpolated map~$(1-t)\mathrm{Id}+t\phi$, so that~$\gamma_\phi(t)=\phi_t.D$,~$\gamma_\psi(t)=\psi_t.D'$ and~$\gamma_{\psi\circ \phi}(t)=(\psi \circ \phi)_t.D$. Besides, the concatenated path $\gamma_{\psi}. \gamma_\phi(t)$ equals~$\phi_{2t}.D$ for~$t\leqslant \frac{1}{2}$ and~$\psi_{2t-1}.D'=\psi_{2t-1}.(\phi.D)$ for~$t\geqslant \frac{1}{2}$. A homotopy between~$\gamma_{\psi\circ\phi}$ and~$\gamma_{\psi}.\gamma_{\phi}$ can then be defined as:
\[\mathrm{H}(h,t):= [h\phi_{2t}+ (1-h)(\psi\circ \phi)_t].D \text{ for } t\leqslant \frac{1}{2},\] and
\[\mathrm{H}(h,t):= [h\psi_{2t-1}\circ \phi+ (1-h)(\psi\circ \phi)_t].D \text{ for } t\geqslant \frac{1}{2}.\]
Hence, we obtain a functor from~$h\CategoryBarc$ to~$\Ent(\CatBarc)$, which is the identity on objects. Given barcodes~$D,D'$, we show that this functor gives a bijection $h\CategoryBarc(D,D')\overset\cong \longrightarrow\Ent(\CatBarc)(D,D')$.

Let~$\gamma$ be an elementary entrance path from~$D$ to~$D'$. We construct a morphism~$\phi$ from~$D$ to~$D'$ using the classification of the maps~$(b^{\gamma}(t),d^{\gamma}(t))$:
\begin{itemize}
\item[(a)] For each interval $(b',d')\in D'$, there is a unique $(b,d)\in D$ for which $\lim_{t\rightarrow 1^-}(b^{\gamma}(t),d^{\gamma}(t))=(b',d')$. We then set~$\phi(b):=b'$ and~$\phi(d):=d'$;
\item[(b)] For all other intervals $(b,d)\in D$, we have $\lim_{t\rightarrow 1^-}(b^{\gamma}(t)-d^{\gamma}(t))=0$. We then set~$\phi(b)=\phi(d)$ to be an arbitrary value such that~$\phi$ remains non-decreasing.
\end{itemize}
We extend~$\phi$ to a non-decreasing map of the unit interval arbitrarily. Since~$\gamma(1)=D'$, we have~$\phi.D=D'$ by construction. Besides, the interpolated path $\gamma_{\phi}(t)=[t\phi +(1-t)\mathrm{Id}].D$ is an elementary entrance path beween~$D$ and~$D'$, which satisfies the (i) and (ii) of Lemma~\ref{lemma_homotopic_entrance_paths} with respect to~$\gamma$, hence is homotopic to~$\gamma$. More generally, an arbitrary entrance path~$\gamma$ from~$D$ to~$D'$ is homotopic to a finite concatenation of elementary paths. In turn,~$\gamma$ is homotopic to an interpolated morphism~$\gamma_{\phi}$ by applying the previous argument to each elementary path. Therefore, the map~$[\phi]\in h\CategoryBarc(D,D')\mapsto [\gamma_\phi]\in \Ent(\CatBarc)$ is surjective.

Besides, comparing the (ii) of Lemma~\ref{lemma_homotopic_zigzag_index} with the (ii) of Lemma~\ref{lemma_homotopic_entrance_paths}, we see that if two morphisms between~$D$ and~$D'$ induce interpolated paths that are homotopic, then they must be homotopic. Consequently, the map~$[\phi]\in h\CategoryBarc(D,D') \mapsto [\gamma_{\phi}]\in \Ent(\CatBarc)$ is injective.
\end{proof}
\begin{remark}
\label{remark_classification_stratified_coverings}
It is well-known that coverings over a topological space~$\TopSpace$ satisfying mild properties are classified by the fundamental groupoïd of~$\TopSpace$. When~$\TopSpace$ is stratified, it is natural to consider stratified coverings, i.e. maps restricting to coverings over each individual stratum. If~$\TopSpace$ is homotopically stratified with locally simply connected and locally connected strata, the stratified coverings which are either local homeomorphisms or {\em branched covers} are classified by the entrance path category of~$\TopSpace$~\cite{woolf2008fundamental}. That is, functors from~$\Ent(\TopSpace)$ to $\mathbf{Set}$ functorially give rise to such stratified coverings, in fact also to constructible cosheaves, and conversely.\footnote{See~\cite{curry2016classification} for similar classifications of functors over~$\Ent(\TopSpace)$ when~$\TopSpace$ is conically stratified, and~\cite{treumann2009exit} for some $2$-categorical equivalences.} Although in the case of barcodes the inverse image~$\persmap^{-1}$ is valued in~$h\mathbf{Top}$, we have a natural set valued functor:
$$
\xymatrix{
\Ent(\CatBarc) \ar[rr]^{\persmap^{-1}} &&
h\mathbf{Top} \ar[rr]^{\pi_{0}} && \mathbf{Set}. }
$$
%It might be interesting to further investigate the information about the complex~$\SComplex$ encoded by this functor and its associated branched cover over~$\CatBarc$.
%
\end{remark}
We next prove  an analogue of Theorem~\ref{theorem_barcode_category} for entrance paths.
%that is (each connected component of) the space of entrance paths between two fixed barcodes is contractible. 
%
\begin{proposition}
\label{proposition_entrance_path_space_contractible}
For any two barcodes~$D,D'$ in~$\CatBarc$, the space of morphisms~$\Path_\leq(\CatBarc)(D,D')$ has finitely many path connected components each of  which is contractible. In other words,
the category~$\Path_\leq(\CatBarc)$ is {\em homotopy discrete}, i.e. $\Path_\leq(\CatBarc)(D, D') \simeq h\Path_\leq(\CatBarc) (D, D')$.
\end{proposition}
\begin{proof}
Let~$D$ and~$D'$ be two barcodes, and let~$\Holink(D,D')$ be the space of elementary entrance paths from~$D$ to~$D'$. For homotopically stratified metric spaces, the space of entrance paths and that of elementary entrance paths are homotopy equivalent~\cite[Theorem~4.9]{miller2009popaths}, from which we deduce that:
\[\Path_\leq (\CatBarc) (D, D') \simeq \Holink(D,D').\]
%Since~$\CatBarc$ is a metric space and is homotopically stratified by Proposition~\ref{proposition_barcodes_homotopically_stratified}, we can apply Proposition~4.9 from~\cite{miller2009popaths} to obtain a homotopy equivalence:
%
%\[\Path_< (D, D') \simeq \Holink(D,D').\]
%
The proof of the statement then follows from Lemma~\ref{lemma_homotopic_entrance_paths}. In more detail, let~$\Omega$ be a path connected component in~$\Holink(D,D')$, and let~$\gamma_0\in \Omega$. Recall that we can partition the intervals in~$D$ into sets~$D_{(a)}^{\gamma_0}$ and~$D_{(b)}^{\gamma_0}$ as follows:
\begin{itemize}
\item[(a)] For each interval $(b',d')\in D'$, there is a unique $(b,d)\in D$ for which $\lim_{t\rightarrow 1^-}(b^{\gamma_0}(t),d^{\gamma_0}(t))=(b',d')$;
\item[(b)] For all other intervals $(b,d)\in D$, we have $\lim_{t\rightarrow 1^-}(b^{\gamma_0}(t)-d^{\gamma_0}(t))=0$.
\end{itemize}
From Lemma~\ref{lemma_homotopic_entrance_paths}, any other path~$\gamma\in \Omega$ satisfies~$D_{(a)}^{\gamma}=D_{(a)}^{\gamma_0}$ and~$D_{(b)}^{\gamma}=D_{(b)}^{\gamma_0}$. Define~$r:\Omega \times [0,1] \times [0,1) \rightarrow \Barc_K$ by:
\[r: (\gamma,h,t)\in \Omega \times [0,1] \times [0,1) \longmapsto \big\{((1-h)b^{\gamma}(t)+ h b^{\gamma_0}(t)  ,(1-h)d^{\gamma}(t)+ h d^{\gamma_0}(t))\big\}_{(b,d)\in D} \in \StratumD_D\subseteq \Barc_K.\]
We can continuously extend~$r$ at time~$t=1$ by~$r(\gamma,h,1):= D'$. We then get a deformation retraction
\[R:(\gamma,h)\in \Omega \times [0,1] \longmapsto r(\gamma,h,.)\in \Omega \] 
of~$\Omega$ onto~$\{\gamma_0\}$. 
\end{proof}
Combining Proposition \ref{proposition_entrance_path_category_equal_catbarc} and \ref{proposition_entrance_path_space_contractible} we can summarise our results in this section with the following.
\begin{corollary} For any two barcodes $D, D' \in \CatBarc$ we have
\[\Path_\leq (\CatBarc)(D, D') \simeq h\Path_\leq(\CatBarc)(D , D')= \Ent(\CatBarc)(D,D')\cong h\CategoryBarc(D,D') \simeq \CategoryBarc(D,D'),\]
and hence the natural weak equivalences of categories
\[\Path_\leq (\CatBarc) \overset \simeq \longrightarrow h\Path_\leq (\CatBarc) = \Ent(\CatBarc) \overset \cong \longleftarrow h\CategoryBarc \overset \simeq \longleftarrow \CategoryBarc.\]
\end{corollary}

%\newpage 
%
\section{Variations of the fiber problem}
We adapt our analysis to two further  situations of interest, namely when we remove the constraint that filters and barcodes take value in the unit interval, and when we restrict~$\persmap$ to the subspace of filters determined by their values on vertices. Finally, we point out that the action of the symmetries of~$\SComplex$ on the filters restricts to the fibers.
\subsection{The case of unbounded filters and barcodes}
\label{section_unbounded_situation}
In the previous sections, the  values of filters and the interval  endpoints of barcodes  were constrained to lie in the unit interval. Here we briefly outline how our analysis can be adapted when we consider  the unbounded case and replace  the interval~$\I $ by the real line~$\R$. We denote by~$\Filt (\R)_\SComplex $ the filter functions with unrestricted real values and by~$\Barc (\R) $ the space of finite barcodes with unrestricted endpoints. As before, persistent homology defines a map
\[
\persmap: \Filt (\R)_\SComplex \longrightarrow
\Barc (\R).
\]
Let~$\IncHomSpaceR$ be the group of continuous automorphisms of the ordered real line that are the identity outside a compact set. Similarly let~$\NonDecMapsR$ be the monoid of continuous order preserving maps of the real line that are the identity outside a compact set. Both spaces then act on the extended spaces of filters and barcodes. The proof of Lemma~\ref{lemma_commutes_PH_homeo} generalises to show that the persistence map above is equivariant with respect to these extended actions. The actions are continuous as in Proposition~\ref{proposition_continuous_barcode_action} when we equip~$\IncHomSpaceR$ and~$\NonDecMapsR$ with the~$L^\infty$ topology. Using the action of~$\IncHomSpaceR$ one can construct stratifications of filter and barcode spaces such that the $\IncHomSpaceR$-orbits are the strata and the analogues of Propositions~\ref{prop_orbits_stratification_filter} and~\ref{prop_orbits_stratification_barcodes} hold.
%The group~$\IncHomSpaceR$ of continuous automorphisms of the ordered real line and the monoid~$\NonDecMapsR$ of continuous order preserving maps of the real line 

Indeed, most of the results and their proofs can easily be adapted with the caveat that fibers no longer have to be compact.
Thus,  item~{\bf (a)} of Theorem~\ref{theorem_fiber_bundle_polyhedral} needs to  be reinterpreted: The polyhedron~%$\PolytopeClosed$ 
in the fiber are not necessarily products of closed standard simplices but instead we have
\[\PolytopeClosed\cong \StandSimplex_1 \times \cdots \times \StandSimplex_{\dim \StratumD -1}\times \StandSimplex_{\dim \StratumD}', \]
where the last term  may be  a  simplex missing its last face, i.e.~$\StandSimplex_{\dim \StratumD}'$ may be the closed $i$-simplex~$\StandSimplex^i$ or of the form:
\[\StandSimplex^i_{\infty}:=\{0\leqslant x_1\leqslant \cdots \leqslant  x_i <\infty  \}.\]
 Modulo this subtlety, Theorems~\ref{theorem_fiber_bundle_polyhedral} and Proposition~\ref{proposition_monodromy} hold also in the unbounded case we consider here, and~$\persmap$ is again a stratified fiber bundle whose fibers are (possibly unbounded) polyhedra. The proofs of these results in the unbounded situation do not present additional difficulties.

Next we provide a necessary and sufficient criterion for the simplicial complex~$\SComplex$ that ensures that all the fibers of~$\persmap$ are bounded. 
 %This is a useful property since then, the fiber can be canonically given the structure of a finite simplicial complex, given that polytopal complexes admit finite triangulations. 
\begin{definition}
A subset~$L\subseteq \SComplex$ is {\em $\mathbb k$-removable}, or simply {\em removable}, if~$\SComplex\setminus L$ is a subcomplex of~$\SComplex$ and the inclusion~$\SComplex\setminus L\hookrightarrow \SComplex$ induces an isomorphism on (standard) homology with $\mathbb k$-coefficients.~$\SComplex$ is said to be {\em $\mathbb k$-essential}, or simply {\em essential}, if it has no removable subsets.
\end{definition}

For instance, any pair~$(\simplex,\simplex')$ where~$\simplex'$ has~$\simplex$ as its only co-face provides an example of a removable subset for any~$\mathbb k$. This is an elementary collapse familiar from simple homotopy theory. More elaborate examples include the wedge product~$\SComplex \vee A$ of two simplicial complexes~$\SComplex $ and~$A$, where~$A$ is $\mathbb k$-acyclic, i.e.~$A$ has trivial reduced (ordinary) homology with~$\mathbb k$ coefficients. Then~$L= A \setminus \{*\}$ is removable. A rich source of such~$A$ are the classifying spaces of perfect groups, or the classifying spaces of finite groups when~$\mathbb k$ is of characteristic zero.

\begin{proposition}
\label{proposition_compact_fiber}
The fibers of the persistence map~$\persmap$ are all compact if and only if the complex~$\SComplex$ is essential.
\end{proposition}
\begin{proof}

If~$\SComplex$ is not essential, it has a removable subset~$L\subseteq \SComplex$. Given a partial filter $f:\SComplex\setminus L \rightarrow \R$ and a real value~$x$ with $x\geq \max _{\sigma \in \SComplex \setminus L} f(\sigma)$,~$f$ can be extended to a filter~$ f_x$ on all of~$\SComplex$ by assigning the common value~$x$ to all the simplices in~$L$.  Thus the fiber of~$D=\persmap ( f_x)$ contains the open half line
$\{ f_x \, | \, x \in [ \max_{\sigma \in \SComplex \setminus L}  f(\sigma), \infty ) \}$ and is hence not compact.

Conversely, if~$\persmap$ has a non-compact fiber over some barcode~$D$, it means there exists an~$f\in
\persmap^{-1}(D)$ attaining  values higher than
the largest (bounded) endpoint~$\max (D)$ of~$D$. Therefore the set~$L$ of simplices on which~$f$ takes value larger than~$\max ( D)$ is removable.
\end{proof}
Next we will exhibit a family of simplicial complexes that are essential.
%The assumption that~$\SComplex$ is essential is quite reasonnable, as we now detail briefly.
We say that~$\SComplex$ is a {\em triangulated  (oriented) manifold}  if  its geometric realisation~$|\SComplex|$ is homeomorphic to a closed (orientable) manifold. Note that this manifold will necessarily be compact since~$\SComplex$ is finite.
\begin{proposition}
\label{proposition_manifold_essential}
Let~$\SComplex$ be a triangulated manifold. If either the field of coefficients~$\field$ is of characteristic 2 or~$\SComplex$ is oriented, then~$\SComplex$ is essential. 
\end{proposition}
\begin{proof}
Without loss of generality, we may assume that~$\SComplex$ is connected and of dimension~$d$.
Then any filter~$f$ attains its maximum value on a top dimensional simplex, since all lower-dimensional simplices have co-faces. By our assumptions, we have~$\mathrm{H}_d(\SComplex)\cong \field$ and a generator of this top dimensional homology class is the sum of all top dimensional simplices of~$\SComplex$ (with appropriate signs). At the level of barcodes, this means that there is an infinite interval in homological degree~$d$ starting at~$\max_{\simplex \in \SComplex}f(\simplex)$. Therefore the fiber of the persistence map over any barcode~$D$ is bounded, and since it is closed by continuity of~$\persmap$, it is also compact and hence essential by the previous result.
\end{proof}
The converse of Proposition~\ref{proposition_manifold_essential} is false as  can be seen from the following simple counterexample.
\begin{example}
\label{example_converse_manifold_essential_wrong}
Let~$\SComplex$ be the wedge product of two triangles. So~$\SComplex$ has five vertices and six 1-simplices. 
Its first homology group is of rank 2 but any subcomplex will have at most rank 1. Thus~$\SComplex$ is essential, but~$\SComplex$ is not a manifold.
\end{example}

The point-set topology is a little delicate when working with the unbounded real line~$\R$ instead of the compact interval~$\I$. This is part of the reason why we chose to work with~$\I$ for the main part of our paper.
For example,  Proposition~\ref{prop_quotient_topology} cannot be adapted to the unbounded situation: If the fibers of~$\persmap$ are not compact then the bottleneck topology and the quotient topology induced by~$\persmap$ do not necessarily agree on the image~$\CatBarc(\R)$ as  Example~\ref{example_quotient_not_bottleneck_unbounded_situation} below shows. 

\begin{example}
\label{example_quotient_not_bottleneck_unbounded_situation}
Let us consider again the Example~\ref{example_line_complex} of the complex~$\SComplex$ representing the unit interval with vertices~$a,b$ and $1$-simplex~$\simplex$. As a set~$\CatBarc(\R)$ can be identified (as in Example~\ref{example_line_complex}) with the 3-simplex~$\StandSimplex'_3=\{(x_1,x_2,x_3), -\infty<x_1\leq x_2 \leq x_3< \infty\}$ with two missing faces, where in addition the 2-dimensional face corresponding to~$-\infty< x_1 < x_2=x_3 < \infty$ is collapsed to the line segment~$-\infty < x_1=x_2=x_3 < \infty$. Consider the set 
\[U:=\big\{ (x_1,x_2,x_3)\,  | \, x_3-x_2<e^{-x_2} \big\} \subseteq \CatBarc(\R) \]
of barcodes whose unique bounded bar~$(x_2,x_3)$ has length less than~$e^{-x_2}$. We then have
\[\persmap^{-1}(U):=\big\{ f\in \Filt_\SComplex(\R)\,  | \, f(\simplex)-\max(f(a),f(b))<e^{-\max(f(a),f(b))} \big\} \subseteq \Filt_\SComplex(\R) \subseteq \R^3, \]
which is open in~$\Filt_\SComplex(\R)$ for the usual topology induced by the $\|.\|_\infty$-metric. Therefore~$U$ is an open set in the quotient topology. However it does not contain any bottleneck ball, hence is not an open set in the bottleneck topology.
\end{example}

Furthermore, the choice of topology on~$\NonDecMapsR$ and~$\IncHomSpaceR$ matters in the unbounded situation: Replacing the~$L^{\infty}$ topology by the compact open topology results in the actions not being (sequentially) continuous as can be seen in the following example.
\begin{example} 
\label{Example_action_not_continuous}
Consider the sequence of barcodes~$D_n$ containing a single interval~$(n, n+2^{-n})$. The sequence~$D_n$ converges to the empty diagram~$D_{\emptyset}$ in the bottleneck topology. In addition, let~$\phi_n:\R\rightarrow \R$ be the map such that~$\phi_n(n)=n$,~$\phi_n(n+2^{-n})=n+1$,~$\phi_n|_{[n,n+2^{-n}]}$ and $\phi_n|_{[n+2^{-n}, n+2]}$ are linear, and outside $[n, n+2]$ ~$\phi_n$ is the identity. Then the sequence~$\phi_n$ converges to the identity map of the real line in the compact open topology. If the action were continuous in both variables, the sequence~$\phi_n.D_n$ would converge to the empty diagram~$\mathrm{Id}.D_{\emptyset}=D_{\emptyset}$. However, each of the barcodes~$\phi_n.D_n$ contains a unique interval~$(n,n+1)$, and the sequence does therefore not converge in the bottleneck topology. 
\end{example}

However, in our analysis, we have never needed to make full use of the continuity of the action of~$\NonDecMaps$. Instead, it is enough to ensure that the action is continuous w.r.t. the choice of $\phi$. Namely, fixing~$D\in \Barc$, the map $\phi\in \NonDecMaps \mapsto \phi.D\in \Barc$ is continuous. In the current unbounded situation, it can also be proven that the map $\phi\in \NonDecMapsR \mapsto \phi.D\in \Barc (\R)$ is continuous, where we consider the compact open topology on~$\NonDecMapsR$. This is precisely what is needed to carry the analysis through in a similar fashion. 

Finally, if we do not impose that maps in~$\NonDecMapsR$ and~$\IncHomSpaceR$ equal the identity outside a compact set, then the analysis breaks down in the~$L^\infty$ topology. For instance, straight line interpolations on which our results rely, would not always give continuous paths. 
%An alternative is to replace the~$L^\infty$ topology by the compact open topology on~$\NonDecMapsR$ and~$\IncHomSpaceR$. However the analogue Proposition~\ref{proposition_continuous_barcode_action} does not hold: the action of~$\NonDecMapsR$ and~$\IncHomSpaceR$ on~$\Barc(\R)$ are not sequentially continuous. 
%
\subsection{The case of lower star filters}
\label{section_lower_stars}
The lower star filtration form an interesting subspace of the space of all filters on~$\SComplex$ and  one might want to restrict one's attention to these as for example in~\cite{cyranka2018contractibility}. We summarise briefly how our analysis can be adapted and compared to this case.

Let~$\SComplex$ be a finite simplicial complex with vertex set~$\Vset$. A {\em lower star filter} on~$\SComplex$ is a filter~$f\in \Filt_\SComplex$ such that for any simplex~$\simplex\in \SComplex$:
\[f(\simplex)= \max_{v\in \simplex}f(v).\]
Being determined by their values on vertices, such filters offer many advantages in practice. Any function~$f: V \to \I$  can be extended uniquely to a lower star filter. Hence, the subspace~$\Low_\SComplex\subseteq\Filt_\SComplex$ of lower star filters is canonically isomorphic to~$\I^\Vset$. We denote its image under the persistence map~$\persmap$ by~$\Barc_\SComplex^{\Low}$.

The actions of~$\IncHomSpace$ and~$\NonDecMaps$ by post-composition restrict to~$\I^\Vset$ and, by the equivariance of~$\persmap$, also to~$\Barc_\SComplex^{\Low}$. 
As the strata are given by $\IncHomSpace$-orbits, we see that both~$\Low_\SComplex$  and~$\Barc_\SComplex^{\Low}$
are sub-stratified spaces,  %of~$\Filt_\SComplex$ 
each consisting of a subcollection of full strata from~$\Filt_\SComplex$ and~$\Barc_\SComplex$ respectively. %Similarly, by equivariance of $\persmap$ the actions on $\Barc_\SComplex$ restrict to $\Barc_\SComplex ^\Low$ making the latter into a (strongly) sub-stratified space 
Thus~$\persmap$ restricts to a strongly stratified map
%, since inside a given stratum~$\StratumF\subseteq \Filt_\SComplex$, the pre-order on simplices of~$\SComplex$ induced by all filters is the same. 
%
\[\persmap_{|\Low}:\Low_\SComplex \longrightarrow \Barc_\SComplex^\Low \]
and hence
satisfies similar properties as~$\persmap$. In particular, the fiber~$\persmap_{\Low}^{-1}(D)$ has again the structure of a polyhedral complex and the analogue of  Theorem~\ref{theorem_fiber_bundle_polyhedral} holds. 
%In the case where~$\SComplex$ is a complex obtained from a finite subdivision of the unit interval, an alternative decomposition of the fiber into polyhedra has been found in~\cite{cyranka2018contractibility}.  

The space~$\Barc_\SComplex^\Low$ also gives rise to a subcategory~$\CategoryBarc^{\Low}$ of~$\CategoryBarc$. We note that this is a full subcategory. Thus Theorem~\ref{theorem_barcode_category} and Proposition~\ref{proposition_monodromy} also hold for this subcategory. In particular,~$\CategoryBarc^\Low$ is homotopy discrete.
As before, we can associate to morphisms in~$\CategoryBarc^\Low$ monodromies between fibers, turning the inverse image into a functor
\[\persmap_{|\Low}^{-1}:\CategoryBarc^{\Low} \longrightarrow \mathbf{Top}. \]
Up to homotopy, the monodromies between fibers are again polyhedral maps. 
\begin{remark}
\label{remark_barycentric_subdivision_lower}
Vice-versa, we may also consider the space of filters on~$\SComplex$ as a subspace of the space of  lower star filters on its barycentric subdivision~$\hat \SComplex$. Recall that the vertices of~$\hat \SComplex$ are the simplices of~$\SComplex$. Thus~$f \in \Filt_\SComplex$ uniquely gives rise to~$\hat f \in \Low_{\hat \SComplex}$ via
\[\hat f( \hat \simplex ) := f(\simplex),\]
where~$\hat \simplex$ is the vertex of~$\hat \SComplex$ corresponding to the simplex~$\simplex$ in~$\SComplex$. 
This way, we get a nested sequence of  spaces:
\[\Low_\SComplex \subset \Filt_\SComplex \subset \Low_{\hat \SComplex}.\]
%
%Note that neither inclusion is an equality.
It is a straightforward exercise to show that $\persmap (f) = \persmap (\hat f)$. Thus we also have a nested sequence of barcode spaces:
\[
\Barc_\SComplex^\Low \subset \Barc_\SComplex \subset \Barc_{\hat \SComplex }^\Low.
\]
All these inclusions are also $\NonDecMaps$-equivariant and~$\persmap$ defines an equivariant map between these nested sequences. Thus, by similar arguments as before,~$\persmap$ and the inclusions are compatible with the stratifications in the strongest sense giving rise to a sequence of full, homotopy discrete subcategories
\[
\CategoryBarc^\Low \subset \CategoryBarc \subset \mathbf{Bar_{\hat \SComplex}}^\Low.
\]
It would be interesting to analyse how the fibers of~$\persmap$, or more generally the fiber functors on these three categories are related.
\end{remark}
%However, all lower star filters of~$\hat \SComplex$ that arise this way  have the vertices of~$\SComplex $ as local minima (a local minimum is obtained by~$f$ at a vertex~$v$ if $f(\simplex) \geq f(v)$ for all simplices $\simplex $ with $v \in \simplex$);
%In particular the map $\Filt_\SComplex \to 
%\Low (\hat \SComplex)$ taking~$f$ to~$\hat f$
%is an embedding of stratified spaces and the image is a union of strata.
%}
%{\bf Question:} We have~$\Low_\SComplex \subset \Filt_\SComplex \subset \Low (\hat \SComplex)$. It is natural to ask the question
%how the fibers of~$\persmap $ on these three spaces  are related.
%
\subsection{Symmetries restricted to  fibers}
\label{sec:symmetry}
In this brief section we examine how symmetries of the simplicial complex restrict to the fibers of the persistence map. For simplicity we return to filters and barcodes in the unit interval~$\I$, but the analysis can be carried out in the unbounded situation or when restricting to lower star filters in the same way.  

Let~$\SymK$ be the group of isomorphisms of the simplicial complex $\SComplex$. Then $\SymK$ can be identified as the subgroup of the group of symmetries  $\text{Sym} (\SComplex_0)$ of the vertices  $ \SComplex_0$ which consists of all those $s$ that map a subset $\simplex \in \SComplex$ to a subset $s(\simplex ) \in \SComplex$.
Pre-composition with the inverse induces a left action of $\SymK$ on the space $\Filt_\SComplex$ of filters via 
\[s.f:= f\circ {s}^{-1}.\]
Thus $f$ and $s.f$ take  the same values and furthermore, if $\StratumF\subseteq \Filt_\SComplex$ is a filter stratum then $s.\StratumF$ is another stratum of the same dimension, $\dim s.\StratumF= \dim \StratumF$, and $s$ maps $ \StratumF$ to $s(\StratumF)$ via an affine isomorphism. 
\begin{proposition}
\label{proposition_persistence_equivariant}
For all $f \in \Filt_\SComplex$ and all $s \in \SymK$ we have
\[
\persmap (f) = \persmap (s.f).
\]
Equivalently, the action of~$\SymK$ on~$\Filt_\SComplex $ restricts to the fiber $\persmap^{-1} (D)$ for every~$D \in \Barc_{\SComplex} $. Furthermore,~$\SymK$ acts through maps of polyhedra on $\persmap^{-1} (D)$.
\end{proposition}

\begin{proof}
The symmetry $s$ maps the sublevel-set filtration of $f$ isomorphically to that of $s.f$. Thus the associated persistence modules are isomorphic and so the two resulting barcodes in~$\Barc_\SComplex$ are the same. Hence, $f$ and $s.f$ are in the same fiber. 
Since $s$ defines an affine isomorphism from a stratum~$\StratumF$ to the stratum~$s.\StratumF$ and preserves fibers, it restricts to an affine isomorphism from $\persmap^{-1} (D) \cap \StratumF$ to $\persmap^{-1} (D)\cap s(\StratumF)$ for any barcode~$D\in \Barc_K$.
\end{proof}
\begin{proposition}
\label{proposition_symmetries_fiber}
Given two barcodes $D,D'$ and a morphism $\phi \in \IncMaps(D,D')$, the monodromy $\Monodromy_\phi: \persmap^{-1} (D)\rightarrow  \persmap^{-1} (D')$ is $\SymK$-equivariant. 
\end{proposition}
\begin{proof}
 Let~$f\in \persmap^{-1} (D)$,~$s\in \SymK$ and~$\simplex\in \SComplex $. Then the statement of the proposition follows from
 \[
 \Monodromy_\phi(s.f)(\simplex)= 
 (\phi \circ s.f) (\simplex) =
 \phi(f(s^{-1}(\simplex)))= 
 \Monodromy_\phi (f) (s^{-1} (\simplex)) =
 s.\Monodromy_\phi(f)(\simplex). \qedhere
 \]
\end{proof}

\newpage
\appendix
\section{A detailed example: the fiber of the persistence map over the triangle}
\label{section_triangle}
In this section we illustrate our theory developed so far on an example. The simplest non-contractible simplicial complex~$\SComplex$ is a triangle. We denote its vertices by~$a,b,c$ and its edges by~$ab,ac,bc$.
\begin{center}
\begin{tikzpicture}[thick, scale=0.6]
\node[] (a) at (0,-1) {$a$};
\node[] (b) at (4,-1) {$b$};
\node[] (c) at (2,3) {$c$};
\draw[] (a)--(b) node[midway,above] {$ab$};
\draw[] (a)--(c)  node[midway,above left] {$ac$};
\draw[] (b)--(c) node[midway,above right] {$bc$};

%\node[label=above:{$C_{n}(\R^d)$}] (a) at (0,0) {};
%\node[] (a) at (0,-1) {$a$};
%\node[] (c) at (-0.4,-0.4) {$\mathcal{F}$};

\end{tikzpicture}
\end{center}

In this simple case, a filter is any map $f:\SComplex\rightarrow \I$ such that its value on an edge is greater (or equal) than the value on the endpoints of this edge. By the elder rule and since $\SComplex$ is connected, $\min(f(a),f(b),f(c))$ is the left endpoint of the unique infinite interval %$(\min(f(a),f(b),f(c)),\infty)$
in~$\persmap_0(f)$. For similar reasons,~$\persmap_1(f)$ contains a unique unbounded interval with left endpoint given by $\max(f(ab),f(bc),f(ac))$. %Moreover, as $\SComplex$ contains $6$ simplices, two of which contribute to these only two unbounded intervals, the other four values determine two (possibly of length zero) intervals in $\persmap_0$. 

In addition, the example of the triangle  has interesting symmetries. The symmetries of the triangle  $\SymK=\mathbf{D}_3$ is the dihedral group which can also be identified with the symmetric group~$\Sigma_3$, the set of bijections of the set~$\{a,b,c\}$ of vertices of~$\SComplex$.

Our goal in this section is 
to determine all the barcode strata and the corresponding fibers, see Figure~\ref{fig:all_barcode_strata}; 
describe the action of~$\SymK \simeq \mathbf{D}_3$ on the fibers, and compute the monodromies between the non-discrete fibers.

{\bf Summary of the results:}
\begin{enumerate}
\item In section~\ref{section_example_computation_barcode_strata} we compute all~$34$ barcode strata in the image~$\CatBarc=\persmap(\Filt_\SComplex)$;
\item In section~\ref{section_example_computation_fibers} we compute the fibers of~$\persmap$ over the distinct barcode strata. 
We find only five strata with  fibers that are not discrete.
By Theorem~\ref{theorem_fiber_bundle_polyhedral}, these fibers are polyhedral complexes, and  we exhibit the polyhedra in~$\R^\SComplex$ making up the complexes. 
The results in section~\ref{sec:symmetry} guarantee that~$\SymK$ acts on each fiber, and we describe this action on the fibers for these five  barcode strata;
\item Finally, by the results in section~\ref{sec:theorem_2}, the closure containment relations between barcode strata yield monodromies between fibers, which up to homotopy are polyhedral maps. We describe the monodromies between non-discrete fibers in section~\ref{section_example_computation_monodromies}.
\end{enumerate}
This simple example of the triangle shows that the fibers of the persistence map can be topologically distinct from each other. Furthermore, the topology of the fibers can also be more complex than that of the underlying simplicial complex~$\SComplex$, especially for low dimensional strata in the space of barcodes.
Similar observations hold when restricting the fibers to the subspace~$\Low_{\SComplex}$ of lower star filters, as we detail in section~\ref{subsection_lowerstar_triangle}. We contrast this with the case studied in~\cite{cyranka2018contractibility} where~$\SComplex$ is a triangulation of the interval~$[0,1]$ and  lower star filters are considered. In that case the fibers of the persistence map are all disjoint unions of contractible sets. 
In particular, our example of the triangle shows that the fiber of the persistence map does not have to be
a union of contractible sets.
%We now know that this topological triviality does not extend beyond this setting. 
%
\newpage
\begin{figure}[H]
\centering
\includegraphics[width=0.9\textwidth]{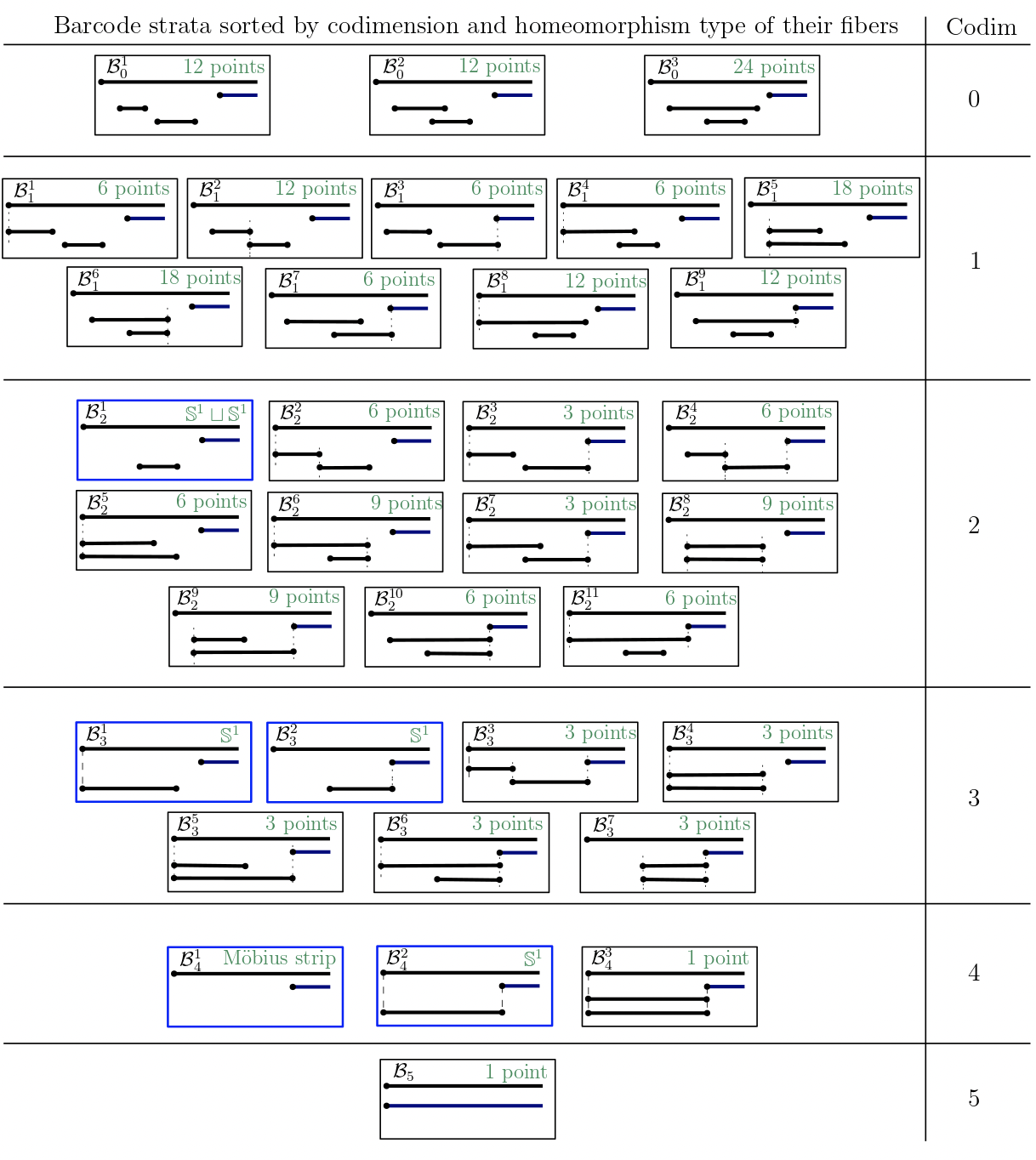}
\caption{Each stratum is represented by one of its barcodes, where the blue interval is the one corresponding to the degree~$1$ homology of the triangle and the others correspond to degree~$0$ homology. Strata come with a label,
for instance~$\Dzeromid$, whose subscript gives the codimension and the superscript  allows to enumerate strata of a given codimension. The homeomorphism type of the fiber of~$\persmap$ over each stratum is given in green. For instance, the fiber~$\persmap^{-1}(\Donemid)$ is finite and consists of~$12$ distinct filters. The blue boxes highlight the five strata with non-discrete fibers.}
\label{fig:all_barcode_strata}
\end{figure}
\subsection{Computation of~$\CatBarc$ and its strata}
\label{section_example_computation_barcode_strata}
For notational convenience, we replace the unit interval with a bigger interval, $\I:=[-10,10]$. This allows considering barcodes with only integer valued endpoints.
In addition, we restrict ourselves to strata of barcodes with endpoints strictly in~$(-10,10)$, since they completely determine strata (and their fibers) where the endpoints~$-10$ and~$10$ are allowed.

{\bf The top dimensional filter strata} of the space of filters~$\Filt_\SComplex$ correspond to injective filters and can equivalently be thought of as orderings of the simplices in the triangle, where an edge must appear after its vertices. %Each stratum contains a unique filter $f$ with image 
%$\mathrm{Im} f=\{0,1,2,3,4,5\}$. 
This allows us to count  these top dimensional strata. Namely, considering the case where the vertices of the triangle appear before all the edges, and separately the case where one edge appears before the last vertex,  we get the following count of  possible orderings and hence 
\[
 3! \times 3! + 3 \times 2 \times 1 \times 1 \times 2! = 36 + 12 = 48 \text{ top dimensional filter strata.}
\]

{\bf The top dimensional barcode strata} in the image~$\CatBarc=\persmap(\Filt_\SComplex)$, by Proposition~\ref{proposition_image_PH_determined_by_topdim}, are given by the image of the top dimensional filter strata. We argue that there are precisely three:
From section~\ref{sec:symmetry}, the persistence map is~$\SymK$-equivariant, so we may restrict ourselves to those filter strata, viewed as orderings, for which the vertex~$a$ is first, $f(a)=0$, followed by~$b$, $f(b)=1$. The following simplex must either be~$ab$ or~$c$. In the first case, $f(ab)=2$, all filters satisfying this yield the same barcode $[\{(0,\infty),(1,2),(3,4)\},\{(5,\infty)\}]$. We denote by~$\Dzeroleft$ the corresponding codimension~$0$ barcode stratum. In the second case, $f(c)=2$, we get two other barcodes depending on whether~$ab$ is the next simplex in the ordering or not, $[\{(0,\infty),(1,3),(2,4)\},\{(5,\infty)\}]$ and $[\{(0,\infty),(1,4),(2,3)\},\{(5,\infty)\}]$. We denote the corresponding barcode strata by~$\Dzeromid$ and~$\Dzeroright$. 

{\bf The list of all barcode strata} in~$\CatBarc$, by  Proposition~\ref{proposition_image_PH_determined_by_topdim},
 can be derived from the three top dimensional barcode strata~$\Dzeroleft$,~$\Dzeromid$ and~$\Dzeroright$ by collapsing their interval endpoints. We draw all the~$34$ resulting strata in Figure~\ref{fig:all_barcode_strata}, sorted by codimension, and give them labels that are used in the rest of the section.  
\subsection{Computation of fibers and the action of~$\SymK$}
The fibers over the various barcode strata in the image of the persistence map, whose computations are detailed in this section, are summarised in Figure~\ref{fig:all_barcode_strata} (in green in the top right corner of each box).
\label{section_example_computation_fibers}
\paragraph{Strata with discrete fibers.}
Most of the barcode strata have finite fibers. For instance, the unique lowest-dimensional stratum is the one labelled by~$\Dfive$: there only the two essential homological features are present, and appear at the same time. This implies that all the simplices of the triangle must appear at a given time, and therefore the fiber of $\persmap$ over~$\Dfive$ consists of a unique constant function. This agrees with the prediction of Proposition~\ref{proposition_barcode_minimal_endpoints} that the fiber is contractible. 
The barcode strata with maximal number of bounded intervals, that is~$2$ such intervals, have zero bounded deficit (Def.~\ref{definition_barcode}). There are~$28$ such strata. By Proposition~\ref{proposition_improved_bounding_dimension_fiber}, their fibers are discrete. For instance, a simple counting argument gives discrete fibers for the top dimensional barcode strata~$\Dzeroleft$,~$\Dzeromid$ and~$\Dzeroright$, with~$12$,~$12$, and~$24$ points in their fibers respectively. More generally, Proposition~\ref{proposition_improved_bounding_dimension_fiber} upper-bounds the dimension of the fibers over arbitrary strata by their bounded deficit. In this example, note that the bounded deficit in fact equals the dimension of the fiber in all cases, except in the degenerate case of the stratum~$\Dfive$. In the remainder of this section we compute explicitly the fibers over the five barcode strata that have fibers of dimension greater than~$0$. 

\paragraph{Stratum $\Dtwo$.} We take a representative barcode~$D:=[\{(b_1,\maxvalue), (b_2,d_2)\}, \{(b_3,\maxvalue)\}]\in \Barc^{2}$, where $b_1<b_2<d_2<b_3$, of the codimension $2$ stratum~$\Dtwo$.
\begin{figure}[H]
\centering
\includegraphics[width=0.15\textwidth]{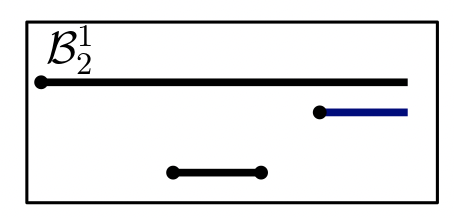}
\end{figure}
Without loss of generality, $(b_1,b_2,d_2,b_3)=(0,1,2,3)$. We use the identification of a filter~$f$ with the vector $(f(a),f(b),f(c),f(ab),f(ac),f(bc))\in \I^6$. If~$f$ yields barcode~$D$, then:
\begin{itemize}
    \item There is a vertex $v$ (resp. an edge $e$) at which $f$ attains its minimum (resp. maximum) value, equal to $0$ (resp. $3$).
    \item There is another vertex $v'$ and edge $e'$ such that $f(v'),f(e')=1,2$.
    \item The remaining vertex and {\bf incident} edge have same value $0\leqslant t \leqslant 3 $.
\end{itemize}

Let us assume that the vertex $a$ should create the first connected component, while the vertex $c$ should create the second one. Then the vertex $b$ must appear at the same time as an edge connecting $b$ to either $a$ or $c$. If $b$ appears at the same time as $ab$, and if we wish that $bc$ is the edge closing the loop, we get a~$1$-simplex $\{( 0, t, 1, t,2, 3)\}_{0 \leqslant t \leqslant 3}$ in the fiber, which we break into three~$1$-simplices that lie in (the closures of) different filter strata: $\{( 0, t, 1, t,2, 3)\}_{0 \leqslant t \leqslant 1}$,
$\{(0, t, 1, t,2, 3)\}_{1 \leqslant t \leqslant 2}$ and 
$\{(0, t, 1, t, 2,3)\}_{2\leqslant t \leqslant 3}$. If rather the edge~$ac$ closes the loop of the triangle, we get the two sets $\{(0, t, 1, t, 3, 2)\}_{1 \leqslant t \leqslant 2}$ and
$\{( 0, t, 1, t,3, 2)\}_{0\leqslant t \leqslant 1}$ in the fiber. Now, let us assume that $b$ appears at the same time as $bc$, which imposes that $b$ appears after time $1$. Then if $ab$ closes the loop of the triangle, we get the two sets 
$\{(0, t, 1, 3, 2,t)\}_{1 \leqslant t \leqslant 2}$ and
$\{(0, t, 1, 3, 2,t)\}_{2\leqslant t \leqslant 3}$ in the fiber, while if we wish that the edge $ac$ closes this loop, we get the unique set $\{( 0, t, 1, 2,3, t)\}_{1\leqslant t \leqslant 2}$ in the fiber. All in all, we have gathered~$8$ embeddings of the standard $1$-simplex in~$\I^6$. By symmetry, if we vary the choice of two vertices that create the first two connected components, we get~$6$ analogous collections of~$8$ embeddings of the standard $1$-simplex described above, which together cover the fiber. This provides a description of the fiber in terms of a graph, whose incidence structure is explicited in Fig.~\ref{fig:fiber_case_D}. In particular, the fiber of $\persmap$ over $D$ is homeomorphic to $\mathbb{S}^1\sqcup \mathbb{S}^1$. 

We further depitct the action of~$\SymK$ on~$\persmap^{-1}(D)$ in Fig.~\ref{fig:fiber_case_D}. We consider the cyclic map~$(a,b,c)\mapsto (b,c,a)$ and the elementary transposition~$(a,b,c)\mapsto (b,a,c)$ as generators. We see that the cyclic map preserves the two connected components of the fiber~$\mathbb{S}^1\sqcup \mathbb{S}^1$, whereas the elementary transposition swaps them. More generally, even permutations~$g\in \SymK$ preserve the connected components of the fiber, while odd permutations exchange them.
\begin{figure}[H]
\centering
%width=0.8\linewidth, height=12.5cm
\includegraphics[width=0.62\textwidth]{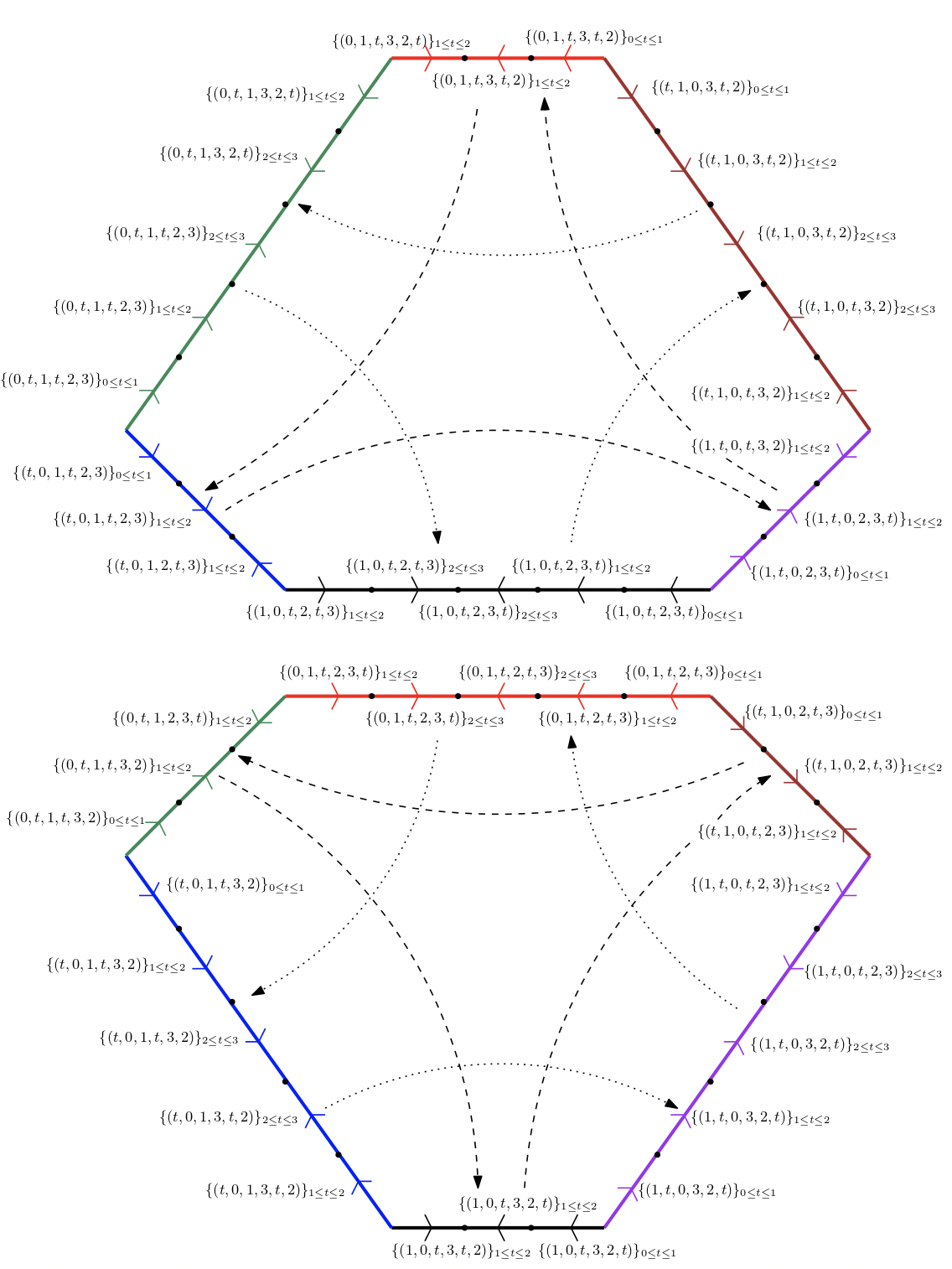}
\caption{The fiber of the barcode in stratum~$\Dtwo$. We observe two connected components. Each edge represents an embedding of the standard $1$-simplex in the fiber, oriented with an arrow toward increasing values of the parameter~$t$. The six colors correspond to the six possible choices of two vertices responsible for the appearance of the first two connected components. If~$a$ and~$c$ (resp.~$a$ and~$b$,~$b$ and~$a$,~$b$ and~$c$,~$c$ and~$a$,~$c$ and~$b$) create these components, we color the edge in green (resp. red, black, blue, purple and brown). The cyclic symmetry $(a,b,c)\mapsto (b,c,a)$ acts on the fiber by rotating each connected component by an angle of~$\frac{2\pi}{3}$, as depicted by the dotted arrows. The transposition $(a,b,c)\mapsto (b,a,c)$ acts by swapping the connected components by reflecting along the horizontal axis.}
\label{fig:fiber_case_D}
\end{figure}
\paragraph{Stratum~$\Dthreemid$.}
Fixing two endpoints via~$d_2=b_3$ in the stratum~$\Dtwo$, we get the stratum~$\Dthreemid$ represented below.
\begin{figure}[H]
\centering
\includegraphics[width=0.15\textwidth]{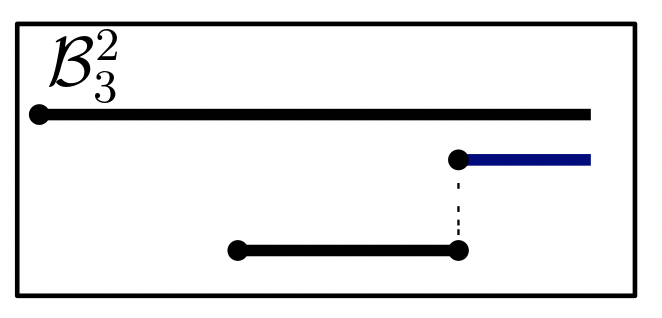}
\end{figure}
As a representative of this stratum, we consider the barcode $D=[\{(0,\infty),(1,2)\},\{(2,\maxvalue)\}]\in \Barc^{2}$. The fiber of~$D$ can be obtained similarly as in the previous case, that is by providing a cover of the fiber by embeddings of the $1$-simplex in~$\I^6$. Imposing that vertices~$a$ and then~$c$ should be responsible for the appearance of the first two components, we get the sets $\{(0, t, 1, t,2, 2)\}_{0\leqslant t \leqslant 1}$,  $\{( 0, t, 1, t,2, 2)\}_{1\leqslant t \leqslant 2}$ and
$\{( 0, t, 1, 2,2, t)\}_{1\leqslant t \leqslant 2}$ in the fiber. By symmetry in the choice of these two vertices, we get a total of~$6\times 3$ embeddings of the standard $1$-simplex that together cover the fiber. This decomposition describes the fiber as a graph, which is described in Fig~\ref{fig:fiber_case_D'} and is homeomorphic to~$\mathbb{S}^1$. The action of the symmetries on the circle is described in the figure as well. 
\begin{figure}[h]
\centering
\includegraphics[width=0.8\linewidth, height=10cm]{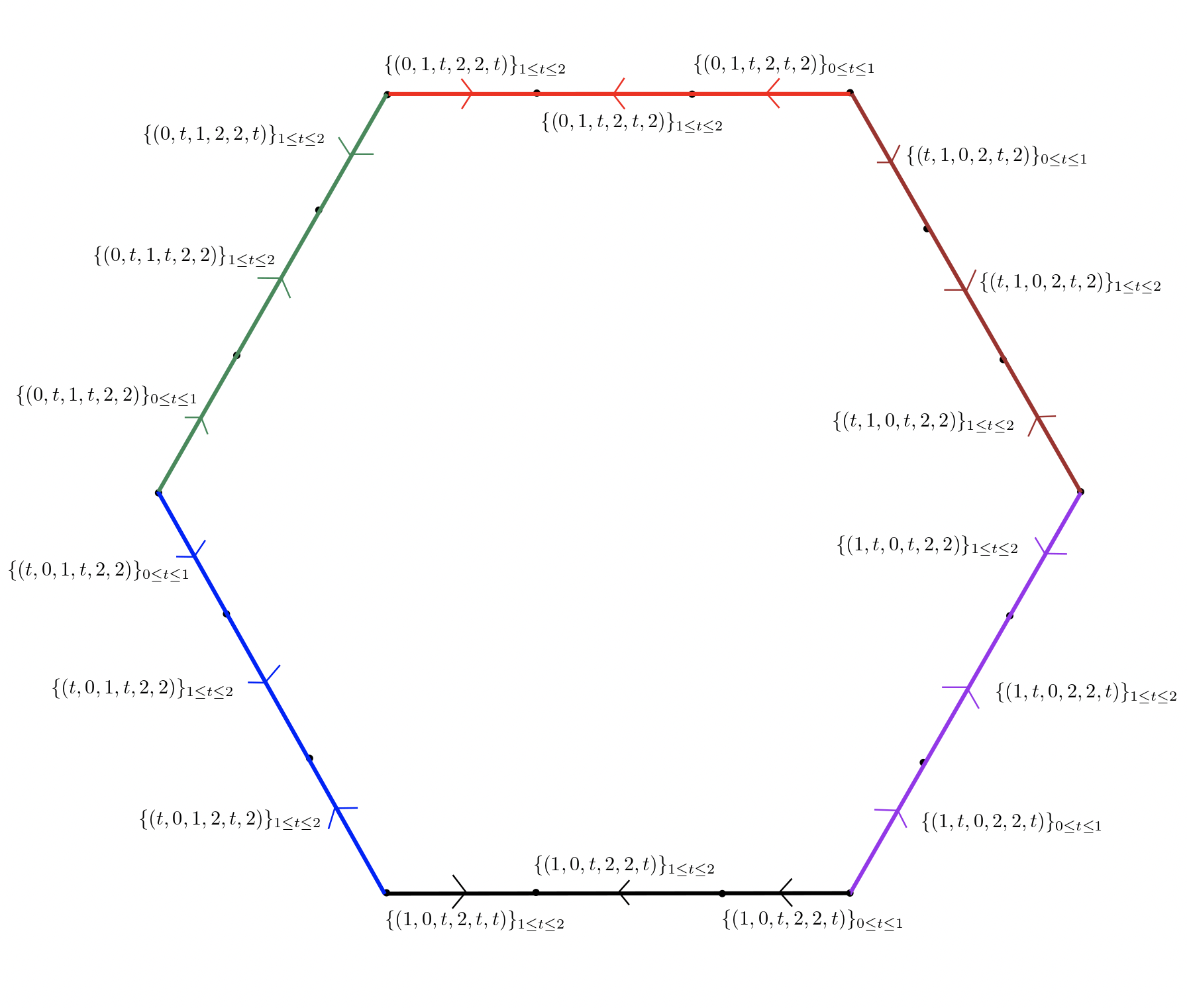}

\caption{The fiber of the barcode in stratum~$\Dthreemid$. Each edge represents an embedding of the standard $1$-simplex in the fiber, oriented toward increasing values of the parameter~$t$. The six colors correspond to the six possible choices of two vertices responsible for the appearance of the first two connected components in the barcode. The coloring convention is the same as in case~$\Dtwo$, Fig~\ref{fig:fiber_case_D}. The cyclic permutation of the triangle acts as a rotation by an angle of~$\frac{2\pi}{3}$ oriented counter-clockwise, while the elementary transposition acts as the horizontal symmetry of the hexagon.} 
%This action is further inherited from the one described in stratum~$\Dtwo$ by the Figures~\ref{fig:fiber_case_D_internal_symmetry} and~\ref{fig:fiber_case_D_transversal_symmetry}, under the process of mapping the fiber of stratum~$\Dtwo$ to the fiber of stratum~$\Dthreemid$, see Fig~\ref{fig:caseD_to_D'}. 
\label{fig:fiber_case_D'}
\end{figure}
\paragraph{Stratum~$\Dthreeleft$.}
A representative of the stratum~$\Dthreeleft$ is $D:=[\{(0,\infty),(0,1)\},\{(2,\infty)\}]\in \Barc^{2}$.
\begin{figure}[H]
\centering
\includegraphics[width=0.15\textwidth]{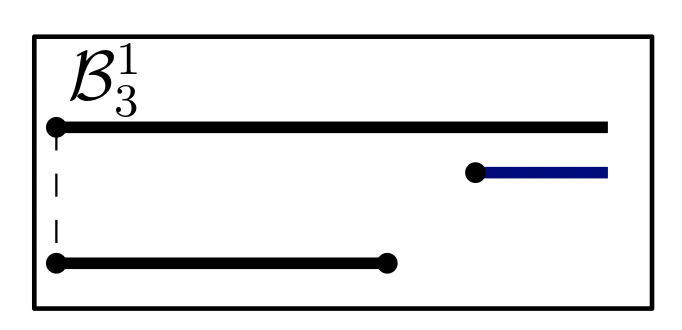}
\end{figure}
We describe the fiber with an explicit cover by embeddings of the $1$-simplex in~$\I^6$. Imposing that vertices~$a$ and~$c$ should be responsible for the appearance of the first two components, we get the following embeddings:
\begin{itemize}
\item  $\{(0, t, 0, 1,2, t)\}_{0\leqslant t \leqslant 1}$,  $\{( 0, t, 0, 2,1, t)\}_{0\leqslant t \leqslant 1}$,
$\{( 0, t, 0, t,1, 2)\}_{0\leqslant t \leqslant 1}$, $\{( 0, t, 0, t,2, 1)\}_{0\leqslant t \leqslant 1}$,$\{( 0, t, 0, t,1, 2)\}_{1\leqslant t \leqslant 2}$ and $\{( 0, t, 0, 2,1, t)\}_{1\leqslant t \leqslant 2}$.
\end{itemize} 
By symmetry in the choice of these two vertices, we get a total of~$3\times 6$ embeddings of the standard $1$-simplex that together cover the fiber. By inspecting adjacency relations, this decomposition describes the fiber as a graph isomorphic to the graph of the fiber over~$\Dthreemid$. In particular, the fiber is homeomorphic to~$\mathbb{S}^1$.

\paragraph{Stratum~$\Dfourleft$.} 
Let $D=[\{(b_1,\maxvalue)\},\{ (b_2,\maxvalue)\}]\in \Barc^2$ (where~$b_1<b_2$) be a barcode in the stratum~$\Dfourleft$. 
\begin{figure}[H]
\centering
\includegraphics[width=0.15\textwidth]{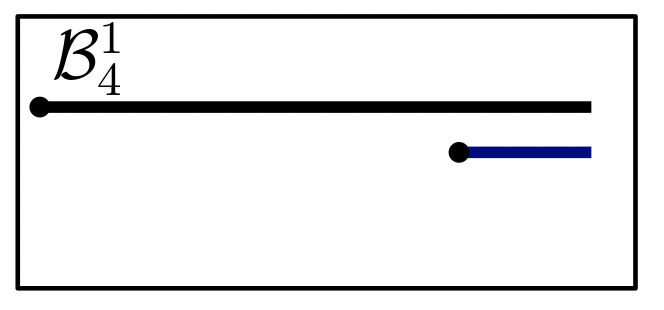}
\end{figure}
We may set~$b_1=0$ and~$b_2=1$ for simplicity. If a filter~$f:\SComplex\rightarrow \I$ yields the barcode~$D$, then there are two pairs~$(v_i,e_i)_{i=1,2}$ of vertices and edges such that $0\leqslant f(e_i)=f(v_i)=: t_i \leqslant 1$ and the remaining vertex~$v_0$ and edge~$e_0$ realize the minimum~$0$ and maximum~$1$ of~$f$ respectively. 

For a fixed choice of~$(v_i, e_i)_{i=0,1,2}$ as above, the two parameters~$t_1 \leq t_2$ describe a 2-simplex. This $2$-simplex corresponds to one of the top dimensional polyhedron in the fiber of~$\persmap$ over~$D$, which we know is a polyhedral complex from Theorem~\ref{theorem_fiber_bundle_polyhedral}. To count and describe these 2-simplices, fix~$v_0$, say~$v_0 =a$. 
Then we distinguish between the two cases: (i) where~$e_0$ contains~$v_0$ and (ii) where~$e_0$ does not contain~$v_0$. In case (i), if we choose~$e_0 = ac$, then~$b$ and~$ab$ simultaneously appear at time~$t_1$ and finally~$c$ and~$bc$ appear at time~$t_2$. The resulting simplex in the fiber is denoted by~$\FirstSimplexMobius:=\{(0, t_1, t_2, t_1, 1,t_2)\}_{0\leqslant t_1\leqslant t_2 \leqslant 1}$. In case (ii), there is only one choice for~$e_0$, i.e.~$e_0=bc$. The other vertex-edge pairs have to be the pairs~$(b, ab)$ and~$(c,ac)$. If we decide that the pair~$(b, ab)$ appears before~$(c,ac)$, we obtain the simplex~$\SecondSimplexMobius:=\{(0, t_1, t_2, t_1,t_2, 1)\}_{0\leqslant t_1\leqslant t_2 \leqslant 1}$. All the other simplices in the fiber may be derived from the action of~$\SymK$ on~$\FirstSimplexMobius$ and~$\SecondSimplexMobius$. More precisely, letting~$\tau:(a,b,c)\mapsto (b,a,c)$ be the elementary transposition and~$c:(a,b,c)\mapsto (b,c,a)$ the cyclic permutation, we obtain all the $2$-simplices in the fiber:
\begin{itemize}
\item[(i)] $\FirstSimplexMobius=\{(0, t_1, t_2, t_1, 1,t_2)\}$, 
$\tau.\FirstSimplexMobius=\{(t_1, 0, t_2, t_1,t_2, 1)\}_{0\leqslant t_1\leqslant t_2 \leqslant 1}$, 
$c.\FirstSimplexMobius=\{(t_2, 0, t_1, 1,t_2, t_1)\}_{0\leqslant t_1\leqslant t_2 \leqslant 1}$, 
$c^2.\FirstSimplexMobius=\{(t_1, t_2, 0, t_2,t_1, 1)\}_{0\leqslant t_1\leqslant t_2 \leqslant 1}$, 
$\tau c.\FirstSimplexMobius=\{(0, t_2, t_1, 1,t_1, t_2)\}_{0\leqslant t_1\leqslant t_2 \leqslant 1}$, 
$\tau c^2.\FirstSimplexMobius=\{(t_2, t_1, 0, t_2,1, t_1)\}_{0\leqslant t_1\leqslant t_2\leqslant 1}$;
\item[(ii)] $\SecondSimplexMobius=\{(0, t_1, t_2, t_1,t_2, 1)\}_{0\leqslant t_1\leqslant t_2 \leqslant 1}$, 
$\tau. \SecondSimplexMobius=\{(t_1, 0, t_2, t_1,1, t_2)\}_{0\leqslant t_1\leqslant t_2 \leqslant 1}$, 
$c. \SecondSimplexMobius=\{(t_2, 0, t_1, t_2,1, t_1)\}_{0\leqslant t_1\leqslant t_2 \leqslant 1}$, 
$c^2. \SecondSimplexMobius=\{(t_1, t_2, 0, 1,t_1, t_2)\}_{0\leqslant t_1\leqslant t_2 \leqslant 1}$, 
$\tau c. \SecondSimplexMobius=\{(0, t_2, t_1, t_2,t_1, 1)\}_{0\leqslant t_1\leqslant t_2 \leqslant 1}$, 
$\tau c^2. \SecondSimplexMobius=\{(t_2, t_1, 0, 1,t_2, t_1)\}_{0\leqslant t_1\leqslant t_2 \leqslant 1}$. 
\end{itemize} 
All the~$12$ sets are embeddings of the $2$-simplex~$\StandSimplex^2$ in~$\I^6$, where we represent the simplex~$\StandSimplex^2$ in~$\R^2$ conveniently for our purpose as in Fig~\ref{fig:two_simplex}.
\begin{figure}[H]
\centering
\includegraphics[width=0.4\textwidth]{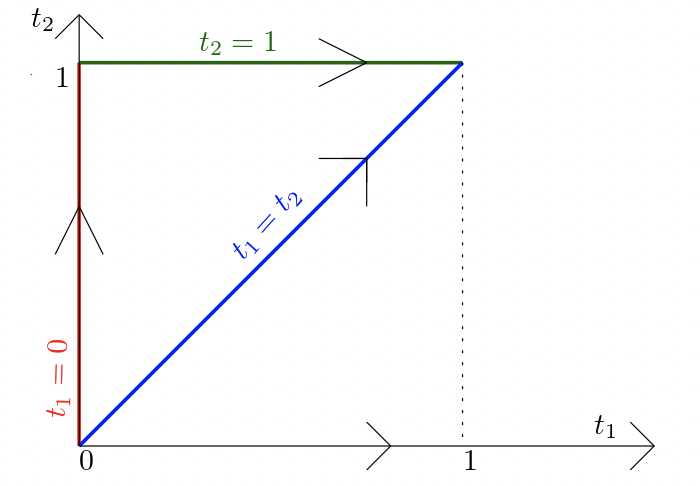}
\caption{An embedding of the standard simplex~$\StandSimplex^2$, with faces in red, green and blue.}
\label{fig:two_simplex}
\end{figure}

The $2$-simplices in the orbit of~$\SecondSimplexMobius$ meet with~$3$ distinct other $2$-simplices at its~$3$ faces (obtained by setting~$t_1=0$,~$t_2=1$ or~$t_1=t_2$). The $2$-simplices in the orbit of~$\FirstSimplexMobius$ only meet with two other simplices. Glued together, these simplices form a Möbius band embedded in~$\R^6$, as explicited in Figure~\ref{fig:fiber_triangle_mobius}. Therefore the fiber of the persistence map over the barcodes in the stratum~$\Dfourleft$ is isomorphic, as a simplicial complex, to the Möbius band. 

We again consider how the symmetry group~$\Sigma_3$ of the triangle acts on the fiber. Note that the action on the fiber must preserve the orientations and colors described in Fig.~\ref{fig:fiber_triangle_mobius}. The elementary transposition simply rotates the Möbius band by an angle of~$\pi$. The action of the cyclic permutation is slightly more involved as it reverses and translates the Möbius band.
\begin{figure}[H]
\centering
\includegraphics[width=0.8\textwidth]{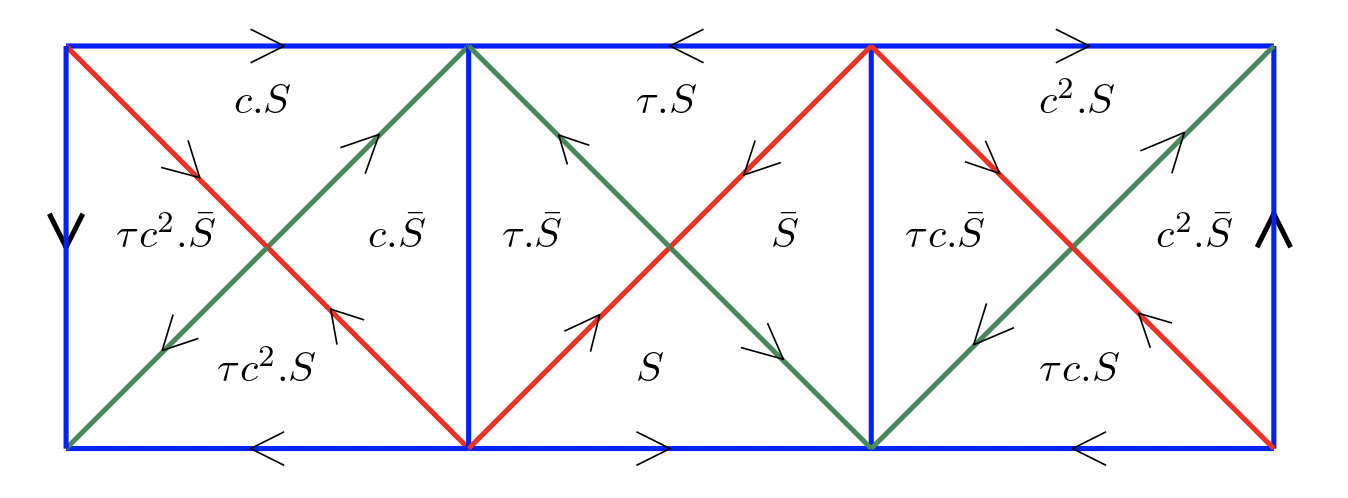}
\caption{The $2$-simplices in the fiber of the stratum~$\Dfourleft$ glued together in a Möbius strip. Blue (resp. red and green) edges correspond to setting~$t_1=t_2$ (resp.~$t_1=0$ and~$t_2=1$) in their co-faces. The left and right extreme blue oriented edges are identified. 
The action of the transposition~$(a,b,c)\mapsto (b,a,c)$ on the fiber can be described as a rotation of the Möbius strip by an angle of~$\pi$.
The action of the cyclic permutation~$(a,b,c)\mapsto(b,c,a)$ on the fiber can be described as the composition of (i) the symmetry of the Möbius strip around its middle horizontal line, followed by (ii) a unit translation on the left of each simplex.}
\label{fig:fiber_triangle_mobius}
\end{figure}
\paragraph{Stratum~$\Dfourright$.}
The last stratum of barcodes whose fiber we explicitely compute has representative $D:=[\{(0,\maxvalue)\},\{(0,1),(1,\maxvalue)\}]$:
\begin{figure}[H]
\centering
\includegraphics[width=0.15\textwidth]{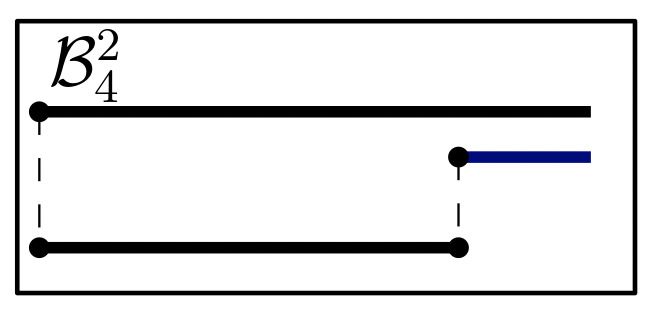}
\end{figure}
In this case, the fiber is the union of the segments:
\begin{itemize}
    \item $\{( 0, t, 0, t,1, 1)\}_{0\leqslant t \leqslant 1}$, $\{( 0, t, 0, 1,1, t)\}_{0\leqslant t \leqslant 1}$, 
    $\{( 0, 0, t, 1,1, t)\}_{0\leqslant t \leqslant 1}$, $\{( 0, 0, t, 1,t, 1)\}_{0\leqslant t \leqslant 1}$, 
    $\{( t, 0, 0, 1,t, 1)\}_{0\leqslant t \leqslant 1}$, $\{( t, 0, 0, t,1, 1)\}_{0\leqslant t \leqslant 1}$; 
\end{itemize}
which assemble into a regular hexagon, and therefore the fiber of~$\persmap$ over the stratum~$\Dfourright$ is homeomorphic to~$\mathbb{S}^1$. 

\subsection{Computation of monodromies between fibers}
\label{section_example_computation_monodromies}
We describe the monodromies between non-discrete fibers. We focus on pairs of barcode strata that differ by one dimension, since a monodromy between an arbitrary pair of strata is a composition of such elementary monodromies by Corollary~\ref{corollary_monodromoy_homotopic_polyhedral_and_composition}.
\paragraph{Monodromy from~$\Dtwo$ to~$\Dthreeleft$ and~$\Dthreemid$.}

From the previous section, the fibers over~$\Dthreeleft$ and~$\Dthreemid$ are isomorphic cyclic graphs. We only describe monodromies from~$\persmap^{-1}(\Dtwo)$ to~$\persmap^{-1}(\Dthreemid)$ since monodromies from~$\persmap^{-1}(\Dtwo)$ to~$\persmap^{-1}(\Dthreeleft)$ are identical. The stratum~$\Dtwo$ contains the stratum~$\Dthreemid$ in its closure. From Lemma~\ref{lemma_closure_barcode_strata}, this ensures that the set of morphisms~$\IncMaps(\Btwo,\Bthreemid)$ is non-empty for any representatives~$(\Btwo,\Bthreemid)\in \Dtwo \times \Dthreemid$. For clarity, we take as representatives the barcodes from the previous section, that is $\Btwo=[\{(0,\maxvalue),(1,2)\},\{(3,\maxvalue)\}]$ and $\Bthreemid=[\{(0,\maxvalue),(1,2)\},\{(2,\maxvalue)\}]$.

From section~\ref{sec:theorem_2}, the number of different homotopy classes of monodromies from~$\persmap^{-1}(\Btwo)$ to~$\persmap^{-1}(\Bthreemid)$ is upper-bounded by the cardinality of~$\pi_0(\IncMaps(\Btwo,\Bthreemid))$. Moreover, the homotopy type of an element~$\phi \in \IncMaps(\Btwo,\Bthreemid)$ is completely characterized by its index by Remark~\ref{remark_index_invariant}. In the current situation, any map~$\phi$ from~$\Btwo$ to~$\Bthreemid$ must collapse the third and fourth endpoint of~$\Btwo$. This means that up to homotopy, there is a unique monodromy map from~$\persmap^{-1}(\Btwo)$ to~$\persmap^{-1}(\Bthreemid)$ to describe.
It is then natural to choose~$\phi$ to be any linear extension of the map:
\[\phi: \{0,1,2,3\}\mapsto \{0,1,2,2\}.\]
The resulting monodromy~$\Monodromy_\phi$, from Proposition~\ref{proposition_monodromy}, is a simplicial map between fibers. To visualize~$\Monodromy_\phi$, it is convenient to see how the fiber~$\persmap^{-1}(\Btwo)$ is progressively deformed into a subset of~$\persmap^{-1}(\Bthreemid)$ under the path~$t\mapsto t \Monodromy_\phi + (1-t)\text{Id}$. As depicted in Fig.~\ref{fig:caseD_to_D'}, this transformation has the effect to collapse some edges and to identify some edges and nodes. 

In the same figure, we see that the monodromy is a surjection onto the fiber~$\persmap^{-1}(\Bthreemid)$. By Proposition~\ref{proposition_symmetries_fiber}, the equivariance of the monodromy w.r.t. the~$\SymK=\Sigma_3$ action on the fibers predicts that the image, through the monodromy, of the action on~$\persmap^{-1}(\Btwo)$ must equal the action on~$\persmap^{-1}(\Bthreemid)$. Let us for instance consider the cyclic permutation $(a,b,c)\mapsto (b,c,a)$ of the triangle, which acts on the two hexagons constituting~$\persmap^{-1}(\Btwo)$ by rotation of an angle of~$\frac{2\pi}{3}$. The monodromy identifies the two hexagons in the fiber, hence the induced action of~$g$ on~$\persmap^{-1}(\Bthreemid)$ is the rotation by the same angle, which agrees with the direct computation of the action of~$g$ on~$\persmap^{-1}(\Bthreemid)$ in Fig.~\ref{fig:fiber_case_D'}. The same observation can be made about the elementary transposition $\tau:(a,b,c)\mapsto (b,a,c)$ of the triangle.

\begin{figure}[H]
\includegraphics[width=0.9\linewidth, height=12cm]{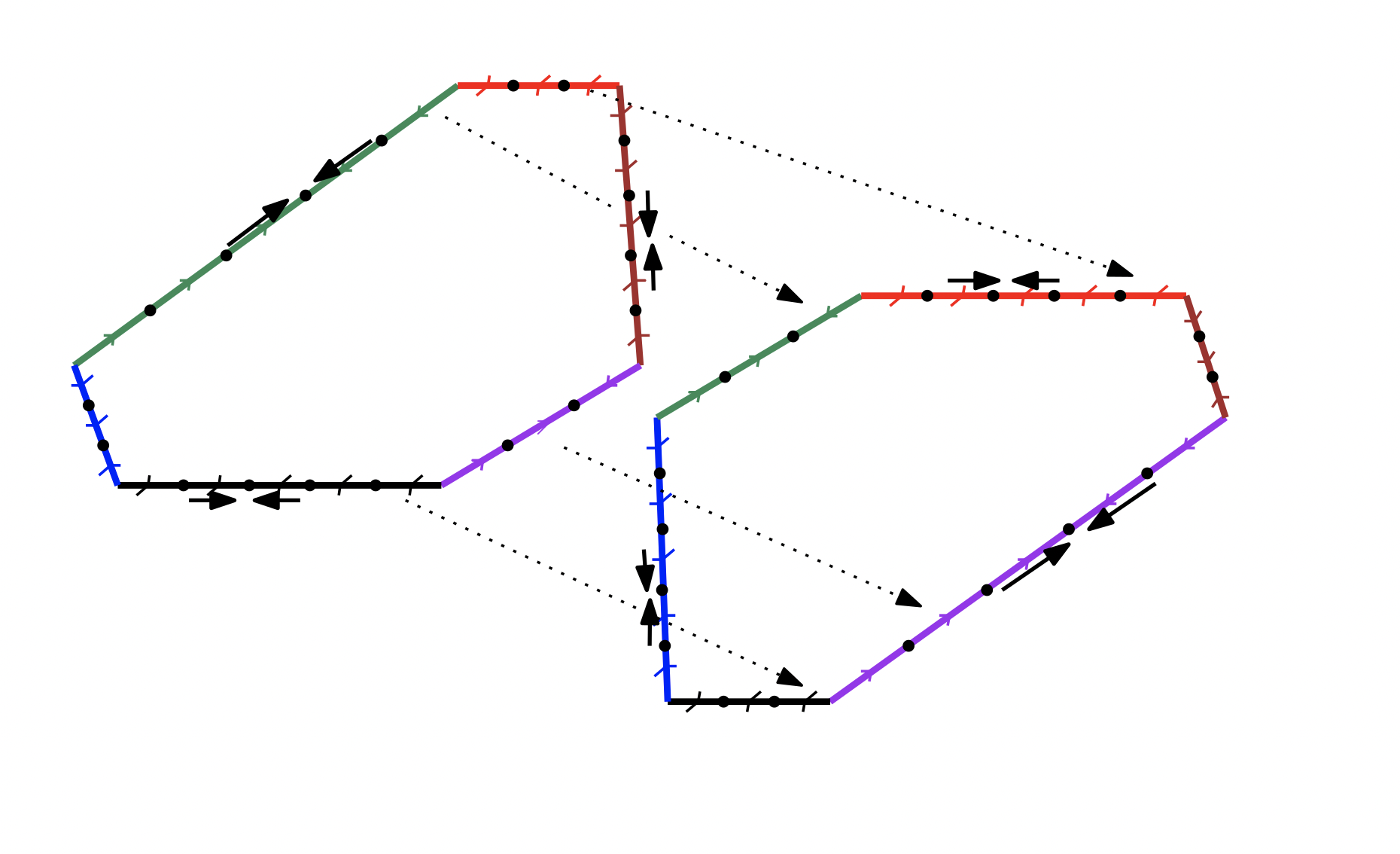}
\caption{Starting with the two irregular hexagons describing the fiber of the barcode $[\{(0,\maxvalue), (1,2)\},\{ (3,\maxvalue)\}]$ of stratum~$\Dtwo$ (see Fig.~\ref{fig:fiber_case_D}), the arrows describe the process of continuously tracking the fiber as the interval~$(3,\maxvalue)$ gets closer to the interval~$(2,\maxvalue)$, thus ending to the barcode $[\{(0,\maxvalue), (1,2)\}, \{(2,\maxvalue)\}]$ of stratum~$\Dthreemid$. Plain arrows show edges of the fiber that are collapsed during this process. Meanwhile, the two hexagons merge into the regular hexagon depicted in Fig.~\ref{fig:fiber_case_D'}. This merging happens by identifying edges of the two components, of the same color and orientation, following the dotted arrows. }
\label{fig:caseD_to_D'}
\end{figure}
\paragraph{Monodromy from~$\Dthreeleft$ and~$\Dthreemid$ to~$\Dfourleft$.}

The stratum~$\Dthreemid$ contains the stratum~$\Dfourleft$ in its closure. We let $\Bthreemid:=[\{(0,\maxvalue),(1,2)\},\{(2,\maxvalue)\}]$ be the representative barcode of the stratum~$\Dthreemid$, and $\Bfourleft:=[\{(0,\maxvalue)\},\{(1,\maxvalue)\}]$ be that of~$\Dfourleft$, both as in the previous section. The set of morphisms~$\IncMaps(\Bthreemid,\Bfourleft)$ is non-empty, and in fact contains a unique homotopy class by Remark~\ref{remark_index_invariant}, since all maps~$\phi\in \IncMaps(\Bthreemid,\Bfourleft)$ must collapse the interval~$(1,2)$. It is then natural to choose~$\phi$ to be any linear extension of the map:
\[\phi: \{0,1,2\}\mapsto \{0,1,1\}.\]
From Proposition~\ref{proposition_monodromy}, the resulting monodromy~$\Monodromy_\phi: \persmap^{-1}(\Bthreemid)\rightarrow \persmap^{-1}(\Bfourleft)$ is a map of polyhedral complexes. Recall that the fiber~$\persmap^{-1}(\Bthreemid)$ is a polyhedral complex described in Fig.~\ref{fig:fiber_case_D'}. The monodromy map~$\Monodromy_\phi$ collapses~$12$ out of the~$18$ edges in~$\persmap^{-1}(\Bthreemid)$. The remaining~$6$ edges form a regular hexagon which is mapped onto the green circle of the Möbius strip describing the fiber of~$\Dfourleft$, see Fig~\ref{fig:fiber_triangle_mobius}. Likewise, the monodromy from the fiber of (a representative of) the stratum~$\Dthreeleft$ to the fiber over~$\Dfourleft$ is unique up to homotopy, and collapses~$12$ out of the~$18$ edges in the fiber over~$\Dthreeleft$, sending the remaining~$6$ edges onto the red circle of the Möbius strip describing the fiber of~$\Dfourleft$.

\paragraph{Monodromy from~$\Dthreeleft$ and~$\Dthreemid$ to~$\Dfourright$.}
Recall that the fiber over~$\Dthreemid$ is a cyclic graph with~$18$ edges depicted in Figure~\ref{fig:fiber_case_D'}, while the fiber over~$\Dfourright$ is a regular hexagon. There is again a unique (up to homotopy) monodromy~$\Monodromy_\phi$ between these fibers. The simplicial map~$\Monodromy_\phi$ collapses~$12$ out of the~$18$ edges in~$\persmap^{-1}(\Bthreemid)$. The remaining~$6$ edges (one for each color in Figure~\ref{fig:fiber_case_D'})  form the regular hexagon~$\persmap^{-1}(\Dfourright)$. In particular,~$\Monodromy_\phi$ is a fibration. The monodromy from~$\persmap^{-1}(\Bthreeleft)$ to~$\persmap^{-1}(\Dfourright)$ can be described in the same way.  
\subsection{Lower star filters} 
\label{subsection_lowerstar_triangle}
If we consider the restriction of~$\persmap$ to the subspace~$\Low_\SComplex \subseteq \Filt_\SComplex$ of lower star filters, there are only~$2$ barcode strata in the image:~$\StratumD^1_4$ and~$\StratumD_5$. The reason for this is that edges enter the sublevel set filtration of a lower star filter at the same time as one of their vertices, hence there are no bounded intervals in the resulting barcode. The fibers of the restricted persistence map over the two strata can be derived from the general case. Namely, the fiber over~$\StratumD_5$ is the constant filter, while the fiber over~$\StratumD^1_4$ is a hexagon which embeds into the M\"obius strip as the green zig-zag in Fig.~\ref{fig:fiber_triangle_mobius}. 

\bibliographystyle{plain}
\bibliography{reference}
%
%%%%%%%%%%%%%%%%%%%%%%%%%%%%%%%%%%%%%%%%%%%%%%%%

%
\end{document}